%% file: dspaper3.tex
\documentclass{amsart}
\usepackage{amsthm}
\usepackage{amssymb}
\usepackage{amsfonts}
\usepackage{mathrsfs}
\usepackage[all,cmtip]{xy}
\usepackage{tikz}     
\usepackage{multirow}
\usepackage{url}

\begin{document}

\date{}
\title[The moduli stack of parabolic bundles over $\mathbb{P}^1$]{The moduli stack of parabolic bundles over $\mathbb{P}^1$, quiver representations, and the Deligne-Simpson problem}
\author{Alexander Soibelman}\thanks{This work was partially supported by NSF grant DMS 1101558.}
\address{Max Planck Institute for Mathematics, Vivatsgasse 7,
53111 Bonn, Germany}
\email{asoibel@mpim-bonn.mpg.de}

\newtheorem{thm}{Theorem}[subsection]
\newtheorem{defn}[thm]{Definition}
\newtheorem{lmm}[thm]{Lemma}
\newtheorem{prp}[thm]{Proposition}
\newtheorem{conj}[thm]{Conjecture}
\newtheorem{exa}[thm]{Example}
\newtheorem{cor}[thm]{Corollary}
\newtheorem{que}[thm]{Question}
\newtheorem{ack}[thm]{Acknowledgments}
\newtheorem{clm}[thm]{Claim}
\newtheorem*{Kac}{Kac's Theorem}
\newtheorem*{DS}{The Deligne-Simpson Problem}
\newtheorem*{aDS}{The Additive Deligne-Simpson Problem}
\newtheorem*{mDS}{The Multiplicative Deligne-Simpson Problem}

\theoremstyle{definition}
\newtheorem{rmk}[thm]{Remark}

\newcommand{\dimn}{\operatorname{dim}}

\begin{abstract}
In ``Quantization of Hitchin's Integrable System and Hecke Eigensheaves", Beilinson and Drinfeld introduced the ``very good" property for a smooth complex equidimensional stack.  They prove that for a semisimple group $G$ over $\mathbb{C}$, the moduli stack $\text{Bun}_G(X)$ of G-bundles over a smooth complex projective curve $X$ is ``very good", as long as $X$ has genus $g > 1$.  In the case of the projective line, when $g = 0$, this is not the case.  However, the result can sometimes be extended to the projective line by introducing additional parabolic structure at a collection of marked points and slightly modifying the definition of a ``very good'' stack.  We provide a sufficient condition for the moduli stack of parabolic vector bundles over $\mathbb{P}^1$ to be very good.  We then use this property to study the space of solutions to the Deligne-Simpson problem.
\end{abstract}

\maketitle

\tableofcontents

\section{Introduction}
\input{dsintro.tex}

\subsection{Acknowledgments}
I am extremely grateful to my advisor D. Arinkin for his numerous comments, corrections, and explanations.  I would like to thank P. Belkale, A. Braverman, I. Cherednik, E. Frenkel, A. Goncharov, M. Kontsevich, S. Kumar, Z. Lin, A. Polishchuk, L. Rozansky, J. Sawon, A. Varchenko for the interest they have expressed in my work and for the useful discussions.  I would also like to thank W. Crawley-Boevey for clarifying the status of the multiplicative Deligne-Simpson problem and inspiring much of the work seen here.

\section{Very Good Property}

\input{very_good_property.tex}

\section{Quivers and their Representations}

\input{quivers_and_their_representations.tex}

\section{Moduli of Parabolic Bundles}

\input{moduli_of_parabolic_bundles.tex}

\section{Stability for Parabolic Bundles}

\input{stability_for_parabolic_bundles.tex}

\section{Quivers and Parabolic Bundles}

\input{quivers_and_parabolic_bundles.tex}

\section{Application to the Deligne-Simpson Problem}

\input{application_to_the_deligne-simpson_problem.tex}

\bibliographystyle{hacm}
\bibliography{ds}{}

\end{document}

%% file: dsintro.tex
\subsection{The Very Good Property}
In \cite{BD1991} Beilinson and Drinfeld introduced the notion of a ``very good" stack. They require this property in order to avoid using derived categories in their study of D-modules on the moduli stack $\textrm{Bun}_G(X)$ of $G$-bundles over $X$, where $G$ is a semisimple algebraic group and $X$ is a smooth complex projective curve. 

A smooth complex equidimensional stack $\mathcal{Y}$ will be called \textit{very good} if
$$
\textrm{codim} \{y \in \mathcal{Y}| \textrm{dim Aut}(y)  = n\} > n, \textrm{for } n>0,
$$
where $\textrm{Aut}(y)$ is the automorphism group of $y \in \mathcal{Y}$.  If $\textrm{dim Aut}(y) > 0$ for all $y \in \mathcal{Y}$, then the stack $\mathcal{Y}$ cannot be very good.  In this situation, $\mathcal{Y}$ will be called \textit{almost very good} if
$$
\textrm{codim} \{y \in \mathcal{Y}| \textrm{dim Aut}(y) - m  = n\} > n, \textrm{for } n>0,
$$
where $m = \textrm{min dim Aut}(y)$.  Beilinson and Drinfeld demonstrate that $\textrm{Bun}_G(X)$ is very good when $X$ has genus $g>1$.  However, in the $g = 0$ case, when $X = \mathbb{P}^1$, this is no longer true.  

We approach the very good property in the genus $g = 0$ case, for $G = \textrm{GL} (n,\mathbb{C})$, by introducing additional parabolic structure at a finite collection of marked points.  Since the reductive group $\textrm{GL} (n,\mathbb{C})$ has a one-dimensional central subgroup $\mathbb{C}^*$ that acts by dilation on the fibers, the automorphism group of any parabolic bundle has a one-dimensional subgroup.  It follows that the moduli stack of parabolic bundles can never be very good. 

It turns out, however, that a sufficiently elaborate parabolic structure on a vector bundle is enough to make the corresponding moduli stack of parabolic bundles over $\mathbb{P}^1$ almost very good.  This is equivalent to showing that the quotient of the moduli stack by the classifying stack of $\mathbb{C}^*$ is very good. 
\subsection{The Very Good Property for Moduli of Parabolic Bundles}

Seshadri introduced the notion of a parabolic structure on a vector bundle in \cite{Se1977}, furnishing parabolic bundles with a stability condition analogous to the usual one for vector bundles.  Expanding upon this, Mehta and Seshadri proved the existence of a moduli space of semistable parabolic bundles on a smooth projective curve of genus $g\ge 2$ in \cite{MS1980}.  

Parabolic bundles over an algebraic curve generalize vector bundles by defining additional structure in the fibers over specified points. Namely, let $X$ be a smooth complex projective curve (in the future, we restrict ourselves to the case when $X = \mathbb{P}^1$).  A \textit{parabolic bundle} $\mathbf{E}$ over $X$ consists of a vector bundle $E$ over $X$, a collection of distinct points $(x_1, \dots, x_k)$ on $X$, and a flag $E_{x_i} = E_{i0} \supseteq E_{i1} \supseteq E_{iw_{i-1}}  \supseteq E_{iw_i} = 0$ in the fiber over each such point $x_i$.  

If $D = (x_1, \dots, x_k)$ and $w = (w_1, \dots, w_k)$, we say that the parabolic bundle $\mathbf{E}$ has \textit{weight type} $(D,w)$.  If $\alpha_0 = \textrm{rk } E$ and $\alpha_{ij} = \textrm{dim } E_{ij}$, for $1 \le i \le k$ and $ 1 \le j \le w_{i-1}$, we say that $\mathbf{E}$ has \textit{dimension vector} $\alpha = (\alpha_0,\alpha_{ij})$.   

Note that one possible way of introducing stability and semistability for parabolic bundles, is by defining a \textrm{parabolic degree}.  To do this, additional numbers called \textit{weights} are assigned to each subspace in each flag.  Since we do not limit ourselves to stable or semistable parabolic bundles, we do not require weights to be part of the definition.  Parabolic bundles without weights are sometimes referred to as ``quasi-parabolic" bundles.

In order to formulate our main result, we need to specify which dimension vectors give rise to very good parabolic bundles.  Let $I = \{0\} \cup \{(i,j)|1 \le i \le k, 1 \le j \le w_{i-1} \}$. For a dimension vector $\alpha \in \mathbb{Z}^I$, we define the \textit{Tits quadratic form} as:
$$
q(\alpha) = \sum_{i\in I} \alpha_i^2 -\sum_{i\in I} \alpha_i\alpha_{i+1},  
$$
where $\alpha_{w_i} = 0$.  Let $p(\alpha) = 1 - q(\alpha)$.  We write:
$\delta(\alpha) = -2\alpha_0 + \sum_i \alpha_{i1}$.
We say that $\alpha$ is in the \textit{fundamental region} if
\begin{align*}
&\delta(\alpha) \ge 0\\
-2\alpha_{ij} + \alpha_{ij-1} + \alpha_{ij+1} \ge 0, & \textrm{ for } 1 \le i \le k \textrm{ and } 1 \le j \le w_{i-1} 
\end{align*}
(note that we assume $\alpha_{i0} = \alpha_0$, for all $i$).  We now introduce our main result.

\begin{thm}
The moduli stack $\textrm{Bun}_{D,w,\alpha}(\mathbb{P}^1)$ of parabolic bundles over $\mathbb{P}^1$  of weight type $(D,w)$ and dimension vector $\alpha$ is almost very good if $\alpha$ is in the fundamental region and $\delta(\alpha) > 0$.
\end{thm}
Note that in this case $m = 1$. The vector $\alpha$ can be used to define a product of partial flag varieties 
$$Fl(\alpha) = \\ \prod_i Fl(\alpha_0, \alpha_{i1}, \dots, \alpha_{iw_i}).$$  
That is, $\alpha_0$ is the dimension of the ambient space $\mathbb{C}^{\alpha_0}$, and for a fixed $1 \le i \le k$, each $\alpha_{ij}$ is the dimension of the $j$-th subspace in the flag.  The group $\textrm{PGL}(\alpha_0)$ acts diagonally on $Fl(\alpha)$, so it makes sense to discuss the very good property of the resulting quotient stack.  Indeed, when the underlying vector bundle is trivial, we can use Theorem 1.2.1 to obtain:

\begin{thm}
The quotient stack $\textrm{PGL}(\alpha_0)\backslash Fl(\alpha)$ is very good, if $\alpha$ is in the fundamental region and $\delta(\alpha) > 0$.
\end{thm}

Theorem 1.2.2 may also be obtained from Crawley-Boevey's results in \cite{CB2001}, after noticing that $Fl(\alpha)$ is the quotient of the space of star-shaped quiver representations of dimension $\alpha$ with injective arrows by the group $H(\alpha) = \prod_{i,j} \textrm{GL}(\alpha_{ij})$, acting by conjugation on the arrows.  In this case, the very good property is equivalent to Crawley-Boevey's inequality $p(\alpha) > \sum_i p(\beta_i)$ (see \cite{CB2001}), for any decomposition $\alpha = \sum_i \beta_i$ into the sum of positive roots corresponding to the star-shaped quiver (see Sections 3.3 and 3.4 below).  The condition that $\alpha$ is in the fundamental region and $\delta(\alpha) > 0$ implies this inequality. 

\subsection{The Deligne-Simpson Problem}

Let $D = (x_1, \dots, x_k)$ be a collection of distinct points on a Riemann surface $X$.  Let $\Omega_{X}^1(\textrm{log } D)$ be the sheaf of \textit{logarithmic differential forms} on $X$.  That is, the sections of $\Omega_{X}^1(\textrm{log } D)$ are differential forms that have a pole of order at most one at each point in $D$.  A \textit{logarithmic connection} or a \textit{connection with regular singularities} (in $D$) on a vector bundle $E$ over $X$ is a $\mathbb{C}$-linear morphism
\begin{align*}
&\nabla: E \rightarrow E \otimes \Omega_{X}^1(\textrm{log } D) \textrm{ such that }\\
&\nabla(fs) = s \otimes df + f\nabla(s) \textrm{ for } f \in \mathcal{O}_{X},  s \in E.
\end{align*}
Note that the connection $\nabla$ has residues at the points of $D$, so that there exists $\textrm{Res}_{x_i} \nabla \in \textrm{End}(E_{x_i})$, for each fiber $E_{x_i}$ over $x_i \in D$.

Let $C_1, \dots, C_k$ be conjugacy classes of complex, linear endomorphisms of vector spaces of dimension $n$.  We can formulate the following:
\begin{DS}
Does there exist (for some $D$ and vector bundle $E$) a connection $\nabla$ on $\mathbb{P}^1$ with regular singularities such that $\textrm{Res}_{x_i} \nabla \in C_i$?
\end{DS}
We will use this formulation of the Deligne-Simpson problem instead of the ones given below, as it is easier to generalize to the case of connections with irregular singularities (see Section 1.5).


The \textit{Riemann-Hilbert correspondence} provides an equivalence between the category of connections $\nabla$ with regular singularities in $D$ on vector bundles over $\mathbb{P}^1$ and the category of representations of the fundamental group of $\mathbb{P}^1-D$ by way of the monodromy representation of $\nabla$ (see \cite{De1970}).  This provides the following reformulation of the Deligne-Simpson problem:
\begin{mDS}
Given $k$ conjugacy classes  $C_1,$ $\dots, C_k$ of complex matrices in $\textrm{GL}(n,\mathbb{C})$, do there exist $A_1 \in C_1, \dots, A_k \in C_k$ such that $A_1 \cdot A_2 \cdots  A_k = \textrm{Id}$?
\end{mDS}
This was the original version of the Deligne-Simpson problem, suggested in a letter from Deligne to Simpson, who considered it in his paper \cite{Si1991}.  

By considering connections on trivial (and trivialized) vector bundles over $\mathbb{P}^1$ we get another version of the Deligne-Simpson problem:
\begin{aDS}
Given $k$ conjugacy classes  $C_1, \dots, C_k$ of complex matrices in $\mathfrak{gl}_n(\mathbb{C})$, do there exist $A_1 \in C_1, \dots, A_k \in C_k$ such that  
$A_1 + \cdots +  A_k = 0$?
\end{aDS}
The multiplicative Deligne-Simpson problem and its additive analogue were studied by Crawley-Boevey in \cite{CB2003}-\cite{CB2004}, Katz in \cite{Katz1996}, Kostov in \cite{Ko1999}, \cite{Ko2001}, \cite{Ko2002}, \cite{Ko2003}, \cite{Ko22004}, \cite{Ko12004}, \cite{Ko2005}, and Simpson in \cite{Si2009}, among others.




There are several approaches to solving the Deligne-Simpson problem. In \cite{Katz1996}, Katz describes an algorithm for the existence of rigid local systems, which Kostov applies in \cite{Ko1999}-\cite{Ko2005} to determine when solutions to various cases of the Deligne-Simpson problems exist.  The algorithm, called the \textit{middle convolution algorithm},  proceeds by changing the rank of the local system by a number $\delta$, called the \textit{defect}, dependent on $C_1, \dots, C_k$.  Solutions exist for the original rank, as long as they exist for the altered rank.  This continues until $\delta \ge 0$, in which case there are solutions by a nontrivial existence theorem, or until one arrives at a situation when solutions cannot exist.  

In \cite{CB2003}, Crawley-Boevey proposes another approach to solving the additive version of the Deligne-Simpson problem by examining fibers of the moment map on the cotangent bundle to the space of representations of the star-shaped quiver and the representations of the deformed preprojective algebra associated to this quiver.  This gives him a necessary and sufficient condition for the existence of solutions in the additive case.  In \cite{CBS2006}, he and Shaw provide a sufficient condition for the existence of solutions of the multiplicative Deligne-Simpson problem using a multiplicative analogue of the preprojective algebra.  This condition is also necessary (\cite{CB2013}).  A multiplicative analogue of the moment map approach of \cite{CB2003} may be found in \cite{Y2008}.  

\subsection{The Deligne-Simpson problem and the very good property}

Let $\mathbf{E}$ be a parabolic bundle over $\mathbb{P}^1$ of weight type $(D,w)$.  Let $\zeta=(\zeta_{ij})_{1\le i\le k, 1\le j\le w_i}$.  A \textit{$\zeta$-parabolic connection} on $\mathbf{E}$ is a connection $\nabla$ on the underlying vector bundle $E$ with regular singularities in $D$, such that 
$$
(\textrm{Res}_{x_i} \nabla - \zeta_{ij})(E_{ij-1}) \subset E_{ij},
$$
for all $1 \le i \le k$ and $1 \le j \le w_i$.

Given semisimple conjugacy classes $C_1, \dots, C_k$ of $n$-dimensional complex vector space endomorphisms and an ordering on the eigenvalues of these conjugacy classes, one can write a dimension vector $\alpha$, where $\alpha_0 = n$ and $\alpha_{ij}$ is the dimension of the direct sum of the first $j$ eigenspaces, for the above ordering on the eigenvalues.  One can also obtain a complex vector $\zeta = (\zeta_{ij})$ simply as the vector of eigenvalues for $C_1, \dots, C_k$, counting multiplicity.  For these $\zeta$ and $\alpha$, the $\zeta$-parabolic connections on parabolic bundles with dimension vector $\alpha$ over $\mathbb{P}^1$ will have residues in the conjugacy classes $C_1, \dots, C_k$. 

Conversely, a $\zeta$-parabolic connection on a parabolic bundle with dimension vector $\alpha$ over $\mathbb{P}^1$ determines semisimple conjugacy classes $C_1, \dots, C_k$, with $\zeta$ being the vector of eigenvalues (counting multiplicity), and  $\alpha_{ij}-\alpha_{ij+1}$ being the dimension of the eigenspace for $\zeta_{ij}$.

Given the situation described in the previous two paragraphs, it follows that semisimple conjugacy classes may be used to determine (not uniquely) a moduli stack of parabolic bundles $\textrm{Bun}_{D,w,\alpha}(\mathbb{P}^1)$.  Furthermore,  the moduli stack of solutions of the Deligne-Simpson problem may be defined as $\textrm{Conn}_{D,w,\alpha,\zeta}(\mathbb{P}^1)$, the moduli stack of $\zeta$-parabolic connections on parabolic bundles over $\mathbb{P}^1$ of weight type $(D,w)$ and dimension vector $\alpha$.  By presenting $\textrm{Conn}_{D,w,\alpha,\zeta}(\mathbb{P}^1)$ as a twisted cotangent bundle over the moduli stack of parabolic bundles $\textrm{Bun}_{D,w,\alpha}(\mathbb{P}^1)$, we prove the following theorem:
\begin{thm}
If $\textrm{Bun}_{D,w,\alpha}(\mathbb{P}^1)$ is almost very good and $\sum_{i=1}^k \sum_{j=1}^{w_i} \zeta_{ij}(\alpha_{ij-1} - \alpha_{ij})$ is an integer, then $\textrm{Conn}_{D,w,\alpha,\zeta}(\mathbb{P}^1)$ is a nonempty, irreducible, locally complete intersection of dimension $2p(\alpha)-1$.
\end{thm}

Theorem 1.4.1 and Theorem 1.2.1 give us the following corollary:
\begin{cor}
If $\alpha$ is in the fundamental region, $\sum_{i=1}^k \sum_{j=1}^{w_i} \zeta_{ij}(\alpha_{ij-1} - \alpha_{ij})$ is an integer, and $\delta (\alpha) > 0$, then $\textrm{Conn}_{D,w,\alpha,\zeta}(\mathbb{P}^1)$ is a nonempty, irreducible, locally complete intersection of dimension $2p(\alpha)-1$.  
\end{cor}

If the vector bundles underlying the parabolic bundles are trivial, then Theorem 1.4.1 may be used to obtain the following: 

\begin{thm}
If the conjugacy classes $C_i$ are semisimple, the corresponding quotient stack $\textrm{PGL}(\alpha_0) \backslash Fl(\alpha)$ is very good and the eigenvalues of all the $C_i$ add up to $0$, then the space of solutions of the additive Deligne-Simpson problem for $C_1, \dots, C_k$ is a nonempty, irreducible complete intersection of dimension $2\cdot \dimn Fl(\alpha) - \alpha_0^2 + 1$.  
\end{thm}
Applying Theorem 1.2.2 we obtain:
\begin{cor}
If the conjugacy classes $C_i$ are semisimple, the eigenvalues of all the $C_i$ add up to $0$, $\alpha$ is in the fundamental region, and $\delta(\alpha) > 0$, then the space of solutions of the additive Deligne-Simpson problem for $C_1, \dots, C_k$ is a nonempty, irreducible complete intersection of dimension $2\cdot\textrm{dim } Fl(\alpha) - \alpha_0^2 + 1$.
\end{cor}

We can obtain results similar to Theorem 1.4.3 and Corollary 1.4.4 for the multiplicative Deligne-Simpson problem. Indeed, let $C_1, \dots , C_k$ be semisimple conjugacy classes in $\textrm{GL}(n,\mathbb{C})$.  The Riemann-Hilbert correspondence provides an analytic isomorphism between the space of solutions to the multiplicative Deligne-Simpson problem for $C_1, \dots, C_k$ and a certain moduli space of $\zeta$-parabolic connections.  This is similar to the analytic isomorphism obtained for the moduli space of stable $\zeta$-parabolic connections in \cite{In2013}, \cite{IIS2006}, or \cite{Y2008}.  We get the following:

\begin{thm}
If we have that the conjugacy classes $C_i$ are semisimple, the corresponding moduli stack $\textrm{Bun}_{D,w,\alpha}(\mathbb{P}^1)$ is almost very good, and the eigenvalues of all the $C_i$ multiply to $1$, then the space of solutions of the multiplicative Deligne-Simpson problem for $C_1, \dots, C_k$ is a nonempty, irreducible complete intersection of dimension $2\cdot\dimn Fl(\alpha) - \alpha_0^2 + 1 = 2p(\alpha) + \alpha_0^2 - 1$.  
\end{thm}
Applying Theorem 1.2.1 we obtain:
\begin{cor}
If the conjugacy classes $C_i$ are semisimple, the eigenvalues of all the $C_i$ multiply to $1$, $\alpha$ is in the fundamental region and $\delta(\alpha) > 0$, then the space of solutions of the multiplicative Deligne-Simpson problem for $C_1, \dots, C_k$ is a nonempty, irreducible complete intersection of dimension $2\cdot \dimn Fl(\alpha) - \alpha_0^2 + 1 = 2p(\alpha) + \alpha_0^2 - 1$.  
\end{cor}

\begin{rmk}
In the above corollaries, $\delta(\alpha)$ is actually equal to the defect $\delta$ that appears in Katz's middle convolution algorithm.  Moreover, for the specific ordering on the eigenspaces described above, the condition that $\alpha$ is in the fundamental region reduces to $\delta(\alpha) \ge 0$.  Therefore, $\delta(\alpha) > 0$ alone is sufficient to obtain the properties for the space of solutions. 
\end{rmk}


\subsection{Further Discussion}

In our formulation, the Deligne-Simpson problem asks whether there exist connections on $\mathbb{P}^1$ with simple poles such that the residues lie in prescribed conjugacy classes.  It is also possible to ask a similar question for connections with poles of higher order.  

We replace the idea of a logarithmic connection on $\mathbb{P}^1$ that has residues in prescribed conjugacy classes with the more general one of a connection with irregular singularities that has prescribed \textit{formal types}.  The notion of formal type (see e.g. \cite{Ar2010}) allows one to classify connections with irregular singularities based on their restrictions to formal neighborhoods of points.  Using this notion it is possible to formulate a more general version of the Deligne-Simpson problem by asking whether there exist connections with irregular singularities on $\mathbb{P}^1$ with prescribed formal types at a fixed collection of points $D$ on $\mathbb{P}^1$.

Hiroe in \cite{Hi2013} solves the ``additive'' version of this problem (when the connections are on trivial vector bundles) by using Boalch's quiver construction from \cite{Bo2008}.  This approach, similar to what Crawley-Boevey does in \cite{CB2003} for the case of regular singularities, suggests that it is possible to apply the very good condition to obtain certain geometric properties for the space of solutions to the irregular version of the additive Deligne-Simpson problem.  Moreover, it may be possible to generalize representations of squids, in order to study the space of solutions to the general version of the irregular Deligne-Simpson problem.

It would also be interesting to extend the result of Theorem 1.2.1 to other reductive groups.  By analogy with flag varieties, it is possible to define a parabolic structure on $G$-bundles, when $G$ is not $GL(n,\mathbb{C})$, by specifying parabolic subgroups $P_i$ at each marked point $x_i \in \mathbb{P}^1$.  Although there is no correspondence with quiver representations for a general $G$, it may be possible to modify Beilinson and Drinfeld's original proof of the very good property for $\textrm{Bun}_G$.  A key part of their argument consists of showing that the \textit{global nilpotent cone} $\textrm{Nilp}(G)$ (introduced in \cite{La1987} and \cite{La1988}), the fiber over $0$ in the Hitchin system, is Lagrangian (see \cite{Gi2001}).  One can consider the parabolic analogue of the Hitchin system, which has its own global nilpotent cone. It has been proved to be Lagrangian in specific instances, such as for complete flags (\cite{F1993}, \cite{Su2007}) or rank 3 (\cite{GP2007}).  However, the author is unaware of a proof for the case of partial flags.

\subsection{Plan}

In the second section, we begin by defining the good property and its variants for algebraic stacks.  We reformulate these properties in terms of the dimension of the corresponding inertia stacks, so that they become easier to prove.  Specifically, we examine the case of the quotient stack, when a semisimple algebraic group $G$ acts on a variety $X$.  In this case, the dimension of the inertia stack is simply the number of parameters (e.g. see \cite{CB2001} or the proof of the Kac Theorem in \cite{CB1993}).  We also apply the very good property for the quotient stack $G \backslash X$ to study the geometry of the fibers of the moment map on $T^*X$. 

In the third section, we demonstrate the very good property for the quotient stack arising from quiver representations.  We largely follow Crawley-Boevey's arguments in \cite{CB1993} and \cite{CB2001} for computing the number of parameters by means of the Kac Theorem.  We also introduce the space of squid representations (following \cite{CB2004}), which will be used to construct moduli of parabolic bundles later.  We then look at the cotangent bundle to this space.  We finish by discussing the special case of star-shaped quiver representations, demonstrating the very good property for the associated quotient stack.

In the fourth section, we prove our main result, Theorem 1.2.1, concerning the almost very good property for the moduli stack of parabolic bundles over $\mathbb{P}^1$.  A key element in the proof involves an estimate of the dimension of the inertia stack.

In the fifth section, we introduce stability and semistability for parabolic bundles.  We explain how to simplify the proof of Theorem 1.2.1 if we restrict ourselves to the open substack of semistable parabolic bundles.  We finish by relating semistable parabolic bundles to semistable quiver representations.

In the sixth section, we introduce several moduli spaces of parabolic bundles, relating parabolic bundles to flag bundles and squid representations.  As an additional example of this relationship, we prove Theorem 1.2.2 in two different ways, using the results of the third and fourth sections.

In the final section, we define a moduli space of $\zeta$-parabolic connections on $\mathbb{P}^1$ in terms of a moment map on the cotangent bundle to squid representations.  We use this moduli space and Theorem 1.2.1 in order to prove Theorem 1.4.1 and Corollary 1.4.2, demonstrating certain properties of the moduli stack of solutions to the Deligne-Simpson problem.  Subsequently, we derive analogous properties for the specific case of the additive Deligne-Simpson problem, proving Theorem 1.4.3 and Corollary 1.4.4.  Finally, we use the Riemann-Hilbert correspondence, Theorem 1.4.1, and Corollary 1.4.2 in order to prove Theorem 1.4.5 and Corollary 1.4.6, which confer the same properties onto the space of solutions to the multiplicative Deligne-Simpson problem. 

%% file: very_good_property.tex
\subsection{Outline}

We wish to define the notions of \textit{good}, \textit{very good}, and \textit{almost very good} for the algebraic stack $\mathcal{Y}$, proving several results that will make it easier to check for these properties.  We do this by reducing each property to inequalities involving components of the inertia stack associated with $\mathcal{Y}$.  

We also consider these properties in the specific case when $\mathcal{Y}$ is the quotient stack $G \backslash X$, obtained from the action of an algebraic group $G$ on a variety $X$.  In this case, we introduce (following \cite{CB1993}) the \textit{number of parameters} and prove a couple of technical lemmas we will use later. As a useful application of $G \backslash X$ being very good, we derive several geometric properties for the fiber of the moment map defined on $T^*X$ by the natural action of $G$.

\subsection{Definitions}

Let $\mathcal{Y}$ be an equidimensional algebraic stack over $\mathbb{C}$, and denote by $\textrm{Aut}(y)$ the automorphism group of $y \in \mathcal{Y}$. Let $\mathcal{Y}^n = \{y \in \mathcal{Y}| \textrm{dim Aut}(y) = n\}$, which gives rise to a reduced locally closed substack of $\mathcal{Y}$.    
\begin{defn}[\cite{BD1991}]
\label{2.1.1}
We call $\mathcal{Y}$ good when: 
$$
\textrm{codim } \mathcal{Y}^n \ge n \quad \forall n > 0,
$$
and we call it very good when: 
$$
\textrm{codim } \mathcal{Y}^n > n \quad \forall n > 0.
$$
\end{defn}
In the case when $\mathcal{Y}$ is smooth, being good is equivalent to the condition that $\textrm{dim } T^*\mathcal{Y}  = 2 \textrm{ dim } \mathcal{Y}$, where $T^*\mathcal{Y}$ is the cotangent stack to $\mathcal{Y}$ (see \cite{BD1991}).  Furthermore, $\mathcal{Y}$ is very good if and only if $T^*\mathcal{Y}^0$ is dense in $T^*\mathcal{Y}$.  Now, suppose there exists an integer $m > 0$ such that for all $y \in \mathcal{Y}$ we have $\textrm{dim Aut}(y) \ge m$.  In this case, we can see that $\mathcal{Y}$ cannot be very good.  
\begin{defn}
\label{2.1.2}
Let $m = \textrm{min dim Aut}(y)$ over all $y \in \mathcal{Y}$.  We say $\mathcal{Y}$ is \textit{almost good} if:
$$
\textrm{codim } \mathcal{Y}^{n+m} \ge n \quad \forall n > 0,
$$  
and we say it is almost very good if: 
$$
\textrm{codim } \mathcal{Y}^{n+m} > n \quad \forall n > 0.
$$  

\end{defn}

\begin{rmk}
\label{2.1.3}
Since $GL(n,\mathbb{C})$ contains the 1-dimensional center $\mathbb{C}^*$, the moduli stack of parabolic bundles provides an example of a stack that cannot be very good.  Instead, we offset the dimension of each automorphism group by $1$ and show that the stack is almost very good.  As we pointed out in Section 1.1, this is equivalent to showing that the quotient of moduli stack of parabolic bundles by the classifying stack of $\mathbb{C}^*$ is very good.  
\end{rmk}

\subsection{The very good property and the inertia stack}
In order to prove our Theorem 1.2.1, we will need to reformulate the very good property in terms of the inertia stack.  Let $\mathcal{I}_{\mathcal{Y}}$ be the inertia stack associated with the stack $\mathcal{Y}$,
which consists of pairs $(y, f)$, such that $y \in \mathcal{Y}$ and $f \in \textrm{Aut}(y)$.
We will be using the following lemma (see Properties of Algebraic Stacks in \cite{Sta2014}):  
 
\begin{lmm}
\label{2.2.1}
Let $f: \mathcal{X}_1 \rightarrow \mathcal{X}_2$ be a flat morphism of stacks of finite type and let $x \in \mathcal{X}_1$.  We have: 
$$
\dimn_{x}(\mathcal{X}_1)_{f(x)} = \dimn_{x} \mathcal{X}_1 - \dimn_{f(x)} \mathcal{X}_2,
$$
where $(\mathcal{X}_1)_{f(x)}$ is the fiber over $f(x)$. 
\end{lmm}

Now, we can obtain:
\begin{thm}
\label{2.2.2}
The stack $\mathcal{Y}$ is good if and only if $\dimn \mathcal{I}_{\mathcal{Y}} \le \dimn \mathcal{Y}$.
\end{thm}
\begin{proof}
Let $\mathcal{I}^n$ be the locally closed, reduced substack of $\mathcal{I}_{\mathcal{Y}}$ consisting of objects $(y,g)$ such that $\textrm{dim Aut}(y) = n$.  Furthermore, let
$f: \mathcal{I}_{\mathcal{Y}} \rightarrow \mathcal{Y}$ be the canonical morphism and let $f_n: \mathcal{I}^n \rightarrow \mathcal{Y}^n$ be its restriction to $\mathcal{I}^n$.  By Lemma \ref{2.2.1}, we have that: 
$$
\textrm{dim } \mathcal{I}^n = n + \textrm{dim } \mathcal{Y}^n.
$$
Note that $\textrm{dim } \mathcal{I}^0 = \textrm{dim } \mathcal{Y}^0$.  Now, suppose $\mathcal{Y}$ is good.  This implies $\textrm{dim } \mathcal{I}^n \le \textrm{dim } \mathcal{Y}$ for $n > 0$.  By the definition of dimension, there exists an $n \ge 0$ such that $\textrm{dim } \mathcal{I}^n = \textrm{dim } \mathcal{I}_{\mathcal{Y}}$.  It follows that $\textrm{dim } \mathcal{I}_{\mathcal{Y}} \le \textrm{dim } \mathcal{Y}$.

Now, suppose $\textrm{dim } \mathcal{I}_{\mathcal{Y}} \le \textrm{dim } \mathcal{Y}$.  We have that: 
$$
n + \textrm{dim } \mathcal{Y}^n = \textrm{dim } \mathcal{I}^n \le \textrm{dim } \mathcal{Y},
$$
for all $n \ge 0$.  Therefore, we obtain that $\textrm{codim } \mathcal{Y}^n \ge n$ for all $n > 0$, and $\mathcal{Y}$ is good.
\end{proof}

From this theorem we can then obtain:

\begin{cor}
\label{2.2.3}
The stack $\mathcal{Y}$ is very good if and only if $\dimn (\mathcal{I}_{\mathcal{Y}} - \mathcal{I}^0) < \dimn \mathcal{Y}$.
\end{cor}

Similarly, we have:

\begin{cor}
\label{2.2.4}
Let $m$ and $\mathcal{I}^n$ be as before.  The stack $\mathcal{Y}$ is almost very good if and only if $\dimn (\mathcal{I}_{\mathcal{Y}} - \coprod_{i = 0}^m \mathcal{I}^i) - m < \dimn \mathcal{Y}$.
\end{cor}

\subsection{The very good property for quotient stacks}
The case when $\mathcal{Y}$ is a quotient stack is of special interest, so we will examine it in detail.
Let $X$ be a variety over $\mathbb{C}$, and let $G$ be an algebraic group over $\mathbb{C}$, acting on $X$. For $y \in \mathcal{Y} = G \backslash X$, we have that $\textrm{dim Aut}(y) = \textrm{dim } G_x$, where $G_x$ is the stabilizer subgroup of a point $x \in X$ corresponding to $y$. 

Note that if the subgroup $H = \{g \in G| g \cdot x = x \textrm{ for all } x \in X  \}$ has nonzero dimension, then $\mathcal{Y}$ cannot be very good.  However, since $H$ is a closed, normal subgroup of $G$, we may instead consider the quotient stack $(G/H)\backslash X$, which may still be good or very good.  We introduce the following definition (see e.g. \cite{CB1993}):

\begin{defn}
\label{2.3.1}
If $G$ is an algebraic group acting on $X$ and $Y \subset X$ is a $G$-stable constructible subset, then we define the \textit{number of parameters} of $G$ on $Y$ as 
$$\textrm{dim}_G Y = \textrm{max}_{s} \{ \textrm{dim } Y \cap X_s + s - \textrm{dim } G \},$$ 
where $X_s = \{ x \in X| \textrm{dim }G_x = s \}$.  
\end{defn}
We can easily see that the number of parameters for $Y = X$ is simply the dimension of the inertia stack associated to the quotient stack $\mathcal{Y}$.  Therefore, by Theorem \ref{2.2.2}, the good condition on $G \backslash Y$ is equivalent to 
$$\textrm{dim}_G X \le \textrm{dim } X - \textrm{dim } G.$$  
Similarly, we can apply Corollary \ref{2.2.3} in order to obtain that $\mathcal{Y}$ is very good if and only if 
$$\textrm{dim}_G X_n < \textrm{dim } X - \textrm{dim } G \textrm{ for all } n > 0.$$

We will be using the following two lemmas from \cite{CB2001}, the first of which is obvious:

\begin{lmm}
\label{2.3.2}
Suppose we have algebraic groups $G_i$ acting on schemes $X_i$.  Let $Y_i \subseteq X_i$ be constructible subsets stable under the action of $G_i$.  We have that $\dimn_G Y = \sum_i \dimn_{G_i} Y_i$, where $G = \prod_i G_i$ and $Y = \prod_i Y_i$.
\end{lmm}

\begin{lmm}
\label{2.3.3}
Let $X$ be a scheme with an algebraic group $G$ acting on it, and $H \subseteq G$ be a closed subgroup.  Furthermore, let $Z \subseteq Y \subseteq X$ be constructible subsets of X, where $Y$ is $G$-stable and $Z$ is $H$-stable.  If $Y = G\cdot Z$ and $Z$ intersects any orbit of $G$ along a finite union of $H$-orbits, then we have that $\dimn_H Z = \dimn_G Y$.
\end{lmm}

\begin{proof}
Let $Z_{s,s'} = \{z \in Z| \textrm{dim } H_z = s, \textrm{dim } G_z = s' \}$.  Since $G$ is an algebraic group and $H$ its subgroup, then we have that $Z$ (and $Z_d$) is stratified by a finite number of the $Z_{s,s'}$, and $Y$ is stratified by a finite number of the $Y_{s'}$.  Also, note that $Z_{s,s'}$ and $Y_{s'}$ are both locally closed, making them constructible subsets of $Z$ and $Y$.  Now, we have, for each $s$ and $s'$, a morphism: 
				\begin{align*}
				 	f_{s,s'}: \quad & G \times Z_{s,s'} \rightarrow Y_{s'} \\
				 		 & (g, z) \mapsto g\cdot z.
				 \end{align*}
				 
The fiber of $f_{s,s'}$ over $y \in Y_{s'}$ consists of pairs $(g,z)$, with $z$ belonging to the same $G$-orbit as $y$ and $g$ contained in a coset of $G_z$.  Because the intersection of the $G$-orbit of $y$ with $Z$ is a finite union of $H$-orbits, we have 
$$\textrm{dim } f_s^{-1}(y)  = s' + \textrm{dim } H - s.$$  
Since for any $y \in Y_{s'}$ there is a $z \in Z_{s,s'}$ for some $s$ such that $f_{s,s'}(z) = y$, then $Y_s'$ is covered by the images of the morphisms $f_{s,s'}$.  By Chevalley's Theorem each such image is constructible, and therefore, $\textrm{dim } Y_{s'} = \textrm{dim } f_{s,s'}(Z_{s,s'})$, for some $s$.  It follows that 
$$\textrm{dim } Z_{s,s'} + \textrm{dim } G = \textrm{dim } Y_{s'} + s' + \textrm{dim } H - s,$$
 which can be rewritten as 
$$ \textrm{dim } Z_{s,s'} + s - \textrm{dim } H = \textrm{dim } Y_{s'} + s' - \textrm{dim } G,$$ 
for a specific $s$.  Now, taking a maximum of both sides over $s'$ we obtain 
$$\textrm{dim}_H Z \ge \textrm{max}_{s'}\{\textrm{dim } Z_{s,s'} + s - \textrm{dim } H\} = \textrm{dim}_G Y.$$
Conversely, since the $Z_{s,s'}$ stratify $Z_s$, there is a value $s$ such that $\textrm{dim} Z_s = Z_{s,s'}$.  Using the above computations for $f_{s,s'}$, we obtain
$$ \textrm{dim } Z_{s,s'} - \textrm{dim } H + s = \textrm{dim } Y_{s'} + s' - \textrm{dim } G.$$  
As we take the maximum over $s$ of both sides we obtain 
$$\textrm{dim}_H Z = \textrm{max}_{s} \{ \textrm{dim } Y_{s'} + s' - \textrm{dim } G \} \le \textrm{dim}_G Y.$$ 
 Together, the two inequalities give us $\textrm{dim}_H Z = \textrm{dim}_G Y$. 
\end{proof}

\subsection{The very good property and the moment map}

Let $X$ be a smooth algebraic variety over $\mathbb{C}$ with a semisimple complex group $G$ acting on it.  This gives rise to a natural Hamiltonian $G$-action on the cotangent bundle $T^*X$ equipped with the standard symplectic form. There is a moment map $\mu: T^*X \rightarrow \mathfrak{g}^*$, defined by:
$$
\mu(y)(\xi) = y(\xi_X(x)),
$$
where $\mathfrak{g}$ is the Lie algebra of $G$, $y \in T^*_xX$, and $\xi_X$ is the vector field on $X$ induced by $\xi \in \mathfrak{g}$.  It is clear from the above description that $\mu$ is linear on each cotangent space $T^*_xX$.  Therefore, the image is a vector subspace of $\mathfrak{g}^*$.

\begin{lmm}
\label{2.4.1}
The image $\mu(T^*_xX)$ is the annihilator of $\mathfrak{g}_x$, where $\mathfrak{g}_x$ is the Lie algebra of the stabilizer of $x \in X$ under the action of $G$.
\end{lmm}
\begin{proof}
Let $\mathfrak{g}_x^{\perp}$ be the annihilator of $\mathfrak{g}_x$ and consider $\xi \in \mathfrak{g}_x$.  For $y \in X$, let $f_y: G \rightarrow X$ be the map that takes $g \in G$ to $g\cdot y$.  By definition, 
$$\xi_X(x) = (df_x)_e(\xi),$$
where $e$ is the identity element of $G$, so we have that $\xi_X(x) = 0$.  Therefore, $\mu(T^*_xX) \subset \mathfrak{g}_x^{\perp}$.  We can compute the dimension of $\mu(T^*_xX)$ as
$$
\textrm{dim }\mu(T^*_xX) = \textrm{dim } T^*_x X - \textrm {dim ker} \left.\mu\right|_{T^*_xX}.
$$
Let $V \subset T_xX$ be the vector subspace spanned by $\xi_X(x)$ for all $\xi \in \mathfrak{g}$.  By definition, $\textrm{ker} \left.\mu\right|_{T^*_xX}$ is the annihilator of $V$.  Therefore, we have that
$$
\textrm{dim ker} \left.\mu\right|_{T^*_xX} = \textrm{dim } T_xX - \textrm{dim V}.
$$
Note that $\mathfrak{g}_x$ contains all $\xi \in \mathfrak{g}$ such that $\xi_X(x) = 0$.  It follows that $\textrm{dim V} = \textrm{dim }\mathfrak{g} - \textrm{dim }\mathfrak{g}_x$.  Thus:
$$
\textrm{dim }\mu(T^*_xX) = \textrm{dim }\mathfrak{g} - \textrm{dim }\mathfrak{g}_x = \textrm{dim } \mathfrak{g}_x^{\perp},
$$
and $\mu(T^*_xX) = \mathfrak{g}_x^{\perp}$.
\end{proof}

Note that the moment map is algebraic, so the fiber $\mu^{-1}(\theta)$ is a closed algebraic subvariety of $T^*X$ for any $\theta \in \mathfrak{g}^*$.
We are now ready to prove the following theorem:
\begin{thm}
\label{2.4.2}
If the quotient stack $G \backslash X$ is very good, then for any $\theta \in \mathfrak{g}^*$ we have that $\mu^{-1}(\theta)$ is a nonempty, equidimensional variety of dimension $2 \dimn X - \dimn G$.  Moreover, there is a bijective correspondence between the irreducible components of $\mu^{-1}(\theta)$ and the irreducible components of $X$.
\end{thm}
\begin{proof}
Let $x \in X$ and let $\pi: T^*X \rightarrow X$ be the natural projection.  By Lemma \ref{2.4.1} we have that 
$$
\dimn \mu(\pi^{-1}(x)) = \dimn \mathfrak{g}^* - \dimn G_x.
$$
Let $X_0 = \{x \in X| \dimn G_x = 0 \}$. We have that $\pi^{-1}(X_0) \cap \mu^{-1}(\theta)$ is nonempty.  Since $G\backslash X$ is very good, then $X_0$ is nonempty.  Consequently,  $\mu$ is surjective, and we have:
$$
\dimn \mu^{-1}(\theta) \ge = 2\dimn X - \dimn G.
$$
In fact, for every irreducible component $I$ of $\mu^{-1}(\theta)$ we have that $\dimn I \ge  2\dimn X - \dimn G$.  

Let $p$ be the restriction of $\pi$ to $\mu^{-1}(\theta)$ and let $I$ be an irreducible component of $\mu^{-1}(\theta)$, as above. Since $X$ is stratified by the dimension of the stabilizer of the $G$-action, there exists an $m \ge 0$ such that  
$$
 \dimn X - \dimn G + m = \dimn I - \dimn p(I).
$$ 
If $m > 0$, by the very good property for the quotient stack  $G \backslash X$ we have the following:
$$
 2\dimn X - \dimn G > \dimn X - \dimn G + m + \dimn p(I)  = \dimn I,
$$
which is impossible by our previous estimate from below.  In that case $m = 0$, and $\dimn I = 2\dimn X - \dimn G$.  It follows that $\mu^{-1}(\theta)$ is an equidimensional variety of dimension $2\dimn X - \dimn G$.  

Let $Z \subset X$ be an irreducible component of $X$.  Since $G \backslash X$ is very good, then $X_0$ intersects $Z$.  Moreover, $X_0$ is open in $X$, so $Y := Z \cap X_0$ is irreducible and open.

  We have that $p^{-1}(Y)$ is irreducible in $\mu^{-1}(\theta)$, since $Y$ is irreducible and the fibers of $p$ are isomorphic to $\mathbb{C}^{\dimn X}$.  It follows that $p^{-1}(Y)$ must be contained entirely in some irreducible component of $\mu^{-1}(\theta)$.  
	
	This means there is a correspondence between the irreducible components of $X$ and the irreducible components of $\mu^{-1}(\theta)$.  Since $X$ is smooth, its irreducible components are disjoint, and therefore the correspondence is injective.  It is also surjective, because the above computation implies $p^{-1}(X_0)$ intersects each irreducible component of $\mu^{-1}(\theta)$.
\end{proof}

We immediately obtain the following corollary:
\begin{cor}
\label{2.4.3}
If $X$ is irreducible and the quotient stack $G \backslash X$ is very good, then $\mu^{-1}(\theta)$ is a nonempty, irreducible, variety of dimension $2 \dimn X - \dimn G$.
\end{cor}

\begin{rmk}
\label{2.4.4}
If we assume the quotient stack $G \backslash X$ merely to be good, then the result that $\mu^{-1}(\theta)$ is an equidimensional variety of dimension $2\dimn X - \dimn G$ still holds.
\end{rmk}

\begin{rmk}
\label{2.4.5}
Note that even if $G$ is not assumed to be semisimple, then Lemma \ref{2.4.1} still holds.  Let $X_s = \{x \in X| \dimn G_x = s \}$.  If the quotient stack $G \backslash X$ is only almost very good for a given $m$, then Theorem \ref{2.4.2} and Remark \ref{2.4.4} still hold, as long as $\mu(\pi^{-1}(X_m))$ contains $\theta$, with the exception that 
$$
\dimn \mu^{-1}(\theta) = 2\dimn X - \dimn G + m.
$$ 
\end{rmk}

%% file: quivers_and_their_representations.tex
\subsection{Outline}
\label{3.0}
Before proceeding with the proof of Theorem 1.2.1, we will consider the very good property for the quotient stack of quiver representations (in coordinate spaces) by the change of basis action at each vertex.  This example is related to the special case of Theorem 1.2.1, when the vector bundle underlying the parabolic bundles is trivial.  Specifically, Theorem 1.2.1 reduces to showing the very good property for $\textrm{PGL}(\alpha_0, \mathbb{C}) \backslash Fl(\alpha)$, which follows from the very good property for the quotient stack of representations of a certain quiver (See Section 6.4 for details).  We will largely follow the arguments outlined in Section 6 of \cite{CB1993} and Sections 1-4 of \cite{CB2001}, since his results imply ours.  We commence by setting up conditions that make the quotient stack of quiver representations very good.

Let $\textrm{Rep}(Q,\alpha)$ be the vector space of representations of the finite, loop-free quiver $Q$ in the standard coordinate spaces over an algebraically closed field $K$.  The dimensions of these coordinate spaces can be encoded as the \textit{dimension vector} $\alpha =  (\alpha_i)_{i \in I_Q}$, where $I_Q$ is the set of vertices for $Q$.  The group $G(\alpha) = \prod_{i \in I_Q} \textrm{GL}(\alpha_i, K)/K^*$ acts on $\textrm{Rep}(Q,\alpha)$ by change of basis at each vertex $i \in I_Q$.  

We define the \textit{Tits quadratic form} on the dimension vectors by 
$$
q(\alpha) = \sum \alpha_i^2  - \sum_{a: i \rightarrow j} \alpha_i\alpha_j,
$$
where the latter sum is taken over all arrows in $Q$ connecting vertex $i$ with vertex $j$.  We set $p(\alpha) = 1 - q(\alpha)$ following \cite{CB1993}.  

We say a dimension vector $\alpha$ is in the \textit{fundamental region} if it is nonzero, has connected support, and satisfies the following inequalities:
$$
2\alpha_{i} - \sum_{a: i \rightarrow j} \alpha_j - \sum_{a: l \rightarrow i} \alpha_l \le 0 \quad \forall i \in I_Q, 
$$
where the sums are taken over all arrows going into $i$ and coming out of $i$.  The main theorem we wish to prove in this section is:
\begin{thm}
\label{3.0.1}
Suppose $\alpha$ is in the fundamental region and $p(\alpha)>\sum_i p(\beta^{(i)})$ for any decomposition $\alpha = \sum_i \beta^{(i)}$ into the sum of two or more dimension vectors, then the quotient stack $G(\alpha) \backslash \textrm{Rep}(Q, \alpha)$ is very good.
\end{thm}
  
Note that in the statement of the theorem it suffices for $\beta^{(i)}$ to be roots of the Kac-Moody algebra associated with $Q$.  We will prove this theorem following an argument of Crawley-Boevey, outlined in \cite{CB2001}.  A key computation in this argument relies on the following theorem from \cite{Kac1980} and \cite{Kac1982}:
\begin{Kac}
Let $\alpha$ be a dimension vector for representations of a quiver $Q$, and let $\textrm{Ind}(Q, \alpha)$ consist of indecomposable representations of $Q$ with dimension vector $\alpha$.  We have that  $\textrm{Ind}(Q, \alpha)$ is nonempty if and only if $\alpha$ is a root of the Kac-Moody algebra associated with $Q$.  In this case, $\dimn_{G(\alpha)} \textrm{Ind}(Q, \alpha) = p(\alpha)$.
\end{Kac}
The specific quiver $Q^{st}$ related to $Fl(\alpha)$ is called the \textit{star-shaped quiver} (see Section \ref{3.4} below).  The dimension vectors of this quiver have the form $\alpha = (\alpha_0, \alpha_{ij})$, where $1 \le i \le k$ and $1 \le j \le w_i-1$.  Let $\delta(\alpha) = -2\alpha_0 + \sum_j \alpha_{ij}$.    If we consider only representations of the star-shaped quiver, we can replace the condition on $p(\alpha)$ in the statement of Theorem \ref{3.0.1} by the weaker $\delta(\alpha) > 0$, which is easier to check.  The reduction is accomplished through the following key proposition:

\begin{prp}
\label{3.0.2}
Suppose $\delta(\alpha) > 0$ and $\alpha$ is in the fundamental region, then $p(\alpha) > \sum_i p(\beta^{(i)})$, for any decomposition $\alpha = \sum_i \beta^{(i)}$ into the sum of two or more vectors in $\mathbb{Z}_{\ge 0}^{I_{Q^{st}}}$.
\end{prp}

\subsection{General Definitions}
\label{3.1}
Let us recall some definitions from the theory of quiver representations.  
\begin{defn}
\label{3.1.1}
A \textit{quiver} is a directed multigraph.  For a quiver $Q$, we have:
\begin{itemize}
\item $I_Q$ is the vertex set of $Q$.
\item $A_Q$ is the arrow set of $Q$.
\item $h: A_Q \rightarrow I_Q$ is a map that sends an arrow to its head vertex.
\item $t: A_Q \rightarrow I_Q$ is a map that sends an arrow to its tail vertex.
\end{itemize}  
We call $Q$ finite if $I_Q$ and $A_Q$ are both finite.  We call $Q$ loop-free if every $a \in A_Q$ has distinct head and tail vertices.
\end{defn}
Let $K$ be an algebraically closed field.  A \textit{representation} of $Q$ consists of a collection of $K$-vector spaces $V_i$ indexed by $i \in I_Q$, together with a family of linear maps $f_a: V_{t(a)} \rightarrow V_{h(a)}$ for each $a \in A_Q$.

Given two quiver representations $V = (V_i,f_a)$ and $W = (W_i, g_a)$, a \textit{morphism} from $V$ to $W$ is a family of linear maps $\phi = (\phi_i)_{i\in I_Q}$ such that $\phi_i: V_i \rightarrow W_i $ and $\phi_{h(a)} \circ f_a = g_a \circ \phi_{t(a)}$.

Let $R(Q)$ denote the \textit{path algebra} corresponding to the quiver $Q$.  That is, $R(Q)$ is the associative $K$-algebra generated by the paths of $Q$, where multiplication is given by the concatenation of paths.  Therefore, we can interpret quiver representations as modules over $R(Q)$.  Representations of the quiver $Q$ form an abelian category $R(Q)-\mathbf{Mod}$.

We wish to restrict ourselves to representations of $Q$ where the spaces $V_i$ are all standard coordinate spaces.
\begin{defn}
\label{3.1.2}
Let $Q$ be a quiver and let $K$ be an algebraically closed field.  A representation of $Q$ in coordinate spaces, with dimension vector $\alpha = (\alpha_i)_{i\in I_Q}\in \mathbb{Z}^{I_Q}_{\ge 0}$, is an element of the following vector space:
$$
\textrm{Rep}(Q,\alpha) = \bigoplus_{a \in A_Q} \textrm{Mat}(\alpha_{h(a)} \times \alpha_{t(a)}, K).
$$
\end{defn}
From now on, we will only consider representations of $Q$ in coordinate spaces. Let $\textrm{Mat}(\alpha)_x$ be the algebra of endomorphisms of the module $K_x$ corresponding to $x \in \textrm{Rep}(Q, \alpha)$.  This is a subalgebra of $\textrm{Mat}(\alpha) = \prod_{i \in I_Q} \textrm{Mat}(\alpha_i, K)$, consisting of collections of matrices $g = (g_i)$ such that $g_{h(a)} \circ a = a \circ g_{t(a)}$ for all $a \in A_Q$.

\begin{defn}
\label{3.1.3}
Let $Q$ be finite.  We define the \textit{Euler-Ringel form} on the dimension vectors $\alpha$ in the following way:
$$
\langle \alpha, \beta \rangle = \sum_{i\in I_Q} \alpha_i \beta_i - \sum_{a \in A_Q} \alpha_{t(a)}\beta_{h(a)}.
$$
\begin{itemize}
\item $q(\alpha) = \langle \alpha, \alpha \rangle$ is called the Tits quadratic form.
\item We set $p(\alpha) = 1 - q(\alpha)$.
\end{itemize}
\end{defn}

Let $G(\alpha) = \prod_{i \in I_Q} \textrm{GL}(\alpha_i, K)/K^*$.  The group $G(\alpha)$ acts on $\textrm{Rep}(Q,\alpha)$ by conjugation.  In the next section, we will consider the very good property for $\textrm{Rep}(Q,\alpha)$ with respect to this action.  Note that it is easy to check that $\textrm{dim } \textrm{Rep}(Q, \alpha)  = \textrm{dim } G(\alpha) + p(\alpha)$.

\subsection{Quivers and Kac-Moody algebras}
\label{3.2}
We wish to introduce several concepts from the theory of Kac-Moody Lie algebras and relate these concepts to the theory of quiver representations.  For details concerning Kac-Moody algebras and associated concepts, see \cite{Kac1990}.  

\begin{defn}
\label{3.2.1}
Let $A = (a_{ij})_{i,j=1}^n$ be a complex $n \times n$ matrix of rank $l$.  We say that $A$ is a generalized Cartan matrix if the following holds:
\begin{itemize}
\item[(1)] $a_{ii} = 2$,
\item[(2)] $a_{ij}$ are nonpositive integers for $i \neq j$,
\item[(3)] $a_{ij} = 0$ implies $a_{ji} = 0$.
\end{itemize}
\end{defn}

Consider a triple $(\mathfrak{h}, \Pi, \Pi^{\vee})$, where $\mathfrak{h}$ is a complex vector space, $\Pi = (\epsilon_1, \dots, \epsilon_n) \subset \mathfrak{h}^*$ and $\Pi^{\vee} = (\epsilon_1^{\vee}, \dots, \epsilon_n^{\vee}) \subset \mathfrak{h}$ are indexed subsets.  We call the triple $(\mathfrak{h}, \Pi, \Pi^{\vee})$ a \textit{realization} of $A$ if the following three conditions hold:
\begin{itemize}
\item[(1)] the sets $\Pi$ and $\Pi^{\vee}$ are linearly independent,
\item[(2)] $\epsilon_j(\epsilon_i^{\vee}) = a_{ij}$,
\item[(3)] $\textrm{dim } \mathfrak{h} = 2n - l$.
\end{itemize}

\begin{defn}
\label{3.2.2}
Let $(\mathfrak{h}, \Pi, \Pi^{\vee})$ be a realization of a generalized Cartan matrix $A$.  Consider the Lie algebra $\mathfrak{g}(A)$ generated by $e_i, f_i$ (for $i=1, \dots, n$) and $\mathfrak{h}$.  We call $\mathfrak{g}(A)$ a Kac-Moody algebra if it satisfies the following relations:
\begin{itemize}
\item[(1)] $[e_i, f_j] = \delta_{ij}\epsilon_i^{\vee}$ for all $i,j$, 
\item[(2)] $[h,h'] = 0$ for all $h,h' \in \mathfrak{h}$,
\item[(3)] $[h, e_i] = \epsilon_i(h)e_i$,
\item[(4)] $[h, f_i] = -\epsilon_i(h)f_i$,
\item[(5)] $(\textrm{ad } e_i)^{1-a_{ij}}(e_j) = 0$ for $i \neq j$,
\item[(6)] $(\textrm{ad } f_i)^{1-a_{ij}}(f_j) = 0$ for $i \neq j$,
\end{itemize}
where $\textrm{ad}(x)(y) = [x,y]$.
\end{defn}

Two realizations $(\mathfrak{h}_1, \Pi_1, \Pi^{\vee})_1$ and $(\mathfrak{h}_2, \Pi_2, \Pi^{\vee}_2)$ of $A$ are said to be \textit{isomorphic} if there exists a vector space isomorphism $\phi(\mathfrak{h}_1) = \mathfrak{h}_2$ such that $\phi^*(\Pi_1) = \Pi_2$ and $\phi(\Pi_1^{\vee}) = \Pi_2^{\vee}$.  It is clear that any two isomorphic realizations define the same algebra $\mathfrak{g}(A)$.  By Proposition 1.1 in (\cite{Kac1990}), any two realizations of a matrix $A$ are isomorphic.

Let $S = \sum_{i=1}^n \mathbb{Z}\epsilon_i$.  We define the \textit{root space} attached to $\epsilon \in S$ as the following vector space:
$$
\mathfrak{g}_{\epsilon} = \{x \in \mathfrak{g}(A)|[h,x] = \epsilon(h)x \}.
$$
The elements $\epsilon_i$ are called the \textit{simple roots} of $\mathfrak{g}(A)$.
\begin{defn}
\label{3.2.3}
Let $\epsilon$ be in $S$.  We say $\epsilon$ is a root if $\epsilon \neq 0$ and $\dimn \mathfrak{g}_{\epsilon} \neq 0$.
\begin{itemize}
\item A root $\epsilon$ is called positive if it has all positive coefficients in $S$.
\item A root $\epsilon$ is called negative if it has all negative coefficients in $S$.
\end{itemize}
We denote by $\Delta, \Delta_{+}, \Delta_{-}$ the sets of all roots, of positive roots, and of negative roots, respectively.
\end{defn}

As in the theory of semisimple Lie algebras, one can define a Weyl group $W(A)$ of $\mathfrak{g}(A)$, generated by reflections.  For further details of this construction see Remark \ref{3.2.5} below.

Given a finite, loop-free quiver $Q$, we will construct a Kac-Moody algebra associated with $Q$.  Indeed, by symmetrizing the Euler-Ringel form we obtain:
$$
(\epsilon_i, \epsilon_j) =
\begin{cases}
2 & \textrm{if } i=j\\
-(a_{ij} + a_{ji}) & \textrm{otherwise},
\end{cases}
$$
where $\epsilon_i$ is a standard basis vector in $\mathbb{Z}^{I_Q}$ corresponding to the vertex $i \in I_Q$ and $a_{ij}$ is the number of arrows from vertex $i$ to vertex $j$ in $Q$. From this we can clearly see that the symmetrized Euler-Ringel form defines a generalized Cartan matrix $A$.  

Let $S = \sum_{i \in I_Q} \mathbb{Z}\epsilon_i$ and let $\mathfrak{h}' = \mathbb{C} \otimes S$.  Note that we can extend the symmetric bilinear form defined by $A$ to $\mathfrak{h}'$.  This form may be degenerate, so let $\mathfrak{c}$ be its kernel.  Now, set $\mathfrak{h} = \mathfrak{h}' \oplus \mathfrak{c}^*$.  Fix a complement $\mathfrak{h}''$ to $\mathfrak{c}$ in $\mathfrak{h}'$.  We extend the symmetric bilinear form on $\mathfrak{h}'$ to $\mathfrak{h}$ in the following way:
\begin{align*}
& (c,d) = c(d) \textrm{ for } c \in \mathfrak{c}^*, d \in \mathfrak{c}\\
& (c_1, c_2) = 0 \textrm{ for } c_1, c_2 \in \mathfrak{c}^*\\
& (c, h) = 0 \textrm{ for } c \in \mathfrak{c}^*, h \in \mathfrak{h}''.
\end{align*} 
We can see that this defines a nondegenerate symmetric bilinear form on $\mathfrak{h}$.  Therefore, we have a canonical isomorphism between $\mathfrak{h}$ and $\mathfrak{h}^*$.  This gives us a realization of $A$ where $\mathfrak{h}$ is as defined above, $\Pi = (\epsilon_i)_{i=1}^n$, and $\Pi^{\vee} = (\epsilon_i^{\vee})_{i=1}^n$, where $\epsilon_i^{\vee}$ is identified with $\epsilon_i$ under the bilinear form on $\mathfrak{h}$.  The Kac-Moody algebra $\mathfrak{g}(A)$ associated with this realization is the Kac-Moody algebra associated with the quiver $Q$.  Note that a different choice for the complement $\mathfrak{h}''$ defines a realization of $A$ that is isomorphic to the given one.

We can see that the standard basis vectors $\epsilon_i$ are the simple roots of the Kac-Moody algebra associated with $Q$.  Since dimension vectors are in $S$, it makes sense to consider certain dimension vectors as roots of this algebra.  By analogy with roots, we say that a dimension vector $\alpha$ is \textit{positive} (respectively: \textit{negative}, \textit{nonnegative}) if the coefficients it has in $S$ are positive (respectively: negative, nonnegative).

\begin{defn}
\label{3.2.4}
  The \textit{fundamental region} is the set of positive dimension vectors $F = \{\alpha \in \mathbb{Z}^{I_Q}_{> 0}| (\alpha, \epsilon_i) \le 0 \ \forall i \in I_Q\}$ with connected support. 
\end{defn}

\begin{rmk}
\label{3.2.5}
A \textit{reflection} at a vertex $i \in I_Q$ is defined as $s_i(\alpha) = \alpha - (\alpha,\epsilon_i)\epsilon_i$.  Note that, the Weyl group $W$ of the associated Kac-Moody algebra is generated by these reflections.  The real roots are the images of the coordinate vectors under elements of $W$.  The fundamental region consists of integer points of $-C^{\vee}$, where $C^{\vee}$ is the dual of the fundamental chamber of the Weyl group. The imaginary roots are the images of vectors in $-C^{\vee}$ under the action of $W$. 
\end{rmk}

\subsection{The very good property for quiver representations}
\label{3.3}
The contents of this section largely follow Crawley-Boevey in \cite{CB1993} and \cite{CB2001}.  Let $Q$ be a finite loop-free quiver, and fix $\alpha \in \mathbb{Z}^{I_{Q}}_{\ge 0}$.

Let $\textrm{Ind}(Q, \beta^{(1)}, \dots, \beta^{(l)})$ be the $G(\alpha)$-stable constructible set consisting of all quiver representations that can be written as the sum of indecomposable representations of dimension types $\beta^{(1)}, \dots, \beta^{(l)}$, where $\alpha = \sum_i\beta^{(i)}$.  Since $\textrm{Rep}(Q, \alpha)$ is the union of all the $\textrm{Ind}(Q, \beta^{(1)}, \dots, \beta^{(l)})$, the following lemma holds.

\begin{lmm}
\label{3.3.1}
We have $\dimn_{G(\alpha)}\textrm{Rep}(Q, \alpha)= \textrm{max }\{ \dimn_{G(\alpha)}\textrm{Ind}(Q, \beta^{(1)},..,\beta^{(l)})\}$, where the maximum is taken over all decompositions into indecomposables of dimensions $\beta^{(1)}, \dots, \beta^{(l)}$.
\end{lmm}

Note that by the Kac Theorem the dimension vectors $\beta^{(1)} \dots \beta^{(l)}$ are actually positive roots of the Kac-Moody algebra corresponding to $Q$.  We can now prove the following result (see Lemma 4.3 in \cite{CB1993}):

\begin{thm}
\label{3.3.2}
Let one of the following hold:
\begin{enumerate}
\item{The maximum in Lemma \ref{3.3.1} is achieved for $l=1$.}
\item{The maximum in Lemma \ref{3.3.1} is achieved for $l\ge 2$, and for the corresponding
collection $\beta^{(1)},...,\beta^{(l)}$ we have $p(\alpha)>\sum_i p(\beta^{(i)})$.}
\end{enumerate}
Then the stack $G(\alpha) \backslash \textrm{Rep}(Q, \alpha)$ is very good.
\end{thm}
\begin{proof}
The case when $l = 1$ is considered in Theorem \ref{3.3.3} below.  Assume the second case holds.  By Lemma \ref{3.3.1},  $\textrm{dim}_{G(\alpha)} \textrm{Rep}(Q, \alpha) = \textrm{dim}_{G(\alpha)} \textrm{Ind}(Q, \beta^{(1)}, \dots, \beta^{(l)})$, for some $\beta^{(1)}, \dots, \beta^{(l)}$ with $l \ge 2$.  Consider $S = \textrm{Rep}(Q,\beta^{(1)}) \times \cdots \times \textrm{Rep}(Q, \beta^{(l)})$ included in $\textrm{Rep}(Q, \alpha)$ as block diagonal matrices.  Let $J \subseteq S$ consist of elements such that the projection onto each $\textrm{Rep}(Q,\beta^{(i)})$ is indecomposable.  Note that $J$ is constructible.  Now we have $$\textrm{dim}_{G(\alpha)} \textrm{Ind}(Q, \beta^{(1)}, \dots, \beta^{(l)}) = \textrm{dim}_H J  = \sum_i \textrm{dim}_{G(\beta^{(i)})} \textrm{Ind}(Q, \beta^{(i)}) = \sum_i p(\beta^{(i)}).$$  The first equality follows from Lemma \ref{2.3.2} for $G = G(\alpha)$ and $H = \prod_i G(\beta^{(i)})$ (using the Krull-Schmidt Theorem), the second follows from Lemma \ref{2.3.3}, and the third follows from Kac's Theorem.  Now, $\textrm{dim } \textrm{Rep}(Q, \alpha)  = \textrm{dim } G(\alpha) +  p(\alpha)$ and Corollary \ref{2.2.3} imply that the very good condition on $\textrm{Rep}(Q, \alpha)$ holds if $\textrm{dim}_{G(\alpha)} \textrm{Rep}(Q, \alpha) < p(\alpha)$. This, however, is clearly true since $\textrm{dim}_{G(\alpha)} \textrm{Rep}(Q, \alpha) = \sum_i p(\beta^{(i)})$, for the decomposition $\alpha = \sum_i \beta^{(i)}$.
\end{proof}

Note that the inequality on $p(\alpha) > \sum_i p(\beta^{(i)})$ fails for $l = 1$, so it remains to handle this case.  The proof follows Section 6 in \cite{CB1993}.

\begin{thm}
\label{3.3.3}
If $\alpha$ is in the fundamental region, then the space of indecomposable quiver representations, $\textrm{Ind}(Q, \alpha)$, is very good.
\end{thm}

\begin{proof}
Let $N$ consist of nontrivial nilpotents (tuples of nilpotent elements) in the product $\textrm{Mat}(\alpha) = \prod_{i \in I_Q} \textrm{Mat}(\alpha_i, K)$, and let $MN = \{(x,g) \in \textrm{Rep}(Q, \alpha) \times N| g \in \textrm{Mat}(\alpha)_x \}$.
We assume that there exists some $x \in \textrm{Ind}(Q, \alpha)$ with a nontrivial element in its stabilizer, otherwise, $\textrm{Ind}(Q, \alpha)$ is clearly very good.

Consider $\lambda = (\lambda_{i})$, where $\lambda_{i} = (\lambda_{i}^1, \lambda_{i}^2, \dots)$ is a partition of $\alpha_{i}$.  An element $g \in N$ has type $\lambda$ if $\lambda_{i}^r$ is the number of Jordan blocks of size $r$ or larger in the Jordan form of the $i$-th component of $g$.  That is, if $g_{i}$ has type $\lambda_{i}$.

Let $N_{\lambda}$ consist of elements in $N$ of type $\lambda$, and $\textrm{Mod}_{g} = \{ x \in \textrm{Rep}(Q, \alpha)| g \in \textrm{Mat}(\alpha)_x \}$.  We can compute:
$$
\textrm{dim Mod}_g  = \sum_{a \in A_Q} \sum_r \lambda_{h(a)}^r\lambda_{t(a)}^r,
$$ 
and since $N_{\lambda}$ consists of elements with fixed Jordan form, we have:
$$
\textrm{dim } N_{\lambda} = \textrm{dim } G(\alpha) - \textrm{dim }\{ h \in G(\alpha)| hg = gh \} = \textrm{dim } G(\alpha) + 1 - \sum_i \sum_r \lambda_i^r \lambda_i^r.
$$
Both computations come from the formula $\textrm{dim }\{h| g h = h f\} = \sum_r \lambda^r \mu^r$, where $h: V \rightarrow W$, $g \in \textrm{End}(V)$ has type $\lambda = (\lambda^r)$ and $f \in \textrm{End}(W)$ has type $\mu = (\mu^r)$.  The formula follows because $h$ must take invariant subspaces of $g$ to invariant subspaces of $h$ and vice versa.

Consider the natural projection $$p: MN_{\lambda} \rightarrow N_{\lambda},$$  where $(x,g) \in MN_{\lambda}$ implies $g$ has type $\lambda$ .  Since $p^{-1}(g) = \textrm{Mod}_g$, we can compute:
$$
\textrm{dim } MN_{\lambda}  = \textrm{dim } N_ {\lambda} + \textrm{dim Mod}_g = \textrm{dim } G(\alpha) + 1 - \sum_r q(\lambda^r) < \textrm{dim } G(\alpha) + p(\alpha).
$$ 
The last inequality follows from Proposition \ref{3.0.2} (which will be proved later, independent of Theorem \ref{3.3.3}), since $\alpha = \sum_r \lambda^r$ is in the fundamental region, with at least two nonzero $\lambda^r$.  It follows that $\textrm{dim } MN < \textrm{dim }G(\alpha) + p (\alpha)$.

Note that quiver representations correspond to modules over the path algebra of the quiver.  Therefore, $x \in \textrm{Ind}(Q, \alpha)$ corresponds to some indecomposable module $K_x$, and we have $\dimn G(\alpha)_x = \dimn \textrm{Aut}(K_x) - 1 = \dimn \textrm{Mat}(\alpha)_x - 1$.  Note that $\textrm{End}(K_x) \cong \textrm{Mat}(\alpha)_x$, so it follows from Fitting's lemma that $\textrm{End}(K_x)$ can be presented as the direct sum $\mathbb{C} \oplus N_x$, where $N_x$ consists of nilpotents.  Consider the natural surjective projection: 
$$\pi: I_{(m)}N \rightarrow I_{(m)}, $$
where $I_{(m)} = \{x \in \textrm{Ind}(Q, \alpha)|\textrm{dim Mat}(\alpha)_x = m\}$ and $I_{(m)}N = \{(x,g) \in I_{(m)} \times N| g \in N_x \}$.  We have that $\textrm{dim } \pi^{-1}(x) = m - 1$ since the fiber consists of nilpotent elements that stabilize $x \in I_{(m)}$.  Therefore, we can compute:
$$
\textrm{dim }I_{(m)} = \textrm{dim }I_{(m)}N - m + 1 \le \textrm{dim } MN - m + 1 < \textrm{dim }G(\alpha) + p(\alpha) - (m-1).
$$
It follows that $\textrm{Ind}(Q, \alpha)$ is very good.
\end{proof}

\begin{proof}[Proof of Theorem \ref{3.0.1}:]
The theorem follows from Theorem \ref{3.3.2} and Theorem \ref{3.3.3}.
\end{proof}

\subsection{Squids and Star-shaped Quivers}
\label{3.4}
Let $D = (x_1, \dots, x_k)$ be a collection of points of $\mathbb{P}^1$, and let $w = (w_1, \dots, w_k)$ be a collection of positive integers. Consider the following quiver $Q_{D,w}$:

\setlength{\unitlength}{2.0pt}
\[
\begin{picture}(150,80)
\put(-20,45){\circle*{2.5}}
\put(6,45){\circle*{3.0}}

\put(40,10){\circle*{3.0}}
\put(70,10){\circle*{3.0}}
\put(140,10){\circle*{3.0}}

\put(40,45){\circle*{3.0}}
\put(70,45){\circle*{3.0}}
\put(140,45){\circle*{3.0}}

\put(40,70){\circle*{3.0}}
\put(70,70){\circle*{3.0}}
\put(140,70){\circle*{3.0}}
\put(5,47){\vector(-1,0){23}}
\put(5,43){\vector(-1,0){23}}
\put(40,70){\vector(-4,-3){32}}
\put(40,45){\vector(-4,0){32}}
\put(40,10){\vector(-1,1){33}}
\put(70,10){\vector(-1,0){28}}
\put(70,45){\vector(-2,0){28}}
\put(70,70){\vector(-1,0){28}}
\put(96,10){\vector(-1,0){24}}
\put(96,45){\vector(-1,0){24}}
\put(96,70){\vector(-1,0){24}}
\put(140,10){\vector(-1,0){26}}
\put(140,45){\vector(-1,0){26}}
\put(140,70){\vector(-1,0){26}}
\put(100,10){\circle*{1}}
\put(105,10){\circle*{1}}
\put(110,10){\circle*{1}}
\put(100,45){\circle*{1}}
\put(105,45){\circle*{1}}
\put(110,45){\circle*{1}}
\put(100,70){\circle*{1}}
\put(105,70){\circle*{1}}
\put(110,70){\circle*{1}}
\put(40,20){\circle*{1}}
\put(40,27){\circle*{1}}
\put(40,34){\circle*{1}}
\put(70,20){\circle*{1}}
\put(70,27){\circle*{1}}
\put(70,34){\circle*{1}}
\put(140,20){\circle*{1}}
\put(140,27){\circle*{1}}
\put(140,34){\circle*{1}}
\put(5,48){0}
\put(-23,40){$\infty$}
\put(-7,49){$b_0$}
\put(-7,38){$b_1$}
\put(20,62){$c_{11}$}
\put(17,23){$c_{k1}$}
\put(35,2){$[k,1]$}
\put(35,50){$[2,1]$}
\put(35,74){$[1,1]$}
\put(66,2){$[k,2]$}
\put(66,50){$[2,2]$}
\put(66,74){$[1,2]$}
\put(132,2){$[k,w_k-1]$}
\put(132,50){$[2,w_2-1]$}
\put(132,74){$[1,w_1-1]$}
\put(52,2){$c_{k2}$}
\put(24,50){$c_{21}$}
\put(52,50){$c_{22}$}
\put(52,74){$c_{12}$}
\end{picture}
\]
Recall that $R(Q_{D,w})$ denotes the path algebra corresponding to the above quiver. 
\begin{defn}
\label{3.4.1}
A \textit{squid} (see e.g. \cite{CB2004}) is the following algebra:
$$
S_{D,w} = R(Q_{D,w})/\{(\lambda_{i0}b_0 + \lambda_{i1}b_1)c_{i1}\},
$$
where $x_i = (\lambda_{i0}:\lambda_{i1})$.  
\begin{itemize}
\item The part of $Q_{D,w}$ consisting of the vertices $\{0, \infty\}$ and the arrows $\{b_0,b_1\}$ is called the \textit{Kronecker quiver}.
\item The quiver $Q^{st}_{D,w}$ with vertex set $I_{Q_{D,w}} - \{\infty\}$ and arrow set $A_{Q_{D,w}} - \{ b_0,b_1\}$ is called a \textit{star-shaped quiver}.
\end{itemize}
Note that we can identify representations of a star-shaped quiver with representations of the corresponding $Q_{D,w}$ that have $\alpha_{\infty} = 0$.
\end{defn} 
A representation of the Kronecker quiver is called \textit{preinjective} if $\lambda_0b_0 + \lambda_1b_1$ is surjective for all $(\lambda_0:\lambda_1) \in \mathbb{P}^1$.  A representation of $Q_{D,w}$ is called \textit{Kronecker-preinjective}, if the corresponding Kronecker quiver representation is preinjective. 

\subsection{The cotangent bundle for squids}
\label{3.5}
Recall the quiver $Q_{D,w}$ that was introduced in Section \ref{3.4}.  The cotangent bundle $T^*\textrm{Rep}(Q_{D,w},\alpha)$ to the space of representations $\textrm{Rep}(Q_{D,w},\alpha)$ may be identified with the space of representations of the quiver $\overline{Q}_{D,w}$ pictured below.

\begin{tikzpicture}[scale=0.9, every node/.style={scale=0.9}]

\draw[fill] (0,5) circle [radius=0.1];
\node [below] (inf2) at (0,5) {};
\node [above] (inf1) at (0,5) {};
\node [left] (inf) at (0,5) {$\infty$};

\draw[fill](2,5) circle [radius=0.1];
\node [below] (02) at (2,5) {};
\node [above] (01) at (2,5) {};
\node [right] (0) at (2,5) {$0$};

\draw [thick, <-] (inf1)  -- node [above] {$b_0$} (01);
\draw [thick, ->] (inf1)  edge[bend left=50] node[above]{$\hat{b}_0$}(01);
\draw [thick, <-] (inf2) -- node [below] {$b_1$}(02);
\draw [thick, ->] (inf2)  edge[bend right=50] node[below]{$\hat{b}_1$} (02);

\draw[fill](4,8) circle [radius=0.1];
\node [above] (11a) at (4,8) {$[1,1]$};
\node [left] (11l) at (4,8) {};
\node [below] (11b) at (4,8) {};
\node [right] (11r) at (4,8) {};
\draw [thick, <-] (01) -- (11l);
\node at (2.85,7){$c_{11}$};
\draw [thick, ->] (2.25,5.2) -- (11b);
\node at (3.75,7){$\hat{c}_{11}$};

\draw[fill](6,8) circle [radius=0.1];
\node [above] (12a) at (6,8) {$[1,2]$};
\node [below] (12b) at (6,8) {};
\draw [thick, ->] (12a) -- node[above]{$c_{12}$}(11a);
\draw [thick, <-] (12b) -- node[below]{$\hat{c}_{12}$}(11b);

\draw[opacity=0, fill](8,8) circle [radius=0.1];
\node [opacity=0, above] (13a) at (8,8) {$[1,3]$};
\node [below] (13b) at (8,8) {};
\draw [thick, ->] (13a) -- (12a);
\draw [thick, <-] (13b) -- (12b);

\draw[fill](8.5,8) circle [radius=0.05];
\draw[fill](9,8) circle [radius=0.05];
\draw[fill](9.5,8) circle [radius=0.05];

\draw[opacity=0, fill](10,8) circle [radius=0.1];
\node [opacity=0, above] (1w1ma) at (10,8) {$[]$};
\node [below] (1w1mb) at (10,8) {};

\draw[fill](12,8) circle [radius=0.1];
\node [above] (1w1m1) at (12,8) {$[1,w_1-1]$};
\node [below] (1w1m2) at (12,8) {};
\draw [thick, ->] (1w1m1) -- (1w1ma);
\draw [thick, <-] (1w1m2) -- (1w1mb);

\draw[fill](4,6) circle [radius=0.1];
\node[above] (21a) at (4,6) {$[2,1]$};
\node[below] (21b) at (4,6) {};
\draw [thick, ->] (3.8,6) -- node[above]{$c_{21}$}(2.3,5.1);
\draw [thick, <-] (21b) -- node[below]{$\hat{c}_{21}$}(2.4,5);

\draw[fill](6,6) circle [radius=0.1];
\node [above] (22a) at (6,6) {$[2,2]$};
\node [below] (22b) at (6,6) {};
\draw [thick, ->] (22a) -- node[above]{$c_{22}$}(21a);
\draw [thick, <-] (22b) -- node[below]{$\hat{c}_{22}$}(21b);

\draw[opacity=0, fill](8,6) circle [radius=0.1];
\node [opacity=0, above] (23a) at (8,6) {$[2,3]$};
\node [below] (23b) at (8,6) {};
\draw [thick, ->] (23a) -- (22a);
\draw [thick, <-] (23b) -- (22b);

\draw[fill](8.5,6) circle [radius=0.05];
\draw[fill](9,6) circle [radius=0.05];
\draw[fill](9.5,6) circle [radius=0.05];

\draw[opacity=0, fill](10,6) circle [radius=0.1];
\node [opacity=0, above] (2w2ma) at (10,6) {$[]$};
\node [below] (2w2mb) at (10,6) {};

\draw[fill](12,6) circle [radius=0.1];
\node [above] (2w2m1) at (12,6) {$[2,w_2-1]$};
\node [below] (2w2m2) at (12,6) {};
\draw [thick, ->] (2w2m1) -- (2w2ma);
\draw [thick, <-] (2w2m2) -- (2w2mb);

\draw[fill](4,5) circle [radius=0.05];
\draw[fill](4,4) circle [radius=0.05];
\draw[fill](4,3) circle [radius=0.05];

\draw[fill](6,5) circle [radius=0.05];
\draw[fill](6,4) circle [radius=0.05];
\draw[fill](6,3) circle [radius=0.05];

\draw[fill](12,5) circle [radius=0.05];
\draw[fill](12,4) circle [radius=0.05];
\draw[fill](12,3) circle [radius=0.05];

\draw[fill](4,2) circle [radius=0.1];
\node[above] (k1a) at (4,2) {};
\node[left] (k1l) at (4,2) {};
\node[below] (k1b) at (4,2) {$[k,1]$};
\draw [thick, ->] (2.3,4.8) -- node[right]{$\hat{c}_{k1}$}(k1a);
\draw [thick, <-] (2.15,4.7) -- node[left]{$c_{k1}$}(k1l);

\draw[fill](6,2) circle [radius=0.1];
\node [below] (k2b) at (6,2) {$[k,2]$};
\node [above] (k2a) at (6,2) {};
\draw [thick, <-] (k2a) -- node[above]{$\hat{c}_{k2}$}(k1a);
\draw [thick, ->] (k2b) -- node[below]{$c_{k2}$}(k1b);

\draw[opacity=0, fill](8,2) circle [radius=0.1];
\node [opacity=0, below] (k3b) at (8,2) {$[k,3]$};
\node [above] (k3a) at (8,2) {};
\draw [thick, <-] (k3a) -- (k2a);
\draw [thick, ->] (k3b) -- (k2b);

\draw[fill](8.5,2) circle [radius=0.05];
\draw[fill](9,2) circle [radius=0.05];
\draw[fill](9.5,2) circle [radius=0.05];

\draw[opacity=0, fill](10,2) circle [radius=0.1];
\node [opacity=0, below] (kwkma) at (10,2) {$[]$};
\node [above] (kwkmb) at (10,2) {};

\draw[fill](12,2) circle [radius=0.1];
\node [below] (kwkm1) at (12,2) {$[k,w_k-1]$};
\node [above] (kwkm2) at (12,2) {};
\draw [thick, ->] (kwkm1) -- (kwkma);
\draw [thick, <-] (kwkm2) -- (kwkmb);
\end{tikzpicture}

Recall from Section \ref{3.4} that a squid representation is a representation of $Q_{D,w}$ such that $(\lambda_{i0}b_0 + \lambda_{i1}b_1)c_{i1} = 0$.  Further recall that $KS(D,w,\alpha)$ is the space of Kronecker-preinjective squid representations, such that the arrows $c_{ij}$ are injective (see Section \ref{3.4} for details).  

Squid representations form a closed subvariety of representations of $Q_{D,w}$.  Therefore, it follows $T^*KS(D,w,\alpha)$ may be identified with the quotient of $\textrm{Rep}(\overline{Q}_{D,w}, \alpha)$ such that:
\begin{itemize}
\item  The maps $\hat{b}_0 \in \textrm{Hom}(\mathbb{C}^{\alpha_{\infty}}, \mathbb{C}^{\alpha_{0}})$ are taken modulo the relations $\lambda_{0i}c_{1i}A_{i} = 0$, where $A_{i}: \mathbb{C}^{\alpha_{\infty}} \rightarrow \mathbb{C}^{\alpha_{ij}}$ are linear maps.
\item The maps $\hat{b}_1 \in \textrm{Hom}(\mathbb{C}^{\alpha_{\infty}}, \mathbb{C}^{\alpha_{0}})$ are taken modulo the relations $\lambda_{1i}c_{1i}A_{i} = 0$.
\item The maps $\hat{c}_{1i} \in \textrm{Hom}(\mathbb{C}^{\alpha_0}, \mathbb{C}^{\alpha_{i1}})$ modulo the relations $A_i(\lambda_{0i}b_0 + \lambda_{1i}b_1) = 0$.
\end{itemize}

Recall from Section \ref{3.1} that the group 
$$G(\alpha) = \textrm{GL}(\alpha_{\infty}, \mathbb{C}) \times \textrm{GL}(\alpha_0, \mathbb{C}) \times \prod \textrm{GL}(\alpha_{ij}, \mathbb{C})/\mathbb{C}^*$$ acts on $\textrm{Rep}(Q_{D,w},\alpha)$ by change of basis.  This action induces a canonical Hamiltonian action of $G(\alpha)$ on $T^*\textrm{Rep}(Q_{D,w},\alpha)$.  Identifying $\textrm{Rep}(\overline{Q}_{D,w},\alpha)$ with its tangent space at a point, the standard symplectic form on $T^*\textrm{Rep}(Q_{D,w},\alpha)$ may be written as: 
$$
\omega(X,X') = \sum_{l=0,1} \textrm{tr}(b_l\hat{b}_l') -\textrm{tr}(b_l'\hat{b}_l)+ \sum_{\substack{1\le i \le k \\  1 \le j \le w_i-1}} \textrm{tr}(c_{ij}\hat{c}_{ij}') - \textrm{tr}(c_{ij}'\hat{c}_{ij}),
$$
where $X = (b_0,b_1,c_{ij},\hat{b}_0, \hat{b}_1, \hat{c}_{ij})$ and $X' = (b_0',b_1',c_{ij}',\hat{b}_0', \hat{b}_1', \hat{c}_{ij}')$ are cotangent vectors.  Recall that 
$$\textrm{Mat}(\alpha) = \textrm{Mat}(\alpha_{\infty}, \mathbb{C}) \times \textrm{Mat}(\alpha_0, \mathbb{C}) \times \prod_{ij} \textrm{Mat}(\alpha_{ij},\mathbb{C}).$$ 
Using the trace pairing, we can identify $\textrm{Lie}(G(\alpha))^*$ with 
$$\textrm{Mat}(\alpha)_0 = \{ (A_i) \in \textrm{Mat}(\alpha)| \sum_i \textrm{tr}(A_i) = 0 \}.$$ 
Note that $KS(D,w,\alpha)$ is invariant under the $G(\alpha)$ action, and the symplectic form defined above descends to the cotangent bundle $T^*KS(D,w,\alpha)$.  Therefore, we can write the corresponding moment map as:
\begin{align*}
& \mu_{G(\alpha)}(X)_{\infty} =  b_0\hat{b}_0 + b_1\hat{b}_1\\
& \mu_{G(\alpha)}(X)_0  = \sum_{1 \le i \le k} c_{i1}\hat{c}_{i1} - (\hat{b}_0b_0 + \hat{b}_1b_1)\\
& \mu_{G(\alpha)}(X)_{ij} = c_{ij+1}\hat{c}_{ij+1} - \hat{c}_{ij}c_{ij} \textrm{ where $1 \le i\le k$ and $1 \le j \le w_i - 1$ },
\end{align*}
at the vertices $\infty$, $0$, and $[i,j]$, respectively.

\subsection{The very good property for star-shaped quivers}
\label{3.6}
We can simplify the statement of Theorem \ref{3.0.1} if the quiver we are considering is a star-shaped quiver $Q^{st}_{D,w}$, described above in Section \ref{3.4}.  The indexing set for the vertices of $Q^{st}_{D,w}$ is $I_{Q^{st}_{D,w}} = \{0\} \cup \{(i,j)|1 \le i \le k, 1 \le j \le w_{i-1} \}$.  This means a dimension vector of a representation of $Q^{st}_{D,w}$ has the form $\alpha = (\alpha_0, \alpha_{ij})$. 

Recall that $\delta(\alpha) = -2\alpha_0 + \sum_j \alpha_{ij}$.  In the case of a star-shaped quiver, the condition that a dimension vector $\alpha$ is in the fundamental region is equivalent to the following inequalities:
\begin{align*}
&\delta(\alpha) \ge 0\\
-2\alpha_{ij} + \alpha_{ij-1} + \alpha_{ij+1} \ge 0, & \textrm{ for } 1 \le i \le k \textrm{ and } 1 \le j \le w_{i-1} 
\end{align*}
(note that we assume $\alpha_{i0} = \alpha_0$, for all $i$).  We wish to prove:

\begin{thm}
\label{3.6.1}
Suppose $\delta(\alpha) > 0$ and $\alpha$ is in the fundamental region, then the quotient stack $G(\alpha) \backslash \textrm{Rep}(Q^{st}_{D,w}, \alpha)$ is very good.
\end{thm}

Recall from Section \ref{3.2} that we can symmetrize the Euler-Ringel form, in order to define a bilinear symmetric form on dimension vectors of quiver representations.  For the quiver $Q^{st}_{D,w}$ this form can be written as:
$$
(\alpha, \beta) = 2\alpha_0\beta_0 - \sum_{i=1}^k \beta_0\alpha_{i1} + \sum_{i=1}^k \sum_{j=1}^{w_i-1} 2\beta_{ij}\alpha_{ij} - \beta_{ij}\alpha_{ij-1} - \beta_{ij}\alpha_{ij+1},
$$
where $\alpha_{iw_i} = 0$, $\alpha_{i0} = \alpha_0$ and where $\beta_{iw_i} = 0$, $\beta_{i0} = \beta_0$.  The associated Tits quadratic can be expressed as:
\begin{align*}
q(\alpha) & = \frac{1}{2}(\alpha, \alpha) = \alpha_0^2 - \sum_{1 \leq i \leq k} \alpha_0\alpha_{i1} + \sum_{1 \leq i \leq k} \sum_{0 \leq j \leq w_i-1} \alpha_{ij}(\alpha_{ij} - \alpha_{i,j1}) \\ 
& = \alpha_0^2 - \sum_{1 \leq i \leq k} \alpha_0\alpha_{i1}+ \frac{1}{2}\sum_{1 \leq i \leq k}\alpha_{i1}^2 + \sum_{1 \leq i \leq k} \sum_{1 \leq j \leq w_i-1} \frac{1}{2}(\alpha_{ij} - \alpha_{ij+1})^2,
\end{align*}
where $\alpha_{iw_i} = 0$ and $\alpha_{i0} = \alpha_0$.  Recall that $p(\alpha) = 1 - q(\alpha)$.  Note that the Tits form can be defined on real vectors instead of integer vectors.  We distinguish the real version from the integer version by writing $q(x)$, instead of $q(\alpha)$, where $x = (x_0, x_{ij})$ is indexed by $I_{Q^{st}_{D,w}}$. 
   
To prove the theorem, it suffices to show that $\delta(\alpha) > 0$ and $\alpha$ in the fundamental region imply that $p(\alpha)>\sum_i p(\beta^{(i)})$ for any decomposition $\alpha = \sum_i \beta^{(i)}$ into the sum of nonzero dimension vectors.  However, before proving the inequality on $p(\alpha)$, we need several facts about the signature of $q(x)$.  Note that the signature will consist of a triple $(n_{+},n_{-},n_{0})$, corresponding to the positive index of inertia, the negative index of inertia, and the nullity, respectively.

\begin{prp}
\label{3.6.2}
Assume $q(x)$ has rank $n$.  On the $(n-1)$-dimensional subspace defined by $x_0 = 0$, we have that $q(x)$ is positive definite.
\end{prp}
\begin{proof}
From the expansion of $q(x)$ above we obtain that 
$$
q(x) = \frac{1}{2}\sum_{1 \leq i \leq k}x_{i1}^2 + \sum_{1 \leq i \leq k} \sum_{1 \leq j \leq w_k} \frac{1}{2}(x_{ij} - x_{ij+1})^2,
$$ 
for $x_0 = 0$.  It is clear that this implies $q(x) > 0$ for all nonzero $x$ with $x_0 = 0$.
\end{proof}

We immediately obtain:
\begin{cor}
\label{3.6.3}
Assume $q(x)$ has rank $n$.  The signature of $q(x)$ can be $(n,0,0)$, $(n-1,0,1)$, or $(n-1,1,0)$.
\end{cor}

The ordering on the elements of $\alpha$, such that $\alpha_{ij-1} - \alpha_{ij} \ge \alpha_{ij} - \alpha_{ij+1}$, together with $\delta(\alpha) > 0$, imply that $\alpha$ is in the fundamental region.
\begin{proof}[Proof of Proposition \ref{3.0.2}]

Note that the necessary inequality may be rewritten as 
$$ \sum_i q(\beta^{(i)}) - q(\alpha) > l-1.$$  We proceed by induction on $l$.  Consider the base case when $l = 2$.  In this case, we prove that the inequality holds for $\alpha = \beta + \gamma$.  We can directly compute 
\begin{align*}
(\alpha, \beta) & = 2\alpha_0\beta_0 - \sum_{i=1}^k \beta_0\alpha_{i1} + \sum_{i=1}^k \sum_{j=1}^{w_i-1} 2\beta_{ij}\alpha_{ij} - \beta_{ij}\alpha_{ij-1} - \beta_{ij}\alpha_{ij+1} \\
& = \beta_0(2\alpha_0 - \sum_{i=1}^k \alpha_{i1}) + \sum_{i=1}^k \sum_{j=1}^{w_i-1} \beta_{ij}(2\alpha_{ij} - \alpha_{ij-1} - \alpha_{ij+1}) \le 0.
\end{align*}
Similarly, we obtain $(\alpha, \gamma) \le 0$.  By Corollary \ref{3.6.3}, signature of $q(x)$ can be $(n,0,0)$, $(n-1,0,1)$, or $(n-1,1,0)$.  Since $q(\alpha) < 0$ it is $(n-1,1,0)$.   Restrict $q(x)$ to the subspace spanned by $\alpha$ and $\beta$. On this space the signature of $q(x)$ is $(1,1,0)$.  By the Gram-Schmidt process there is an orthogonal basis for this space containing $\alpha$.  That means we can write 
\begin{align*} 
& \beta = a_1\alpha + \delta_1 \\ & \gamma = a_2\alpha + \delta_2, 
\end{align*} 
where $a_i$ are nonnegative with $a_1 + a_2 = 1$, $\delta_1 + \delta_2 = 0$, $(\alpha, \delta_i) = 0$ and $q(\delta_i) \ge 0$ ($q(\delta_i) = 0$ only if $\delta_i = 0$), for all $i$.  
It follows that 
$$q(\beta) + q(\gamma) - q(\alpha) = -(\beta, \gamma) = -a_1a_2(\alpha, \alpha) - (\delta_1, \delta_2) \ge 1,$$
since the last sum is positive and $-(\beta, \gamma)$ is an integer.  Therefore, we have $(\beta, \beta) > (\alpha, \beta) \ge (\alpha, \alpha)$ hence $q(\beta) - q(\alpha) > 0$.  Similarly, we also have $q(\gamma) - q(\alpha) > 0$.  

We proceed by considering cases.  Let us first suppose that $q(\beta) \neq 0$ and $q(\gamma) \neq 0$.  We can assume without loss of generality that $(\alpha, \beta) \le (\alpha, \gamma)$.  We will suppose $(\beta, \gamma) = - 1$ and arrive at a contradiction.  From the previous decomposition in the orthogonal basis, we obtain that $a_1 \ge a_2$.  Therefore, 
$$(\gamma, \gamma) = a_2^2 (\alpha, \alpha) - (\delta_1, \delta_1) \ge a_1a_2(\alpha, \alpha) + (\delta_1, \delta_1) = -1,$$  
and it follows that $q(\gamma) \ge -\frac{1}{2}$.  Since $q(\gamma)$ is an integer we have $q(\gamma) > 0$.  Together with $q(\beta) - q(\alpha) > 0$ this gives us  $q(\beta)+ q(\gamma) - q (\alpha) > 1$, which is what we need.  Now suppose $q(\beta) = 0$.  We have that 
$$(\beta, \gamma) = (\beta, \alpha) = \beta_0(2\alpha_0 - \sum_i \alpha_{i1}) + \sum_{ij} \beta_{ij}(2\alpha_{ij} - \alpha_{ij-1} - \alpha_{ij+1}).$$  
Since $\delta(\alpha) > 0$, we have that $2\alpha_0 - \sum_i \alpha_{i1} \le -1$.  Thus, for $\beta_0 \ge 2$ and $\alpha$ in the fundamental region we have $-(\beta, \gamma) > 1$, contradicting our assumption that $(\beta, \gamma) = -1$.  If $\beta_0 = 0$, then we have  
$$
q(\beta) = \frac{1}{2}\sum_{1 \leq i \leq k}\beta_{i1}^2 + \sum_{1 \leq i \leq k} \sum_{1 \leq j \leq w_i-1} \frac{1}{2}(\beta_{ij} - \beta_{ij+1})^2 > 0,
$$
for nontrivial $\beta$.  This contradicts the original assumption that $q(\beta) = 0$.  If $\beta_0 = 1$, then we can show
$$
q(\beta) = 1 - \sum_{1 \leq i \leq k} \beta_{i1} +  \frac{1}{2}\sum_{1 \leq i \leq k}\beta_{i1}^2 + \sum_{1 \leq i \leq k} \sum_{1 \leq j \leq v_i-1} \frac{1}{2}(\beta_{ij} - \beta_{ij+1})^2 + \frac{1}{2}\sum_{1 \leq i \leq k} \beta_{iv_i}^2 > 0,
$$
where $v_i$ is the maximal entry with $\beta_{iv_i} \ne 0$.  Indeed, the inequality is valid since $\frac{1}{2}\beta_{i1}^2 + \frac{1}{2} \beta_{iv_i}^2 -  \beta_{i1} \ge 0$.  Again this contradicts the assumption that $q(\beta) = 0$. This covers all of the possibilities for $\beta$.  A similar argument works if $q(\gamma) = 0$. Hence, in all cases  $q(\beta) + q(\gamma) - q(\alpha) > 1$.

By induction we may assume that: 
\begin{align*} 
& q(\beta^{(1)}) + \cdots + q(\beta^{(l)}) - q (\alpha)\\ & =  q(\beta^{(1)}) + \cdots + q(\beta^{(i)} + \beta^{(j)}) - (\beta^{(i)}, \beta^{(j)}) + \cdots + q(\beta^{(l)}) - q(\alpha) \\ &> l-2 - (\beta^{(i)}, \beta^{(j)}), 
\end{align*}  
for any choice $i \ne j$.  Therefore, it suffices to prove that there exist differing $1 \le i,j \le l$ such that $(\beta^{(i)}, \beta^{(j)}) < 0$.  Consider the the subspaces spanned by $\alpha, \beta^{(i)}$.  As in the $l=2$ case, each such space has an orthogonal basis consisting of $\alpha$ and a vector on which $q(x)$ is positive.  It follows that for each $i$ we have $\beta^{(i)} = a_i \alpha + \delta_i$, with nonnegative $a_i$ such that $a_1 + \cdots + a_l = 1$, $\delta_1 + \cdots + \delta_l = 0$, $(\alpha, \delta_i) = 0$, and $q(\delta_i) \ge 0$.  Note that $q(\delta_i) = 0$ only when $\delta_i = 0$.  Now fix $\beta^{(i_0)}$.  If $\delta_{i_0} = 0$, then $1 > a_{i_0} > 0$.  There is a $j_0 \ne i_0$ such that 
$$(\beta^{(i_0)}, \beta^{(j_0)}) = a_{i_0}a_{j_0} (\alpha, \alpha) < 0.$$  
Otherwise, we have: 
$$\sum_i (\delta_i,\delta_{i_0}) = 0,$$
so for some $\beta^{(j_0)}$ it is true that $(\delta_{i_0}, \delta_{j_0}) < 0$, because $(\delta_{i_0},\delta_{i_0}) > 0$.  It follows that $$(\beta^{(i_0)}, \beta^{(j_0)}) = a_{i_0}a_{j_0} (\alpha, \alpha) + (\delta_{i_0}, \delta_{j_0}) < 0.$$ 
So, Proposition \ref{3.0.2} is proven.
\end{proof}

\begin{proof}[Proof of Theorem \ref{3.6.1}]
The theorem follows from Theorem \ref{3.0.1} and Proposition \ref{3.0.2}.
\end{proof}

%% file: moduli_of_parabolic_bundles.tex
\subsection{Outline}
\label{4.0}
In this section we will prove our main result, Theorem 1.2.1.  That is, we will prove that the moduli stack of parabolic bundles over $\mathbb{P}^1$ is almost very good under some restrictions on the parabolic structure.  Our proof resembles Crawley-Boevey's arguments in \cite{CB1993} and \cite{CB2001}.  However, Kac's theorem is inapplicable, and we replace it with an algebro-geometric result that works in the case of nontrivial parabolic bundles. 

Recall the notation $D,w,\alpha$ from Section 1.2.  Let $X$ be a complex projective curve and let $\textrm{Bun}_{D, w,\alpha}(X)$ be the moduli stack of parabolic bundles $\mathbf{E}$ of weight type $(D,w)$ and dimension vector $\alpha$ over $X$.  Let $\mathcal{P}_{\textrm{Bun}_{D, w,\alpha}(X)}$ be the stack of pairs $(\mathbf{E},f)$, where $\mathbf{E}$ is in $\textrm{Bun}_{D, w,\alpha}(X)$ and $f$ is its endomorphism.  

Note that $\mathcal{P}_{\textrm{Bun}_{D, w,\alpha}(X)}$ contains the inertia stack associated to $\textrm{Bun}_{D, w,\alpha}(X)$ as an open substack.  That is, it contains the stack $\mathcal{I}_{\textrm{Bun}_{D, w,\alpha}(X)}$, consisting of pairs $(\mathbf{E},f)$, where $\mathbf{E}$ is in $\textrm{Bun}_{D, w,\alpha}(X)$ and $f$ is its automorphism.  

Similarly, it contains the reduced closed substack $\mathcal{N}(D,w, \alpha)$, consisting of pairs $(\mathbf{E},f)$, where $\mathbf{E}$ is in $\textrm{Bun}_{D, w,\alpha}(X)$ and $f$ is its nilpotent endomorphism.

From now on, let $X = \mathbb{P}^1$.  Let $$\tilde{q}(\alpha) = \textrm{min} \sum q(\gamma_i),$$ where the minimum is taken over all positive, finite decompositions $\alpha  = \sum_i \gamma_i$.  We can reduce the proof of Theorem 1.2.1 to a dimension estimate for irreducible components of $\mathcal{P}_{\textrm{Bun}_{D, w,\alpha}(X)}$.  This, in turn, reduces to the following key estimate:
\begin{thm}
\label{4.0.1}
We have the inequality $\dimn \mathcal{N}(D,w, \alpha) \le -\tilde{q}(\alpha)$.
\end{thm}
Note that a similar estimate is used in the proof of Kac's Theorem to compute the number of parameters of $\textrm{Ind}(Q, \alpha)$. We will prove Theorem \ref{4.0.1} by induction on the rank of the nilpotent endomorphism $f$, reducing it to an application of deformation theory and a subsequent computation in hypercohomology.

\subsection{Generalities on Parabolic Bundles}
\label{4.1}

Let $X$ be a smooth complex projective curve.  
\begin{defn}
\label{4.1.1}
A $\textit{parabolic structure}$ on a vector bundle $E$ over $X$ consists of the following:
\begin{itemize}
\item[(1)]  A collection of distinct points $D =(x_1, \dots, x_k)$ in X.
\item[(2)] Flags $E_{x_i} = E_{i0} \supseteq E_{i1} \supseteq E_{iw_{i-1}}  \supseteq E_{iw_i} = 0$ in the fibers over the points $x_i$.
\end{itemize}
We call $E$ together with a parabolic structure on $E$ a $\textit{parabolic bundle}$ over $X$ and denote it by $\mathbf{E}$.  We denote its underlying bundle by $E$.
\end{defn}
If $D = (x_1, \dots, x_k)$ and $w = (w_1, \dots, w_k)$, then we say $\mathbf{E}$ has \textit{weight type} $(D,w)$.  Setting $\alpha_0 = \dimn E_{i0} = \textrm{rk } E$ and $\alpha_{ij} = \dimn E_{ij}$, we call $\alpha = (\alpha_0, \alpha_{i1}, \dots, \alpha_{iw_i-1})$ the \textit{dimension vector} of $\mathbf{E}$.  

A \textit{parabolic subbundle} $\mathbf{F} \subset \mathbf{E}$ is a vector subbundle $F \subset E$ together with a parabolic structure induced on $D$  by the parabolic structure of $\mathbf{E}$.  Note that some texts refer to the structure described above as a ``quasi-parabolic structure'' and to the associated parabolic bundle as a ``quasi-parabolic bundle''.

\begin{defn}
\label{4.1.2}
Let $\mathbf{E}$ and $\mathbf{F}$ be parabolic bundles of weight type $(D,w)$.  We call the morphism of vector bundles $f: F \rightarrow E$ a \textit{morphism of parabolic bundles} if $f_{x_i}(F_{ij}) \subset E_{ij}$ for all $1 \le i \le k$ and $1 \le j \le w_i$, where $f_{x_i}$ is the morphism induced by $f$ on the fiber over $x_i$.  

We denote the subsheaf of morphisms of parabolic bundles between $\mathbf{F}$ and $\mathbf{E}$ by $\mathscr{H}om_{\textrm{Par}}(\mathbf{F}, \mathbf{E}) \subset \mathscr{H}om(F,E)$ and the subsheaf of endomorphisms by $\mathscr{E}nd_{\textrm{Par}}(\mathbf{E}) \subset \mathscr{E}nd(E)$.  
\end{defn}

Note that  $\mathscr{H}om_{\textrm{Par}}(\mathbf{F}, \mathbf{E})$ and $\mathscr{E}nd_{\textrm{Par}}(\mathbf{E})$ are both vector bundles.  Therefore, we can compute the Euler characteristic of $\mathscr{H}om_{\textrm{Par}}(\mathbf{F}, \mathbf{E})$ by applying the Riemann-Roch theorem.  Specifically, let $\mathbf{E}$ have dimension vector $\alpha$ and let $\mathbf{F}$ have dimension vector $\beta$.  The degree of $\mathscr{H}om_{\textrm{Par}}(\mathbf{F}, \mathbf{E})$ may be computed, based on the degree of $\mathscr{H}om(F,E)$, as:
\begin{align*}
\textrm{deg}(\mathscr{H}om_{\textrm{Par}}(\mathbf{F}, \mathbf{E})) & = \textrm{rk}(F) \cdot \textrm{deg}(E) - \textrm{rk}(E) \cdot \textrm{deg}(F) \\ &- \sum_{i=1}^k \sum_{j=1}^{w_i-1}(\alpha_{0}-\alpha_{ij})(\beta_{ij} - \beta_{ij+1}).
\end{align*}
Therefore, by the Riemann-Roch theorem, we obtain that:
\begin{align*}
\chi(\mathscr{H}om_{\textrm{Par}}(\mathbf{F}, \mathbf{E})) & = \textrm{deg}(\mathscr{H}om_{\textrm{Par}}(\mathbf{F}, \mathbf{E})) + (1-g)\alpha_0\beta_0 \\ & = \beta_0 \cdot \textrm{deg}(E) - \alpha_0 \cdot \textrm{deg}(F) - g\alpha_0\beta_0 + \langle \beta, \alpha  \rangle,
\end{align*}
where $\langle \beta, \alpha \rangle$ is as in Definition \ref{3.1.3}.  Note that in the case when $g = 0$ and $\mathbf{F} = \mathbf{E}$, we obtain that
$$
\chi(\mathscr{E}nd_{\textrm{Par}}(\mathbf{E})) = q(\alpha).
$$

\subsection{The moduli stack of parabolic bundles over $\mathbb{P}^1$}
\label{4.2}

Definitions and general properties of algebraic stacks are given in Laumon and Moret-Bailly's book \cite{LMB2000}.  We will view a stack as a sheaf of groupoids in the fppf-topology and an algebraic stack as a stack with a smooth presentation by a scheme.  We will use $\langle \quad \rangle$ to denote a category in which the objects are enclosed by the brackets and the morphisms are all isomorphisms.

As before, let $X$ be the smooth complex projective curve.  Fix the weight type $(D,w)$ as in Section \ref{4.1}.  Let $I = \{0\} \cup \{(i,j)|1 \le i \le k, 1 \le j \le w_{i}-1 \}$, let $d \in \mathbb{Z}$ , and fix $\alpha \in \mathbb{Z}_{\ge 0}^I$, such that $\alpha_0 \ge \alpha_{i1} \ge \cdots \ge \alpha_{iw_i}$, for all $i$.  

\begin{defn}
\label{4.2.1}
 The stack of parabolic bundles of weight type $(D,w)$, degree $d$, dimension type $\alpha$, over $X$ is a functor that associates to a test scheme $T$ the groupoid $\textrm{Bun}_{D,w,\alpha}^d(T) = \left\langle (E,E^{i,j})_{1\le i \le k} \right\rangle$, where 
\begin{itemize}
\item $E$ is a vector bundle on $T\times X$,
\item $E|_{T \times \{x_i\}} \supset E^{i,1}\supset \dots \supset E^{i,w_i-1} \supset E^{i,w_i} = 0$
is a filtration by vector bundles,
\item $\textrm{rk}(E) = \alpha_0$ and $\textrm{rk}(E^{i,j}) = \alpha_{ij}$,
\item $\textrm{deg } E|_{\{y\} \times \mathbb{P}^1} = d$ for all $y \in T$.
\end{itemize}
\end{defn}

In the case when $X = \mathbb{P}^1$, we see that $\textrm{Bun}_{D,w,\alpha}^d$ admits the following presentation as an algebraic stack: $U = \coprod_{N \in \mathbb{Z}_{\ge 0}} \left\langle (\mathbf{E},s_i,t_j)\right\rangle$, where
\begin{itemize}
\item $\mathbf{E}$ is a parabolic bundle on  $X$, 
\item $\deg(E)=d$ and $\mathbf{E}$ has dimension vector $\alpha$,
\item $H^0(E^* \otimes \mathcal{O}(N))$ is generated by global sections,
\item $s_i$ is a basis for $H^0(E^* \otimes \mathcal{O}(N))^*$,
\item $t_j$ is a basis for $H^0(E^* \otimes \mathcal{O}(N-1))^*$.
\end{itemize}
For $X = \mathbb{P}^1$, we will give a more detailed description of $U$ in Section 6.  Let $\mathcal{B} := \textrm{Bun}_{D, w,\alpha}(X) = \coprod_{d \in \mathbb{Z}} \textrm{Bun}_{D,w,\alpha}^d$ be the moduli stack of parabolic bundles of weight type $(D,w)$ and with dimension vector $\alpha$.  We can use the presentation above to turn this stack into an algebraic stack.  

\begin{defn}
\label{4.2.2}
The stack of pairs $\mathcal{P}_{\textrm{Bun}_{D, w,\alpha}^d}$ is a functor that associates to a test scheme $T$ the groupoid $\mathcal{P}_{\textrm{Bun}_{D, w,\alpha}^d}(T) = \left\langle (E,E^{i,j},f)_{1\le i \le k} \right\rangle$, where
\begin{itemize}
\item $E$ is a vector bundle on $T\times X$,
\item $E|_{T \times \{x_i\}} \supset E^{i,1}\supset \dots \supset E^{i,w_i-1} \supset E^{i,w_i} = 0$
is a filtration by vector bundles,
\item $\textrm{rk}(E) = \alpha_0$ and $\textrm{rk}(E^{i,j}) = \alpha_{ij}$,
\item $\textrm{deg } E|_{\{y\} \times \mathbb{P}^1} = d$ for all $y \in T$,
\item $f$ is an endomorphism of $E$ such that $f|_{T \times \{x_i\}}(E^{i,j}) \subset E^{i,j}$ for all $i$.
\end{itemize}
\end{defn}

Before we give a presentation for $\mathcal{P}_{\textrm{Bun}_{D, w,\alpha}^d}$ as an algebraic stack, we will need some preliminary notation.  Let $\mathbf{E}, s_l, t_m$ be as in the description of $U$ above.  Let $G_0(s_l,t_m)$ and $G_1(s_l,t_m)$ be matrices in the bases $s_l$ and $t_m$, representing the morphisms from $H^0(E^* \otimes \mathcal{O}(N))^*$ to $H^0(E^* \otimes \mathcal{O}(N-1))^*$ that correspond to multiplication by $1, -z$, the two standard generating global sections of $\mathcal{O}(1)$.  Note that $\textrm{ker } (\lambda_{i0}G_0(s_l,t_m) + \lambda_{i1}G_1(s_l,t_m))$ contains the flag $E_{x_i} = E_{i0} \supseteq E_{i1} \supseteq E_{iw_{i-1}}  \supseteq E_{iw_i} = 0$, where $x_i = (\lambda_{i0}: \lambda_{i1})$ (see Section 6 for details).

In the case when $X = \mathbb{P}^1$, we see that $\mathcal{P}_{\textrm{Bun}_{D, w,\alpha}^d}$ admits the following presentation as an algebraic stack: $U = \coprod_{N \in \mathbb{Z}_{\ge 0}} \left\langle (\mathbf{E},s_l,t_m,F_1,F_2)\right\rangle$, where 
\begin{itemize}
\item $\mathbf{E}$ is a parabolic bundle on $X$, 
\item $\deg(E)=d$ and $\mathbf{E}$ has dimension vector $\alpha$,
\item $H^0(E^* \otimes \mathcal{O}(N))$ is generated by global sections,
\item $s_l$ is a basis for $H^0(E^* \otimes \mathcal{O}(N))^*$,
\item $t_m$ is a basis for $H^0(E^* \otimes \mathcal{O}(N-1))^*$,
\item $F_1$ is a matrix in the basis $s_l$ acting on $H^0(E^* \otimes \mathcal{O}(N))^*$,
\item $F_2$ is a matrix in the basis $t_m$ acting on $H^0(E^* \otimes \mathcal{O}(N-1))^*$,
\item $F_2 \circ G_r(s_l,t_m) = G_r(s_l,t_m) \circ F_1$ for $r = 0,1$,
\item $F_1(E_{ij}) \subset E_{ij}$ for all $1 \le i \le k$ and $1 \le j \le w_i$. 
\end{itemize}
Since our computations are independent of degree, then we will define the algebraic stack as $\mathcal{P}_{\mathcal{B}} := \mathcal{P}_{\textrm{Bun}_{D, w,\alpha}(X)} = \coprod_{d \in \mathbb{Z}} \mathcal{P}_{\textrm{Bun}_{D, w,\alpha}^d}$.

Let $\mathcal{I}_{\mathcal{B}} = \mathcal{I}_{\textrm{Bun}_{D, w,\alpha}(X)}$ be the inertia stack corresponding to $\textrm{Bun}_{D, w,\alpha}(X)$.  This is an open substack of $\mathcal{P}_{\mathcal{B}}$, where the endomorphism $f$ is taken to be an automorphism.  Note that $\textrm{dim } \mathcal{I}_{\mathcal{B}} = \textrm{dim } \mathcal{P}_{\mathcal{B}}$.  

The following stack will play an essential role in the proof of Theorem 1.2.1:
\begin{defn}
\label{4.2.3}
The stack $\mathcal{N}(D,w,d,\alpha)$ is a functor that associates to a test scheme $T$ the groupoid $\mathcal{N}(D,w,\alpha)(T) = \left\langle (E,E^{i,j},f)_{1\le i \le k} \right\rangle$, where
\begin{itemize}
\item $E$ is a vector bundle on $T\times X$,
\item $E|_{T \times \{x_i\}} \supset E^{i,1}\supset \dots \supset E^{i,w_i-1} \supset E^{i,w_i} = 0$
is a filtration by vector bundles,
\item $\textrm{rk}(E) = \alpha_0$ and $\textrm{rk}(E^{i,j}) = \alpha_{ij}$,
\item $\textrm{deg } E|_{\{y\} \times \mathbb{P}^1} = d$ for all $y \in T$,
\item $f$ is a nilpotent endomorphism of $E$ such that $f|_{T \times \{x_i\}}(E^{i,j}) \subset E^{i,j}$ for all $i$.
\end{itemize}
\end{defn}

We can see that $\mathcal{N}(D,w,d,\alpha)$ is a reduced closed algebraic substack of $\mathcal{P}_{\mathcal{B}}$, given the presentation (assuming that $X = \mathbb{P}^1$): $U = \coprod_{N \in \mathbb{Z}_{\ge 0}} \left\langle (\mathbf{E},s_l,t_m,F_1,F_2)\right\rangle$, where 
\begin{itemize}
\item $\mathbf{E}$ is a parabolic bundle on  $X$, 
\item $\deg(E)=d$ and $\mathbf{E}$ has dimension vector $\alpha$,
\item $H^0(E^* \otimes \mathcal{O}(N))$ is generated by global sections,
\item $s_l$ is a basis for $H^0(E^* \otimes \mathcal{O}(N))^*$,
\item $t_m$ is a basis for $H^0(E^* \otimes \mathcal{O}(N-1))^*$,
\item $F_1$ is a nilpotent matrix in the basis $s_l$ acting on $H^0(E^* \otimes \mathcal{O}(N))^*$,
\item $F_2$ is a nilpotent matrix in the basis $t_m$ acting on $H^0(E^* \otimes \mathcal{O}(N-1))^*$,
\item $F_2 \circ G_r(s_l,t_m) = G_r(s_l,t_m) \circ F_1$ for $r = 0,1$,
\item $F_1(E_{ij}) \subset E_{ij}$ for all $1 \le i \le k$ and $1 \le j \le w_i$. 
\end{itemize}
Our computations are independent of degree, we will define the algebraic stack $\mathcal{N}(D,w,\alpha) :=  \coprod_{d \in \mathbb{Z}} \mathcal{N}(D,w,d,\alpha)$.

Note that $\textrm{Bun}_{D, w,\alpha}(X)$ is smooth, and by Lemma \ref{2.2.1} we can compute its dimension as: 
\begin{align*}
& \textrm{dim Bun}_{D, w,\alpha}(X) = \textrm{dim Bun}_{\textrm{GL}(\alpha_0)}(X) + \textrm{dim } Fl(\alpha) \\
& = (g-1)\alpha_0^2 + \alpha_0^2 - q(\alpha) = g\alpha_0^2 - q(\alpha).
\end{align*}
From now on, let $X = \mathbb{P}^1$.  This means $g = 0$, and therefore $\dimn \textrm{Bun}_{D, w,\alpha}(X) =  -q(\alpha)$.  

\subsection{Proof of Theorem 1.2.1}
\label{4.3}

Let us define
$$
\tilde{q}(\alpha) = \textrm{min} \sum q(\gamma_i),
$$
where the minimum is taken over all positive, finite decompositions $\alpha  = \sum_i \gamma_i$.  We can summarize the properties of $\tilde{q}(\alpha)$ in the following proposition:
\begin{prp}
\label{4.3.1}
Let $\alpha$ and $\beta$ be dimension vectors.  For $\tilde{q}(\alpha)$, we have: 

a) $\tilde{q}(\alpha) \le q (\alpha)$ 

b) $\tilde{q}(\alpha + \beta) \le \tilde{q}(\alpha) + \tilde{q}(\beta)$

c) $\tilde{q}(\alpha) = q(\alpha)$, if $\alpha$ is in the fundamental region. 
\end{prp} 
\begin{proof}
Parts a) and b) follow directly from the definition.  Part c) is equivalent to the inequality 
$$ 
q(\alpha) \le \sum_i q(\gamma_i),
$$
for any finite positive decomposition $\alpha  = \sum_i \gamma_i$ and $\alpha$ in the fundamental region.  This follows from the proof of Proposition \ref{3.0.2}.
\end{proof}

Consider the two-element complex
$$
C^{\bullet}: \mathscr{E}nd_{\textrm{Par}}(\mathbf{W}) \rightarrow \mathscr{H}om_{\textrm{Par}}(\mathbf{V}, \mathbf{W}),
$$
induced by the inclusion of parabolic bundles $i: \mathbf{V} \hookrightarrow \mathbf{W}$.  This complex arises when we consider first-order deformations of pairs $(\mathbf{W}, i)$, for a fixed $\mathbf{V}$.  However, the usual sheaf cohomology is no longer sufficient to determine these deformations.  Instead, one can generalize the notion of sheaf cohomology to $\textit{hypercohomology}$, in order to obtain a cohomology theory for chain complexes of sheaves (see e.g. \cite{We1994}).  By analogy with sheaf cohomology, we can compute hypercohomology by means of a \v{C}ech resolution, for a sufficiently good cover.  It follows that we can study the deformations of the pairs $(\mathbf{W}, i)$ by studying the hypercohomology groups of $C^{\bullet}$. 
\begin{lmm}
\label{4.3.2}
We have that $\mathbb{H}^2(\mathbb{P}^1, C^{\bullet}) = 0$.
\end{lmm}
\begin{proof}
Consider the chain complexes
\begin{align*}
& A^{\bullet}: 0 \rightarrow \mathscr{E}nd_{\textrm{Par}}(\mathbf{W}) \\
& B^{\bullet}: 0 \rightarrow \mathscr{H}om_{\textrm{Par}}(\mathbf{V}, \mathbf{W}),
\end{align*}
which are nontrivial only in degree $1$.  Since $i$ induces the obvious chain map, we have an exact triangle $A^{\bullet} \rightarrow B^{\bullet} \rightarrow C^{\bullet}$, which gives rise to the long exact sequence for hypercohomology
$$
\cdots \rightarrow \mathbb{H}^2(\mathbb{P}^1,A^{\bullet}) \rightarrow \mathbb{H}^2(\mathbb{P}^1,B^{\bullet}) \rightarrow \mathbb{H}^2(\mathbb{P}^1,C^{\bullet}) \rightarrow \mathbb{H}^3(\mathbb{P}^1,A^{\bullet}) \rightarrow \cdots .
$$
Since $A^{\bullet}$ and $B^{\bullet}$ are only nontrivial in degree $1$, we have both that $\mathbb{H}^2(\mathbb{P}^1, A^{\bullet}) = H^1(\mathbb{P}^1, \mathscr{E}nd_{\textrm{Par}}(\mathbf{W}))$ and $\mathbb{H}^2(\mathbb{P}^1,B^{\bullet}) = H^1(\mathbb{P}^1, \mathscr{H}om_{\textrm{Par}}(\mathbf{V}, \mathbf{W}))$. We also obtain that $\mathbb{H}^3(\mathbb{P}^1,A^{\bullet}) = 0$.  Hence, it follows that we have the exact sequence
$$
H^1(\mathbb{P}^1, \mathscr{E}nd_{\textrm{Par}}(\mathbf{W})) \rightarrow H^1(\mathbb{P}^1, \mathscr{H}om_{\textrm{Par}}(\mathbf{V}, \mathbf{W})) \rightarrow  \mathbb{H}^2(\mathbb{P}^1, C^{\bullet}) \rightarrow 0.
$$
Therefore, it follows $\mathbb{H}^2(\mathbb{P}^1,C^{\bullet})$ is the cokernel of $i^*: H^1(\mathbb{P}^1, \mathscr{E}nd_{\textrm{Par}}(\mathbf{W})) \rightarrow H^1(\mathbb{P}^1, \mathscr{H}om_{\textrm{Par}}(\mathbf{V}, \mathbf{W}))$.  Applying Serre Duality, we obtain that $\mathbb{H}^2(\mathbb{P}^1,C^{\bullet})$ is isomorphic to the dual of the kernel of
$$
H^0(\mathbb{P}^1, \mathscr{H}om_{\textrm{Par}}(\mathbf{W}, \mathbf{V})\otimes \Omega^1) \rightarrow H^0(\mathbb{P}^1, \mathscr{E}nd_{\textrm{Par}}(\mathbf{W})\otimes \Omega^1).
$$
However, this map comes from the inclusion of $\mathscr{H}om_{\textrm{Par}}(\mathbf{W}, \mathbf{V}) \hookrightarrow \mathscr{E}nd_{\textrm{Par}}(\mathbf{W})$, which is induced by $i$.  Therefore, the map is injective, so the kernel is trivial.  Thus,  $\mathbb{H}^2(\mathbb{P}^1,C^{\bullet}) = 0$.
\end{proof}

Let $\mathbf{V}$ be a parabolic bundle over $\mathbb{P}^1$ and let $\mathcal{P}_{\mathbf{V}} = \mathcal{P}_{\mathbf{V}}(D,w,\alpha)$ be the algebraic stack consisting of pairs
$\{\mathbf{W}, i: \mathbf{V} \hookrightarrow \mathbf{W} \}$, where $i$ is an inclusion of parabolic bundles and $\mathbf{W}$ is a parabolic bundle of weight type $(D,w)$ and dimension vector $\alpha$.

\begin{lmm}
\label{4.3.3}
Either $\mathcal{P}_{\mathbf{V}}(D,w,\alpha)$ is empty or we have 
$$\dimn \mathcal{P}_{\mathbf{V}}(D,w,\alpha) = \chi(\mathscr{H}om_{\textrm{Par}}(\mathbf{V}, \mathbf{W})) -  \chi (\mathscr{E}nd_{\textrm{Par}}(\mathbf{W})).$$
\end{lmm}
\begin{proof}

Assume that $\mathcal{P}_{\mathbf{V}}$ is nonempty.  The dimension of $\mathcal{P}_{\mathbf{V}}$ is equal to the dimension of the corresponding tangent complex.  We compute its dimension  by considering the deformations of $(\mathbf{W}, i) \in \mathcal{P}_{\mathbf{V}}$.  These deformations are governed by the hypercohomology of the complex $C^{\bullet}$, defined above.  It follows that
$$
\textrm{dim } \mathcal{P}_{\mathbf{V}} = \textrm{dim } \mathbb{H}^1(\mathbb{P}^1, C^{\bullet}) - \textrm{dim } \mathbb{H}^0(\mathbb{P}^1, C^{\bullet}),
$$
since $\mathbb{H}^2(C^{\bullet}) = 0$ by Lemma \ref{4.3.2}. 

Let $\chi(D^{\bullet})$ denote the Euler characteristic of the hypercohomology of a complex of sheaves $D^{\bullet}$ and let $A^{\bullet}, B^{\bullet}$ be as in Lemma \ref{4.3.2}. Since $\chi (D^{\bullet})$ additive on exact triangles, we have that
$$
\chi(C^{\bullet})  = \chi(B^{\bullet}) - \chi(A^{\bullet}).
$$
Moreover, because $\chi(B^{\bullet}) =  - \chi(\mathscr{H}om_{\textrm{Par}}(\mathbf{V}, \mathbf{W}))$ and $\chi(A^{\bullet}) = -\chi (\mathscr{E}nd_{\textrm{Par}}(\mathbf{W}))$, we can simplify this to
$$
\chi(C^{\bullet}) = \chi (\mathscr{E}nd_{\textrm{Par}}(\mathbf{W})) - \chi(\mathscr{H}om_{\textrm{Par}}(\mathbf{V}, \mathbf{W})).
$$
By Lemma \ref{4.3.2},  $\textrm{dim } \mathcal{P}_{\mathbf{V}} = - \chi(C^{\bullet})$.  Thus,
$$
\textrm{dim } \mathcal{P}_{\mathbf{V}} = \chi(\mathscr{H}om_{\textrm{Par}}(\mathbf{V}, \mathbf{W})) -  \chi (\mathscr{E}nd_{\textrm{Par}}(\mathbf{W})).
$$
\end{proof}

Let $\mathbf{F}, \mathbf{G}$ be parabolic bundles over $\mathbb{P}^1$, and let $g$ be an endomorphism of $\mathbf{G}$.  Let $D^{\bullet}$ be the following chain complex: 
$$
\mathscr{H}om_{\textrm{Par}}(\mathbf{G}, \mathbf{F}) \rightarrow \mathscr{H}om_{\textrm{Par}}(\mathbf{G}, \mathbf{F}),
$$
where the connecting map is induced by $g$.
\begin{lmm}
\label{4.3.4}
We can compute the following: $\dimn \mathbb{H}^1(\mathbb{P}^1, D^{\bullet}) - \dimn \mathbb{H}^0(\mathbb{P}^1, D^{\bullet}) = \dimn H^1(\mathbb{P}^1, \mathscr{H}om_{\textrm{Par}}(\textbf{ker } g, \mathbf{F}))$.
\end{lmm}
\begin{proof}
Since $D^{\bullet}$ consists of two copies of $\mathscr{H}om_{\textrm{Par}}(\mathbf{G}, \mathbf{F})$ we can see (by the argument from Lemma \ref{4.3.2}) that the Euler characteristic for hypercohomology is $0$.  That is, we have:
$$
\dimn \mathbb{H}^1(\mathbb{P}^1, D^{\bullet}) - \dimn \mathbb{H}^0(\mathbb{P}^1, D^{\bullet}) = \dimn \mathbb{H}^2(\mathbb{P}^1, D^{\bullet}).
$$
By Serre duality, $\mathbb{H}^2(\mathbb{P}^1, D^{\bullet})$ is isomorphic to $\mathbb{H}^0$ for the complex
$$
\mathscr{H}om_{\textrm{Par}}(\mathbf{F}, \mathbf{G}\otimes \Omega^1_{\mathbb{P}^1}) \rightarrow \mathscr{H}om_{\textrm{Par}}(\mathbf{F}, \mathbf{G}\otimes \Omega^1_{\mathbb{P}^1}),
$$
where the connecting map is induced by $g\otimes \text{Id}$.  However, by definition, this is just: 
$$
H^0(\mathbb{P}^1, \mathscr{H}om_{\textrm{Par}}(\mathbf{F}, (\textbf{ker } g)\otimes \Omega^1_{\mathbb{P}^1})) \cong H^0(\mathbb{P}^1, \mathscr{H}om_{\textrm{Par}}(\mathbf{F}, \textbf{ker } g)\otimes \Omega^1_{\mathbb{P}^1}).
$$
Applying Serre duality, we get: 
$$
\dimn \mathbb{H}^1(\mathbb{P}^1, D^{\bullet}) - \dimn \mathbb{H}^0(\mathbb{P}^1, D^{\bullet}) = \dimn \mathbb{H}^2(\mathbb{P}^1,D^{\bullet}) = \dimn H^1(\mathbb{P}^1, \mathscr{H}om_{\textrm{Par}}(\textbf{ker } g, \mathbf{F})).
$$
\end{proof}
  
Now we can proceed with the proof of our key argument:
\begin{proof}[Proof of Theorem \ref{4.0.1}]
Let $(\mathbf{E}, f)$ be a point of $\mathcal{N}(D,w, \alpha)$.  Let $\mathbf{F} = \mathbf{ker} \ f$ and $\mathbf{G} = \mathbf{E}/ \mathbf{F}$.  We wish to prove this theorem by induction on the rank of the vector bundle $E$ (note that this is $\alpha_0$ in our notation).  To that end, it suffices to prove that for all $\beta$ we have:
$$
\textrm{dim } \mathcal{N}_{\beta}(D,w, \alpha) \le -\tilde{q}(\alpha),
$$
where $\mathcal{N}_{\beta}(D,w, \alpha)$ is a substack consisting of objects $(\mathbf{E},f)$ of $\mathcal{N}(D,w, \alpha)$ such that the corresponding $\mathbf{F}$ belongs to $\textrm{Bun}_{D,w, \beta}(X)$.  In order to accomplish this, consider the morphism
$$
\phi: \mathcal{N}_{\beta}(D,w, \alpha) \rightarrow \mathcal{N}(D,w, \alpha - \beta),
$$
which is defined by sending $(\mathbf{E}, f)$ to
$(\mathbf{G}, f|_{\mathbf{G}})$ $\in \mathcal{N}(D, w, \alpha - \beta)$, with corresponding restrictions on the arrows.  In this case, after applying the induction hypothesis, we get
$$
\textrm{dim } \mathcal{N}_{\beta}(D,w, \alpha) \le \textrm{dim } \mathcal{N}_{\beta}(D,w, \alpha)_x - \tilde{q}(\alpha - \beta) ,
$$
for some $x = (\mathbf{G},g) \in \mathcal{N}(D, w, \alpha - \beta)$.  Now, we wish to compute the dimension of the fiber $\mathcal{X} = \mathcal{N}_{\beta}(D,w, \alpha)_x$.  Let $\mathbf{F_1} = \mathbf{ker} \ g$ and let $\mathcal{X}' = \mathcal{P}_{\mathbf{F_1}}(D, w, \beta)$.  In this case, we have two morphisms $\psi_1: \mathcal{X} \rightarrow \textrm{Bun}_{D,w, \beta}(X)$ and $\psi_2: \mathcal{X}' \rightarrow \textrm{Bun}_{D,w, \beta}(X)$, where $\psi_1$ sends the pair $(\mathbf{E},f)$ to $\textbf{ker } f$ and likewise $\psi_2$ sends $(\mathbf{F},i)$ to $\mathbf{F}$. 

The deformations of elements of the fiber $\mathcal{X}_{\mathbf{F}}$ are governed by the hypercohomology of the complex 
$$
\mathscr{H}om_{\textrm{Par}}(\mathbf{G}, \mathbf{F}) \xrightarrow{g} \mathscr{H}om_{\textrm{Par}}(\mathbf{G}, \mathbf{F}),
$$
defined in Lemma \ref{4.3.4}.  Therefore, by Lemma \ref{4.3.4}, we get that:
$$
\textrm{dim } \mathcal{X}_{\mathbf{F}} = \textrm{dim } H^1 (\mathbb{P}^1, \mathscr{H}om_{\textrm{Par}}(\mathbf{F_1}, \mathbf{F})).
$$

Furthermore, since $f$ induces an injective morphism $\textbf{ker } f^2/\textbf{ker }f \rightarrow \textbf{ker } f$, then the fiber 
$\mathcal{X}'_{\mathbf{F}}$ is nonempty.  Therefore, 
$$
\textrm{dim } \mathcal{X}'_{\mathbf{F}} = \textrm{dim } H^0 (\mathbb{P}^1, \mathscr{H}om_{\textrm{Par}}(\mathbf{F_1}, \mathbf{F})).
$$

Thus, $\textrm{dim } \mathcal{X}_{\mathbf{F}}  = \textrm{dim } \mathcal{X}'_{\mathbf{F}} - \chi(\mathscr{H}om_{\textrm{Par}}(\mathbf{F_1}, \mathbf{F}))$.  We have $\textrm{dim } \mathcal{X} = \textrm{dim } \mathcal{X}' - \chi(\mathscr{H}om_{\textrm{Par}}(\mathbf{F_1}, \mathbf{F}))$. So, we obtain
$$
\textrm{dim } \mathcal{N}_{\beta}(D,w, \alpha) \le \textrm{dim } \mathcal{X}'  - \tilde{q}(\alpha - \beta) - \chi(\mathscr{H}om_{\textrm{Par}}(\mathbf{F_1}, \mathbf{F})).
$$
It follows from Lemma \ref{4.3.3} that $\dimn \mathcal{X}' = \chi(\mathscr{H}om_{\textrm{Par}}(\mathbf{F_1}, \mathbf{F})) -  \chi (\mathscr{E}nd_{\textrm{Par}}(\mathbf{F}))$, which means
$$
\textrm{dim } \mathcal{N}_{\beta}(D,w, \alpha) \le -\chi (\mathscr{E}nd_{\textrm{Par}}(\mathbf{F})) - \tilde{q}(\alpha - \beta).
$$
Since $\chi (\mathscr{E}nd_{\textrm{Par}}(\mathbf{F})) = q(\beta)$ and $\tilde{q}(\alpha) \le \tilde{q}(\alpha - \beta) + \tilde{q}(\beta)$ (by Proposition \ref{4.3.1} b)), we can reduce this to
$$
\textrm{dim } \mathcal{N}_{\beta}(D,w, \alpha) \le -q(\beta) +\tilde{q}(\beta) -\tilde{q}(\alpha).
$$
The result follows from Proposition \ref{4.3.1} a). 
\end{proof}

\begin{cor}
\label{4.3.5}
For $\alpha$ lying in the fundamental region, we have $\dimn \mathcal{N}(D,w,\alpha) \le -q(\alpha)$.  If, in addition, $\delta(\alpha) > 0$, then $\dimn (\mathcal{N}(D,w,\alpha) - \mathcal{N}_{\alpha}(D,w,\alpha)) < -q(\alpha).$
\end{cor}
\begin{proof}
The first statement clearly follows from Theorem \ref{4.0.1} and Proposition \ref{4.3.1} c).  Now, let $\alpha$ be in the fundamental region and $\delta(\alpha)>0$.  By the proof of Theorem \ref{4.0.1}, 
$$\textrm{dim } \mathcal{N}_{\beta}(D,w, \alpha) \le -q(\beta) - \tilde{q}(\alpha - \beta),$$
for all nonnegative $\beta \le \alpha$.  If $\alpha \neq \beta$, then by Proposition \ref{3.0.2}, $\textrm{dim } \mathcal{N}_{\beta}(D,w, \alpha) < -q(\alpha)$.
\end{proof}

Let $c: \mathcal{P}_{\mathcal{B}} \rightarrow \mathbb{A}^{\alpha_0}$ be the morphism defined by sending the pair $(\mathbf{E}, f)$ to the coefficients of the characteristic polynomial $\textrm{char}(f)$ of $f$.  We will need the following lemma:

\begin{lmm}
\label{4.3.6}
There exists a decomposition into nonnegative dimension vectors $\alpha = \sum_{i=1}^r \beta^{(i)}$ such that $\dimn \mathcal{P}_{\mathcal{B}} = r + \sum_{i=1}^r \dimn \mathcal{N}(D,w,\beta^{(i)})$.
\end{lmm}
\begin{proof}
Fix a point of $x \in \mathbb{A}^{\alpha_0}$.  This defines some characteristic polynomial $x(t)= (t -  \lambda_1)^{m_1}(t - \lambda_2)^{m_2} \cdots (t - \lambda_r)^{m_r}$.  Consider $(\mathcal{P}_{\mathcal{B}})_x$, the fiber of $c$ over $x$.  The points of $(\mathcal{P}_{\mathcal{B}})_x$
may be identified with pairs $(\mathbf{E}, f)$, such that $f$ is an endomorphism of the parabolic bundle
$\mathbf{E}$ with $\textrm{char}(f) = x(t)$.  Therefore, $\mathbf{E}$ decomposes as 
$$
\mathbf{E} = \bigoplus_i \textbf{ker}(f-\lambda_i)^{m_i},
$$
and the fiber $(\mathcal{P}_{\mathcal{B}})_x$ is isomorphic to $\prod_i \mathcal{P}_i$.  Here 
$\mathcal{P}_i$ is the substack of pairs $(\mathbf{E}_i, f_i)$, where $\mathbf{E}_i$ is a parabolic bundle and $f_i$ is its endomorphism such that $\textrm{char}(f_i) = (t - \lambda_i)^{m_i}$.  Since $f_i - \lambda_i$
is nilpotent, we can compute 
$$\dimn \mathcal{P}_i = \textrm{dim } \mathcal{N}(D,w,\beta^{(i)}),$$
for some dimension vector $\beta^{(i)} \le \alpha$.  Note that $\alpha = \beta_1 + \cdots + \beta_r$.  Since $c$ maps $(\mathcal{P}_{\mathcal{B}})_x$ to the subvariety consisting of polynomials with $r$ distinct roots, we can compute:
$$
\dimn \mathcal{P}_{\mathcal{B}} = r + \sum_{i=1}^r \dimn \mathcal{N}(D,w,\beta^{(i)}),
$$
for some decomposition $\alpha = \sum_{i=1}^r \beta^{(i)}$ into nonnegative dimension vectors. 
\end{proof}

\begin{proof}[Proof of Theorem 1.2.1]
Suppose $r=1$ in Lemma \ref{4.3.6}.  That is, $\dimn \mathcal{P}_{\mathcal{B}} - 1 = \dimn \mathcal{N}(D,w,\alpha)$.  Let $\mathcal{I}^i$ be components of the inertia stack $\mathcal{I}_{\mathcal{B}}$ as in Theorem \ref{2.2.2}.

Assume that $\dimn \mathcal{I}^{i_0} - 1 = \dimn \mathcal{N}(D,w,\alpha)$ for some $i_0 > 1$.  Otherwise, $\dimn \mathcal{I}^{i} - 1 < \dimn \mathcal{N}(D,w,\alpha)$ for all $i>1$, and we are done by Corollary \ref{4.3.5} and Corollary \ref{2.2.4}.

 Let $\mathcal{B}^{i_0}$ be the substack of $\textrm{Bun}_{D,w, \alpha}(X)$ consisting of vector bundles with $i_0$-dimensional endomomorphism groups.  By the proof of Lemma \ref{4.3.6}, the subspace of nilpotent endomorphisms of an element in $\mathcal{B}^{i_0}$ has dimension $i-1$.

Recall from the proof of Theorem \ref{2.2.2} that $\dimn \mathcal{I}^i = \dimn \mathcal{B}^i + i$.  Since the stack $\mathcal{N}_{\alpha}(D,w,\alpha)$ may be interpreted as pairs $(\mathbf{E}, f)$, where $\mathbf{E}$ is a parabolic bundle and $f$ is the zero endomorphism, we have: 
$$
 \dimn (\mathcal{I}_{\mathcal{B}} - \coprod_{i = 0}^1 \mathcal{I}^i) - 1 = \dimn \mathcal{I}^{i_0} - 1 = \dimn \mathcal{B}^{i_0} + i - 1 = \dimn (\mathcal{N}(D,w,\alpha) - \mathcal{N}_{\alpha}(D,w,\alpha)).
$$
 Therefore, by Corollary \ref{4.3.5} and Corollary \ref{2.2.4}, $\textrm{Bun}_{D,w, \alpha}(X)$ is almost very good.

Now, suppose $r \ge 2$ in Lemma \ref{4.3.6}.  In this case, by Proposition  \ref{3.0.2}, Lemma \ref{4.3.6}, and Corollary \ref{4.3.5}, we have that:
$$
\dimn \mathcal{I}_{\mathcal{B}}  = \dimn \mathcal{P}_{\mathcal{B}} \le \sum_{i=1}^r p(\beta^{(i)}) < p(\alpha).
$$
Therefore, $\dimn \mathcal{I}_{\mathcal{B}} - 1 < \dimn \textrm{Bun}_{D,w, \alpha}(X)$.  It follows from Corollary \ref{2.2.4} that $\textrm{Bun}_{D,w, \alpha}(X)$ is almost very good.
\end{proof}

%% file: stability_for_parabolic_bundles.tex
\subsection{Outline}
\label{5.0}
We wish to define stability and semistability for parabolic bundles similarly to how they are defined for vector bundles without parabolic structure.  Following \cite{MS1980}, we extend the definition of the degree for vector bundles over a curve $X$ to parabolic bundles by introducing additional parameters $\theta$ called \textit{weights}.  We can define the \textit{slope} of a parabolic bundle using the rank and the parabolic degree.  Stability and semistablility are introduced following the usual definition for vector bundles.

Let $X = \mathbb{P}^1$.  If we restrict ourselves to the open substack $\textrm{Bun}_{D, w,\alpha}^{\theta, ss}(X)$ of semistable parabolic bundles in $\textrm{Bun}_{D, w,\alpha}(X)$, we can simplify the proof of Theorem \ref{4.0.1}, the key argument in the proof of Theorem 1.2.1.  To be precise, recall that
$$
\tilde{q}(\alpha) = \textrm{min} \sum q(\gamma_i),
$$
where the minimum is taken over all positive, finite decompositions $\alpha  = \sum_i \gamma_i$ and $\mathcal{N}(D,w,d,\alpha)$ is the stack of pairs $(\mathbf{E},f)$, such that $\mathbf{E} \in \textrm{Bun}_{D, w,\alpha}(X)$ and $f$ is a nilpotent endomorphism of $\mathbf{E}$.  Let $\mathcal{N}^{\theta, ss}(D,w,d,\alpha)$ be the open substack of pairs $(\mathbf{E},f) \in \mathcal{N}(D,w,d,\alpha)$, such that $\mathbf{E}$ is semistable.  We have the following equivalent of Theorem \ref{4.0.1} for semistable parabolic bundles:
\begin{thm}
\label{5.0.1}
We have the inequality $\dimn \mathcal{N}^{\theta, ss}(D,w, \alpha) \le -\tilde{q}(\alpha)$.
\end{thm}

King defines stability and semistability for quiver representations (see \cite{King1994}).  We explore the relationship between parabolic bundles over $\mathbf{P}^1$ and quiver representations (see section 6 or \cite{CB2004}), by presenting a correspondence between certain semistable parabolic bundles and certain semistable squid representations. 
 
\subsection{Definitions}
\label{5.1}

Recall that one can define stability and semistability for vector bundles by introducing the notion of slope. Indeed, for a vector bundle $E$ of rank $r$ and degree $d$ over a complex projective curve $X$, we define the slope of $E$ to be $\mu(E) = \frac{d}{r}$.  We say that $E$ is semistable if $\mu(F) \le \mu(E)$ for all nonzero subbundles $F$ of $E$.  If the inequality is strict for all proper, nonzero subbundles $F$, then we say that $E$ is stable.

We can define stability and semistability for parabolic bundles similarly to how we define them for vector bundles.  Let $D$, $w$ be as in the Introduction.  In order to define an analogue of degree for parabolic bundles, we introduce a collection of real numbers $\theta = (\theta_{ij})$, where $1 \le i \le k$, while $0 \le j \le w_{k} - 1$ and $0 \le \theta_{i1} < \theta_{i2} \cdots < \alpha_{iw_k-1} < 1$.  For a parabolic bundle $\mathbf{E}$ of weight type $(D,w)$, the \textit{weights} associated with $\mathbf{E}$ is such a collection of numbers $\theta$. We may think of $\theta_{ij}$ as being ``attached" to the flag subspace $E_{ij}$.  If $\mathbf{E}$ has dimension vector $\alpha$, then we say that the \textit{multiplicity} of $\theta_{ij}$ is $m_{ij} = \alpha_{ij} - \alpha_{ij-1}$.  

Let $d$ be the degree of the vector bundle $E$.  We call $\textrm{par deg } (\mathbf{E}) := d + \sum_{i,j} m_{ij}\theta_{ij}$ the \textit{parabolic degree} of $\mathbf{E}$.  The \textit{parabolic slope} of $\mathbf{E}$ is defined as: 
$$
\mu(\mathbf{E}) = \frac{d + \sum_{i,j} m_{ij}\theta_{ij}}{\alpha_0}.
$$
We say that $\mathbf{E}$ is $\theta$-\textit{semistable} if for every nonzero parabolic subbundle $\mathbf{F} \subset \mathbf{E}$ we have $\mu(\mathbf{F}) \le \mu(\mathbf{E})$.  If the inequality is strict for all proper, nonzero parabolic subbundles, then we say $\mathbf{E}$ is $\theta$-\textit{stable}. 
  
\subsection{Semistability and the Very Good property}
\label{5.2}
Let $E$ and $F$ be vector bundles over $\mathbb{P}^1$, such that $E$ is a subsheaf or $F$.  Recall that the \textit{saturation} of a vector bundle $E$ in $F$ is the inverse image vector bundle in $F$ of $(F/E)/T(F/E)$, where $T(F/E)$ is the torsion sheaf of $F/E$.  We can obtain the following:

\begin{lmm}
\label{5.2.1}
If $\varphi: \mathbf{E} \rightarrow \mathbf{F}$ is a morphism of parabolic bundles that is injective on the sheaves of sections, then $\textrm{par deg } \mathbf{E} \le \textrm{par deg } \mathbf{V}$, where $\mathbf{V}$ is the parabolic bundle induced by the saturation of $E$ in $F$.  Equality holds only if $E = V$.
\end{lmm}
\begin{proof}
It suffices to consider bundles with parabolic structure only at one point, $x_1 \in X$.  Let $(\theta_0, \dots, \theta_{w_1})$ be the weights of the spaces in the flags at $x_1$ for $\mathbf{E}$ and $\mathbf{V}$, and let $(\alpha_0, \dots, \alpha_{w_1})$, $(\beta_0, \dots, \beta_{w_1})$ be the respective dimensions of the spaces in the flags. Let $d$ and $d'$ be the degrees of $E$ and $V$, respectively.  Note that, considered as a sheaf, $V$ coincides with $E$ except at finitely many points.  That means the induced map on fibers has full rank everywhere except finitely many points.  It is enough to consider the case when the rank drops only at $x_1$.  Let $b$ be the dimension of the kernel of $\varphi_{x_1}$, the map induced by $\varphi$ on the fiber at $x_1$ .  We have that $d' - d \ge b$.  Now, we need to prove:
$$
\textrm{par deg } \mathbf{E} = d + \sum_i \theta_i(\alpha_i-\alpha_{i+1}) \le \textrm{par deg } \mathbf{V} = d' + \sum_i \theta_i(\beta_i-\beta_{i+1}),
$$
We can obtain:
\begin{align*}
& d + \sum_i \theta_i(\alpha_i-\alpha_{i+1})\\ & = d + \theta_{0}\alpha_{0} + \sum_i \alpha_i(\theta_{i+1} - \theta_{i}) \\ & \le d + \theta_{0}\beta_{0} + \sum_i \beta_i(\theta_{i+1} - \theta_{i})\\ & \le
d' +  \theta_{0}\beta_{0} + \sum_i \beta_i(\theta_{i+1} - \theta_{i})\\ & = d' + \sum_i \theta_i(\beta_i-\beta_{i+1}).
\end{align*}
 Now, from the definitions, it follows that $\textrm{par deg } \mathbf{E} \le \textrm{par deg } \mathbf{V}$.  Note that equality can hold only if $d' - d = 0$.  In that case, we have $E = V$.
\end{proof}

\begin{prp}
\label{5.2.2}
The rank of a morphism $\varphi: \mathbf{E} \rightarrow \mathbf{F}$ between two $\theta$-semistable parabolic bundles of equal slope is constant over all the fibers.
\end{prp}
\begin{proof}
It suffices to prove that the image sheaf, $\textrm{Im }\varphi$, is a vector subbundle of $F$.  Note that $\textrm{Im }\varphi$ has the structure of a parabolic bundle, since $\varphi$ is a morphism of parabolic bundles.  Let us denote this parabolic bundle bundle by $\mathbf{W}$.   Consider the saturation $V$ of $\textrm{Im }\varphi$ in $F$.  This is a vector subbundle of $F$ with parabolic structure on $V$ induced by that of $\mathbf{W}$.  Note that $\textrm{par deg } \mathbf{W} \le \textrm{par deg } \mathbf{V}$, by Lemma \ref{5.2.1}.  Since $\mu(\mathbf{E}) = \mu(\mathbf{F})$ and 
$$
\mu(\mathbf{E}) \le  \mu(\mathbf{W}) \le  \mu(\mathbf{V}) \le  \mu(\mathbf{F})
$$
we have that $\mu(\mathbf{W}) =  \mu(\mathbf{V})$.  The underlying vector bundles have the same rank, so $\textrm{par deg } \mathbf{W} = \textrm{par deg } \mathbf{V}$.  By Lemma \ref{5.2.1} this only happens when $\textrm{Im }\varphi = W = V$.  Therefore $\textrm{Im }\varphi$ is a vector subbundle of $F$.
\end{proof}

Since it is possible to reconstruct the Jordan form of an endomorphism $f$ of a vector space from the ranks of the operators $(f - \lambda)^a$, for $\lambda \in \mathbb{C}$ and $a \in \mathbb{Z}_{\ge 0}$, Proposition \ref{5.2.2} implies the following corollary: 

\begin{cor}
\label{5.2.3}
Let $\varphi_x$ be the vector space endomorphism on the fiber over $x \in X$ induced by the parabolic bundle endomorphism $\varphi$.  The conjugacy class  $\varphi_x$ is constant for all $x \in X$.
\end{cor}

Let $X = \mathbb{P}^1$ from now on.  Let $\textrm{Bun}_{D, w,\alpha}^{\theta, ss}(X)$ be the open substack of $\textrm{Bun}_{D, w,\alpha}(X)$ consisting of $\theta$-semistable parabolic bundles.  

\begin{proof}[Proof of Theorem \ref{5.0.1}]
The logic of the proof of this theorem will be the same as for Theorem \ref{4.0.1}.  We will be repeating key parts of the proof of Theorem \ref{4.0.1} for convenience.  

From now on, we will be assuming all parabolic bundles are $\theta$-semistable and all morphisms between parabolic bundles are morphisms of $\theta$-semistable parabolic bundles. Let $(\mathbf{E}, f)$ be a point of $\mathcal{N}^{ss}(D,w, \alpha)$.  Let $\mathbf{F} = \mathbf{ker} \ f$ and $\mathbf{G} = \mathbf{E}/ \mathbf{F}$.  It is easy to see that both $\mathbf{F}$ and $\mathbf{G}$ are $\theta$-semistable, with the same slope as $\mathbf{E}$.  We wish to prove this theorem by induction on the rank of the vector bundle $E$ (note that this is $\alpha_0$ in our notation).  It suffices to prove that for all $\beta$ and all $a$ we have:
$$
\textrm{dim } \mathcal{N}_{\beta, a}(D,w, \alpha) \le -\tilde{q}(\alpha),
$$
where $\mathcal{N}_{\beta, a}(D,w, \alpha)$ is a substack consisting of objects of $\mathcal{N}^{ss}(D,w, \alpha)$ of slope $a$ such that $\mathbf{F} \in \textrm{Bun}_{D,w, \beta}(X)$.  In order to accomplish this, consider the morphism
$$
\phi: \mathcal{N}_{\beta, a}(D,w, \alpha) \rightarrow \mathcal{N}^{ss}(D,w, \alpha - \beta),
$$
which is defined by sending $(\mathbf{E}, f)$ to
$(\mathbf{G}, f|_{\mathbf{G}})$ $\in \mathcal{N}^{ss}(D, w, \alpha - \beta)$, with corresponding restrictions on the arrows.  In this case, by induction, we get
$$
\textrm{dim } \mathcal{N}_{\beta, a}(D,w, \alpha) \le \textrm{dim } \mathcal{N}_{\beta,a}(D,w, \alpha)_x - \tilde{q}(\alpha - \beta) ,
$$
for some $x = (\mathbf{G}, g) \in \mathcal{N}^{ss}(D, w, \alpha - \beta)$.  We want to compute the dimension of the fiber $\mathcal{X} = \mathcal{N}_{\beta, a}(D,w, \alpha)_x$.  To help us do this, fix a parabolic bundle $\mathbf{V}$ with $\mu(\mathbf{V}) = a$ and define $\mathcal{P}_{\mathbf{V}}:= \mathcal{P}_{\mathbf{V}}(D,w,\alpha, a)$ to be the algebraic stack consisting of pairs
$\{\mathbf{W}, i: \mathbf{V} \hookrightarrow \mathbf{W} \}$, where $i$ is an inclusion of parabolic bundles and $\mathbf{W}$ is a parabolic bundle of weight type $(D,w)$, dimension vector $\alpha$, and slope $a$.

By Proposition \ref{5.2.2}, any morphism of $\theta$-semistable parabolic bundles of the same slope has constant rank as morphism of the underlying vector bundles.  Therefore, in the definition of $\mathcal{P}_{\mathbf{V}}$
the image of $\mathbf{V}$ under the inclusion $i$ is a parabolic subbundle of $\mathbf{W}$.  It follows that the deformations of $(\mathbf{W},i)$ are governed by the cohomology of the vector bundle $\mathscr{H}om_{\textrm{Par}}(\mathbf{W / V}, \mathbf{W})$.  Therefore, we can compute
$$
\dimn \mathcal{P}_{\mathbf{V}}(D,w,\alpha,a) = - \chi(\mathscr{H}om_{\textrm{Par}}(\mathbf{W/V}, \mathbf{W})).
$$
Compare this to the computations in Lemma \ref{4.3.2} and Lemma \ref{4.3.3}.

Let $\mathbf{F_1} = \mathbf{ker} \ g$ and let $\mathcal{X}' = \mathcal{P}_{\mathbf{F_1}}(D,w,\beta,a)$.  In this case, we have two morphisms $\psi_1: \mathcal{X} \rightarrow \textrm{Bun}_{D,w, \beta}^{\theta, ss}(X)$ and $\psi_2: \mathcal{X}' \rightarrow \textrm{Bun}_{D,w, \beta}^{\theta, ss}(X)$, where $\psi_1$ sends the pair $(\mathbf{E},f)$ to $\textbf{ker }f$ and likewise $\psi_2$ sends $(\mathbf{F},i)$ to $\mathbf{F}$. For $\mathbf{F} \in \textrm{Bun}_{D,w, \beta}(X)$, we have
\begin{align*}
& \textrm{dim } \mathcal{X}_{\mathbf{F}} = \textrm{dim } H^1 (\mathbb{P}^1, \mathscr{H}om_{\textrm{Par}}(\mathbf{F_1}, \mathbf{F}))\\
& \textrm{dim } \mathcal{X}'_{\mathbf{F}} = \textrm{dim } H^0 (\mathbb{P}^1, \mathscr{H}om_{\textrm{Par}}(\mathbf{F_1}, \mathbf{F})).
\end{align*}
Therefore, $\textrm{dim } \mathcal{X}_{\mathbf{F}}  = \textrm{dim } \mathcal{X}'_{\mathbf{F}} - \chi(\mathscr{H}om_{\textrm{Par}}(\mathbf{F_1}, \mathbf{F}))$.  We have $\textrm{dim } \mathcal{X} = \textrm{dim } \mathcal{X}' - \chi(\mathscr{H}om_{\textrm{Par}}(\mathbf{F_1}, \mathbf{F}))$. So, we obtain
$$
\textrm{dim } \mathcal{N}_{\beta,a}(D,w, \alpha) \le \textrm{dim } \mathcal{X}'  - \tilde{q}(\alpha - \beta) - \chi(\mathscr{H}om_{\textrm{Par}}(\mathbf{F_1}, \mathbf{F})).
$$

It follows from the formula above that 
$$\dimn \mathcal{X}' = - \chi(\mathscr{H}om_{\textrm{Par}}(\mathbf{F / F_1}, \mathbf{F})) = \chi(\mathscr{H}om_{\textrm{Par}}(\mathbf{F_1}, \mathbf{F})) -  \chi (\mathscr{E}nd_{\textrm{Par}}(\mathbf{F})),$$ 
which means
$$
\textrm{dim } \mathcal{N}_{\beta,a}(D,w, \alpha) \le -\chi (\mathscr{E}nd_{\textrm{Par}}(\mathbf{F})) - \tilde{q}(\alpha - \beta).
$$
Since $\chi (\mathscr{E}nd_{\textrm{Par}}(\mathbf{F})) = q(\beta)$ and $\tilde{q}(\alpha) \le \tilde{q}(\alpha - \beta) + \tilde{q}(\beta)$ (by Proposition \ref{4.3.1} b)), we can reduce this to
$$
\textrm{dim } \mathcal{N}_{\beta,a}(D,w, \alpha) \le -q(\beta) +\tilde{q}(\beta) -\tilde{q}(\alpha).
$$
The result follows from Proposition \ref{4.3.1} a). 
\end{proof}

\subsection{Stability for Quiver Representations and Stability for Parabolic Bundles}
\label{5.3}
In \cite{King1994} we find a stability condition for quiver representations similar to the one defined above for parabolic bundles.  More precisely, let $Q$ be a quiver with vertex set $I$. Let $\lambda \in \mathbb{R}^I$ be a collection of real numbers. A quiver representation $R \in \textrm{Rep}(Q, \alpha)$ is called $\lambda-$\textit{semistable} if $\lambda \cdot \alpha = 0$ and any subrepresentation $S$, with dimension vector $\beta$, satisfies  $\lambda \cdot \beta \ge 0$.  If the inequality is strict, it is called $\lambda-$\textit{stable}.

Now, let $\mathcal{E}$ be the category of bundles over $\mathbb{P}^1$ with weight type $(D,w)$, for which the duals to the underlying bundles are generated by global sections.  Additionally, let $\mathcal{Q}$ be the category of Kronecker-preinjective representations of the squid $S_{D,w}$ with injective arrows $c_{ij}$.

Note that the quiver representation semistability condition may be applied to representations of squids.  Specifically, for $\lambda = (\lambda_{\infty}, \lambda_0, \lambda_{ij})$ a squid representation of dimension $\alpha = (\alpha_{\infty},\alpha_{\infty} + \alpha_0,\alpha_{ij})$ is semistable if $\lambda \cdot \alpha = 0$, and if for any subrepresentation, with some dimension $\beta = (\beta_{\infty},\beta_{\infty} + \beta_0,\beta_{ij})$, we have $\lambda \cdot \beta \ge 0$.

In Section 5 of \cite{CB2004} Crawley-Boevey describes an equivalence of categories between $\mathcal{E}$ and $\mathcal{Q}$.  This is the special case of Theorem \ref{6.0.2} below when $T$ is a point and $N = 0$. 

\begin{prp}
\label{5.3.1}
Under this equivalence, $\theta$-semistable parabolic bundles $\mathbf{E}$, with weights $\theta = (\theta_{i1}, \dots, \theta_{iw_i})$ and parabolic degree $a$, correspond to $\lambda-$semistable squid representations, where $\lambda = (- a + 1 + \sum_l \theta_{l1},a-\sum_l \theta_{l1}, \theta_{ij+1} - \theta_{ij})$. 
\end{prp}

\begin{proof}
Let $S$ be the squid representation with dimension vector $\alpha = (\alpha_{\infty}, \alpha_{\infty} + \alpha_0, \alpha_{ij})$ corresponding to $\mathbf{E}$, and consider a subrepresentation $R \subset S$, with dimension vector $\beta = (\beta_{\infty}, \beta_{\infty} + \beta_0, \beta_{ij})$.  Note that under the equivalence of categories $R$ corresponds to a parabolic bundle $\mathbf{F} \subset \mathbf{E}$, which is contained in a parabolic subbundle $\mathbf{F} \subset \mathbf{F'} \subset \mathbf{E}$ of rank $\beta_0$ and degree $d'$.  This bundle is constructed from the saturation $F'$ of $F$.  Note that by Lemma \ref{5.2.1} we have $\textrm{par deg } \mathbf{F} \le \textrm{par deg } \mathbf{F'}$.  By semistability, we have that 

$$
\mu(\mathbf{F}) = \frac{-\beta_{\infty} + \sum_i (\beta_0 - \beta_{i1})\theta_{i1} + \sum_i \sum_j (\beta_{ij} - \beta_{ij-1})\theta_{ij}}{\beta_0} \le \mu(\mathbf{F}') \le a
$$
or
$$
-\beta_{\infty} + \sum_i (\beta_0 - \beta_{i1})\theta_{i1} + \sum_i \sum_j (\beta_{ij} - \beta_{ij-1})\theta_{ij} \le a\beta_0, 
$$
which may be rewritten as 
$$
-\beta_{\infty} + \beta_0(-a + \sum_i \theta_{i1}) + \sum_i \sum_j (\theta_{ij} - \theta_{ij+1})\beta_{ij} \le 0.
$$
This is the same as
$$
-\beta_{\infty}(1 - a + \sum_i \theta_{i1}) + (\beta_0+\beta_{\infty})(-a + \sum_i \theta_{i1}) + \sum_i \sum_j (\theta_{ij} - \theta_{ij+1})\beta_{ij} \le 0
$$
or
$$
\lambda \cdot \beta = \beta_{\infty}\lambda_{\infty} + (\beta_0+\beta_{\infty})\lambda_0 + \sum_i \sum_j \lambda_{ij} \beta_{ij} \ge 0.
$$
Since $\mathbf{E}$ has parabolic degree $a$, a similar argument shows that $\lambda \cdot \alpha = 0$.  It follows by definition that $S$ is $\lambda-$semistable.
\end{proof}

Now, conversely, consider $\lambda-$semistable squid representations, with Kronecker quiver representations of dimension $\alpha = (d, d + \alpha_0)$ and weights $\lambda = (-a + 1 + \sum_l \theta_{l1},a-\sum_l \theta_{l1}, \theta_{ij+1} - \theta_{ij})$ for $0 \le \theta_{ij} \le \theta_{ij+1} < 1$.  

\begin{prp}
\label{5.3.2}
Under the equivalence of categories, $\lambda-$semistable squid representations with Kronecker quiver dimension $(\alpha_{\infty} + N, \alpha_{\infty} + N +\alpha_0)$ and injective arrows $c_{ij}$ correspond to $\theta$-semistable parabolic bundles with rank $\alpha_0$, vector bundle degree $-\alpha_{\infty}-N$, weights $(\theta_{i1}, \dots, \theta_{iw_k})$, and parabolic degree $a$, for some $N \in \mathbb{Z}_{\ge 0}$.
\end{prp}
\begin{proof}
Let $\theta = (\theta_{ij})$ be weights that determine a stability condition for parabolic bundles over $\mathbb{P}^1$.  Consider the set $U$ of all unstable parabolic bundles with rank $\alpha_0$, underlying vector bundle of degree $-\alpha_{\infty}$.  Let $\mathbf{E} \in U$ and let $\mathbf{F} \subset \mathbf{E}$ be a maximal destabilizing parabolic subbundle.  We have:
$$
\frac{-\alpha_{\infty} + \sum_i \theta_{ij}(\alpha_{ij} - \alpha_{ij+1})}{\alpha_0} <\frac{\textrm{deg } F + \sum_i \theta_{ij}(\beta_{ij} - \beta_{ij+1})}{\beta_0},   
$$  
where $(\alpha_0, \alpha_{ij})$ and $(\beta_0, \beta_{ij})$ are the dimension vectors of $\mathbf{E}$ and $\mathbf{F}$.  Since $0 \le \beta_0 \le \alpha_0$, $0 \le \sum_i \theta_{ij}(\alpha_{ij} - \alpha_{ij+1})$ , and $\sum_i \theta_{ij}(\beta_{ij} - \beta_{ij+1}) \le k\alpha_0$, then $\textrm{deg } F$ is bounded, independent of $E$.  All subbundles of a $E$ have duals generated by global sections.  Therefore, the degree of a maximal destabilizing subbundle is bounded over all of $U$.  By Grothendieck's Theorem, the maximal destabilizing parabolic subbundles of bundles in $U$ have finitely many underlying vector bundle structures, up to isomorphism.  There are, likewise, finitely many isomorphism types of bundles over $\mathbb{P}^1$ with fixed rank, degree and dual generated by global sections.   Thus, there exists an $N \in \mathbb{Z}_{\ge 0}$ such that for any $E \in U$ and some maximal destabilizing subbundle $F\subset E$ we have $(E(-N)/F(-N))^* = (E/F)^*(N)$ is generated by global sections.  Note that since $E^*$ is generated by global sections, then $E(-N)^*$ is generated by global sections as well.  

Now, fix an unstable parabolic bundle $\mathbf{E}$ with parabolic degree $a$, vector bundle degree $-\alpha_{\infty}-N$, dimension vector $(\alpha_0, \alpha_{i,j})$, and weights $\theta = (\theta_{ij})$.  The choice of $N$ guarantees that there is a maximum destabilizing parabolic subbundle $\mathbf{F} \subset \mathbf{E}$ such that $(E/F)^*(N)$ is generated by global sections.  Let $\mathbf{F}$ have a vector bundle of degree $-\beta_{\infty}-N$ and the dimension vector $(\beta_0, \beta_{ij})$.  We have that, $\mathbf{E}$ corresponds to a squid representation $S$, with dimension vector $(\alpha_{\infty}+N, \alpha_{\infty}+N+\alpha_0, \alpha_{ij})$ under the equivalence of categories.  Furthermore, $\mathbf{F}$ corresponds to a squid subrepresentation $R \subset S$ with dimension vector $(\beta_{\infty}+N, \beta_{\infty}+N+\beta_0, \beta_{ij})$.  Since, $\mathbf{F}$ is unstable, we have:

$$
\frac{-\beta_{\infty} - N + \sum_i \beta_0\theta_{i1} + \sum_i \sum_j \theta_{ij} (\beta_{ij} - \beta_{ij+1})}{\beta_0} > a
$$
which implies
$$
(\beta_{\infty}+N)(- a + 1 + \sum_i \theta_{i1}) + (\beta_0+\beta_{\infty}+N)(a-\sum_i \theta_{i1}) + \sum_i \sum_j (\theta_{ij+1} - \theta_{ij})\beta_{ij} < 0
$$

or

$$
(\beta_{\infty}+N)\lambda_{\infty} + (\beta_0+\beta_{\infty}+N)\lambda_0 + \sum_{i,j} \lambda_{ij}\beta_{ij} < 0
$$
Therefore, $S$ is not $\lambda$-stable.  It follows that $\lambda$-semistable squid representations correspond to $\theta$-semistable parabolic bundles.     
\end{proof}

By putting together Propositions \ref{5.3.1} and \ref{5.3.2} we obtain the following corollary:

\begin{cor}
\label{5.3.3}
There exists a number $N \in \mathbb{Z}_{\ge 0}$ such that $\lambda-$semistable squid representations with Kronecker quiver dimension $(\alpha_{\infty} + N, \alpha_{\infty} + N +\alpha_0)$ and injective arrows $c_{ij}$ correspond to $\theta$-semistable parabolic bundles with rank $\alpha_0$, vector bundle degree $-\alpha_{\infty}-N$, weights $(\theta_{i1}, \dots, \theta_{iw_i})$, and parabolic degree $0$, under the equivalence of categories.
\end{cor}

%% file: quivers_and_parabolic_bundles.tex
\subsection{Outline}
\label{6.0}
In \cite{CB2004}, Crawley-Boevey provides an equivalence between the category of parabolic bundles of weight type $(D,w)$ over $\mathbb{P}^1$, with dual underlying bundle generated by global sections, and the category of Kronecker-preinjective squid representations with injective arrows $c_{ij}$ (see Section \ref{3.4}).  In this section, we will make similar statements concerning the category of families of parabolic bundles and the category of families of squid representations.  We will use these to prove the representability of several functors related to moduli spaces of parabolic parabolic bundles over $\mathbb{P}^1$.

Let $T$ be a scheme over an algebraically closed field $K$.  Let $\mathscr{V}$ be the category of families over $T$ of vector bundles $E$ over $\mathbb{P}^1$, such that $E^*$ is generated by global sections, and let $\mathscr{R}$ be the category of families over $T$ of preinjective Kronecker quiver representations.  We have:

\begin{thm}
\label{6.0.1}
The categories $\mathscr{V}$ and $\mathscr{R}$ are equivalent.
\end{thm}

For fixed $\alpha = (\alpha_{\infty}, \alpha_{\infty} + \alpha_{0})$, this allows us to prove that the space of preinjective Kronecker quiver representations $KI(\alpha)$ represents a moduli functor for vector bundles over $\mathbb{P}^1$ together with some rigidity conditions.

Let $B$ be a vector bundle over $KI(\alpha)$.  We can define a flag bundle $Fl(B)_i$ over $KI(\alpha)$ of flags of type $(\alpha_0, \alpha_{i1}, \dots, \alpha_{iw_i})$ by gluing flag varieties (using the transitions functions of $B$).  Let $Fl(B)$ be the fibered product of $Fl(B)_i$ for $1 \le i \le k$ over $KI(\alpha)$.  We can prove that $Fl(B)$ represents a moduli functor for parabolic bundles over $\mathbb{P}^1$ together with some rigidity conditions.

Let $\mathscr{P}(D,w)$ be the category of families over $T$ of parabolic bundles $\mathbf{E}$ over $\mathbb{P}^1$ of weight type $(D,w)$, with $E^*$ generated by global sections.  Let $\mathscr{S}(D,w)$ be the category of families over $T$ of Kronecker-preinjective squid representations with injective arrows $c_{ij}$ (see section \ref{3.4}).  We also prove a theorem analogous to Theorem \ref{6.0.1} (cf. Lemma 5.5 in \cite{CB2004}):
\begin{thm}
\label{6.0.2}
The categories $\mathscr{P}(D,w)$ and $\mathscr{S}(D,w)$ are equivalent.
\end{thm}

Let $\alpha = (\alpha_{\infty}, \alpha_{0} + \alpha_{\infty}, \alpha_{ij})$.  Let $KS(D, w, \alpha)$ be the space of Kronecker-preinjective squid representations dimension vector $\alpha$, such that the maps corresponding to the arrows $c_{ij}$ are injective. Theorem \ref{6.0.2} allows us to show that $KS(D, w, \alpha)$ represents another moduli functor of parabolic bundles together with additional rigidity conditions.  Note that these results parallel the correspondence obtained in Section 5 of \cite{CB2004}.

As an application of the above results, we examine the case of parabolic structures on a trivial vector bundle over $\mathbb{P}^1$.  This lets us prove Theorem 1.2.2 as a consequence of the very good property for the moduli stack of parabolic bundles (Theorem 1.2.1) over $\mathbb{P}^1$ or of the very good property for the quotient stack associated to representations of a star-shaped quiver (Theorem \ref{3.6.1}).

\subsection{Moduli Functor: parabolic bundles and flag bundles}
\label{6.1}
In this and the next section, we will use $\langle \quad \rangle$ to denote the isomorphism class of the collection of enclosed objects.  All the schemes we consider from now on will be schemes of finite type. 
Let 
\begin{align*}
& p: T \times_{K} \mathbb{P}^1 \rightarrow T
& \pi: T \times_{K} \mathbb{P}^1 \rightarrow \mathbb{P}^1
\end{align*}
be the two natural projections.  Let $\mathscr{V}$ be the category of vector bundles $E$ over $T \times_{K} \mathbb{P}^1$ such that $E^*|_{\{x\}\times \mathbb{P}^1}$ is generated by global sections for all $x \in T$.  The morphisms of $\mathscr{V}$ are just vector bundle morphisms.  

Let $\mathcal{V}$ and $\mathcal{W}$ be vector bundle over $T$, and let $\Psi_0,\Psi_1$ be morphisms of vector bundles from $\mathcal{V}$ to $\mathcal{W}$ such that on every fiber over $x \in T$ all linear combinations $\lambda_0\Psi_0(x) + \lambda_1\Psi_1(x)$ for $(\lambda_0: \lambda_1) \in \mathbb{P}^1$ are surjective.  Let $\mathscr{R}$ be the category whose objects are four-tuples $(\mathcal{V}, \mathcal{W},\Psi_0,\Psi_1)$.  A morphism in $\mathscr{R}$ between $(\mathcal{V}, \mathcal{W},\Psi_0,\Psi_1)$ and $(\mathcal{V}^{'}, \mathcal{W}^{'},\Psi_0^{'},\Psi_1^{'})$ consists of a pair $(f,g)$ of vector bundle morphisms $f: \mathcal{V} \rightarrow \mathcal{V}^{'}$ and $g: \mathcal{W} \rightarrow \mathcal{W}^{'}$ such that $g \circ \Psi_i = \Psi_i^{'} \circ f$ for $i = 1,2$.  Note that the objects of $\mathscr{R}$ are families of preinjective Kronecker quiver representations but not necessarily in coordinate spaces.

We can now proceed with:
\begin{proof}[Proof of Theorem \ref{6.0.1}]
Let $E \in \mathscr{V}$ be a family of vector bundles over $\mathbb{P}^1$, which are parametrized by $T$.  Define two vector bundles over $T$ by 
\begin{align*}
& \mathcal{V} = p_*(E^*)^*\\
& \mathcal{W} = p_*(E^*(-1))^*,
\end{align*}
where $E(-1) = E \otimes \pi^*(\mathcal{O}(-1))$.  Note that these are indeed vector bundles, by the Cohomology and Base Change theorem (Theorem 12.11 in \cite{Ha1977}).  

Let $1$ and $-z$ be the generators of $\pi^*(\mathcal{O}(1))$ corresponding to the two natural global sections of $\mathcal{O}(1)$.  Since $E^* = E^*(-1) \otimes \pi^*(\mathcal{O}(1))$, we can write two inclusions:
\begin{align*}
& \psi_0: E^*(-1) \rightarrow E^*\\
& \psi_1: E^*(-1) \rightarrow E^*
\end{align*} 
corresponding to these generators.  Let us denote the morphisms induced by $\psi_0$ and $\psi_1$ from $\mathcal{V}$ to $\mathcal{W}$ by $\Psi_0$ and $\Psi_1$.  Note that $\psi_0, \psi_1$ are injections from $E^*(-1)$ to $E^*$, defined by the sections of $\mathcal{O}(1)$, so the morphisms they induce from $\mathcal{W}$ to $\mathcal{V}$ are injective as morphisms of vector bundles.  Moreover, their linear combinations are injective.  Reducing to the fiber, it follows that $\lambda_0\Psi_0(x) + \lambda_1\Psi_1(x)$ is surjective for $(\lambda_0:\lambda_1) \in \mathbb{P}^1$.  This means $(\mathcal{V}, \mathcal{W}, \Psi_0, \Psi_1)$ is a family of preinjective Kronecker quiver representations. 

Let $E_1$ and $E_2$ be objects of $\mathscr{V}$ corresponding to the families $(\mathcal{V}_1, \mathcal{W}_1,\Psi_{1,0},\Psi_{1,1})$ and $(\mathcal{V}_2, \mathcal{W}_2,\Psi_{2,0},\Psi_{2,1})$, respectively.  If $f: E_1 \rightarrow E_2$ is a morphism of objects of $\mathscr{V}$, then it induces morphisms $f_1: \mathcal{V}_1 \rightarrow \mathcal{V}_2$ and $f_2: \mathcal{W}_1 \rightarrow \mathcal{W}_2$, for the corresponding vector bundles over $T$.  Moreover, $f$ induces a morphism $f': E_1(-1) \rightarrow E_2(-1)$, so it follows that $\bar{f}\psi_{1,0} = \psi_{2,0}\bar{f'}$, where $\bar{f}$ and $\bar{f'}$ are induced morphisms on $E_1^*$ and $E_1^*(-1)$ and 
\begin{align*}
&\psi_{1,0}: E_1^*(-1) \rightarrow E_1^*,\\
&\psi_{2,0}: E_2^*(-1) \rightarrow E_2^*
\end{align*}
are inclusions described above corresponding to the generator $1$ of $\pi^*(\mathcal{O}(1))$.
Analogously, we have $\bar{f}\psi_{1,1} = \psi_{2,1}\bar{f'}$, where
\begin{align*}
&\psi_{1,1}: E_1^*(-1) \rightarrow E_1^*,\\
&\psi_{2,1}: E_2^*(-1) \rightarrow E_2^*
\end{align*}
are inclusions described above corresponding to the generator $z$ of $\pi^*(\mathcal{O}(1))$.  We therefore have that $(f_1,f_2)$ is a well-defined morphism of families of quiver representations.

We define the functor $R$ by:
\begin{align*}
&R(E) = (\mathcal{V}, \mathcal{W}, \Psi_0, \Psi_1)\\
&R(f) = (f_1,f_2).
\end{align*}

Conversely, consider the family of preinjective Kronecker quiver representations $(\mathcal{V}, \mathcal{W}, \Psi_0, \Psi_1)$. Since $\mathcal{V}$ and $\mathcal{W}$ are vector bundles over $T$ and $\Psi_0, \Psi_1$ are morphisms between them, then we can define two morphisms of vector bundles over $T \times \mathbb{P}^1$.  Namely:
$$
\phi_0,\phi_1: p^*(\mathcal{V}) \rightarrow p^*(\mathcal{W}),
$$
which are induced by $\Psi_0$ and $\Psi_1$.  We define a morphism:
\begin{align*}
& \phi: p^*(\mathcal{V}) \rightarrow p^*(\mathcal{W})(1) \\
& v \mapsto \phi_0(v)\otimes z - \phi_1(v) \otimes 1,
\end{align*}
where $p^*(\mathcal{W})(1) = p^*(\mathcal{W}) \otimes \pi^*(\mathcal{O}(1))$.  It follows from preinjectivity that $\phi$ is surjective and that $E  = \textrm{ker }\phi$ is a vector bundle.  Therefore we have the exact sequence:
$$
0 \rightarrow E \rightarrow p^*(\mathcal{V}) \rightarrow p^*(\mathcal{W})(1) \rightarrow 0,
$$ 
which can be dualized to give us 
$$
0 \rightarrow p^*(\mathcal{W}^*)(-1) \rightarrow p^*(\mathcal{V}^*)\rightarrow E^* \rightarrow 0. 
$$
If we restrict this to $\{x\} \times \mathbb{P}^1$, then we get a surjection from $p^*(\mathcal{V}^*)|_{\{x\} \times \mathbb{P}^1}$ to $E^*|_{\{x\} \times \mathbb{P}^1}$.  However, $p^*(\mathcal{V}^*)|_{\{x\} \times \mathbb{P}^1}$ is trivial and therefore $E^*|_{\{x\} \times \mathbb{P}^1}$ is generated by global sections for all $x \in T$.  

Let $(\mathcal{V}_1, \mathcal{W}_1,\Psi_{1,0},\Psi_{1,1})$ and $(\mathcal{V}_2, \mathcal{W}_2,\Psi_{2,0},\Psi_{2,1})$ be families of quiver representations corresponding to objects $E_1$ and $E_2$ of $\mathscr{V}$ respectively.  Let $(f_1, f_2)$ be a morphism of the families of quiver representations.  This means we have $f_1: \mathcal{V}_1 \rightarrow \mathcal{V}_2$ and $f_2: \mathcal{W}_1 \rightarrow \mathcal{W}_2$, which commute with $\Psi_{1,0}, \Psi_{1,1}$ and $\Psi_{1,1}, \Psi_{2,1}$, respectively.  It follows that there are induced morphisms $f: p^*(\mathcal{V}_1) \rightarrow p^*(\mathcal{V}_2)$ and $f': p^*(\mathcal{W}_1(1)) \rightarrow p^*(\mathcal{W}_2(1))$.  Let
\begin{align*}
&\phi^1: p^*(\mathcal{V}_1) \rightarrow p^*(\mathcal{W}_1)(1)\\
&\phi^2: p^*(\mathcal{V}_2) \rightarrow p^*(\mathcal{W}_2)(1),
\end{align*} 
be induced by $\Psi_{1,0}, \Psi_{1,1}$ and $\Psi_{1,1}, \Psi_{2,1}$, as above.  Since we have $f' \phi^1 = \phi^2 f$, then $f$ maps $\textrm{ker } \phi^1$ to $\textrm{ker } \phi^2$.  Therefore, $f: E_1 \rightarrow E_2$ is a well-defined morphism of vector bundles.

It follows that we can define a functor $V$ by:
\begin{align*}
&V(\mathcal{V}, \mathcal{W}, \Psi_0,\Psi_1) = E\\ 
&V(f_1,f_2) = f
\end{align*}
Let $E \in \textrm{Ob}(\mathscr{V})$.  Consider $VR(E)$.  It is part of the following short exact sequence:
$$
0 \rightarrow VR(E) \rightarrow p^*(p_*(E^*))^* \rightarrow p^*(p_*(E^*(-1)))^*(1) \rightarrow 0
$$
Similarly, 
$$
0 \rightarrow E \rightarrow p^*(p_*(E^*))^* \rightarrow p^*(p_*(E^*(-1)))^*(1) \rightarrow 0.
$$
Indeed, consider the morphism $E \rightarrow p^*(p_*(E^*))^*$ induced by the pairing between $E$ and $E^*$ and let $p^*(p_*(E^*))^* \rightarrow p^*(p_*(E^*(-1)))^*(1) \rightarrow 0$ be as before.  We have: 
$$p^*(p_*(E^*))|_{\{x\}\times \mathbb{P}^1} = (p_*E^*)_x\otimes \mathcal{O}_{\{x\}\times \mathbb{P}^1} = H^0(E^*|_{\{x\}\times \mathbb{P}^1})\otimes \mathcal{O}_{\{x\}\times \mathbb{P}^1}$$
by the Cohomology and Base Change theorem.  It follows that for each restriction the morphism $p^*(p_*(E^*))|_{\{x\}\times \mathbb{P}^1} \rightarrow E^*|_{\{x\}\times \mathbb{P}^1}$ is surjective.  Therefore, the morphism of the duals $p^*(p_*(E^*)) \rightarrow E^*$ is surjective, so $0 \rightarrow E \rightarrow p^*(p_*(E^*))^*$ is exact, and moreover, $E$ is a subbundle of $p^*(p_*(E^*))^*$.  Since the morphism $p^*(p_*(E^*))^* \rightarrow p^*(p_*(E^*(-1)))^*(1)$ is induced by the inclusion of $E^*(-1)$ into $E^*$, then we have that the image of $E$ lies in the kernel of this morphism.  The kernel and $E$ are vector subbundles of the same rank.  Therefore, the image of $E$ coincides with the kernel, so the sequence is exact.  It follows that $VR(E) \cong E$, and that the identity functor on $\mathscr{V}$ is naturally isomorphic to $VR$. 

Conversely, we have 
$$
RV(\mathcal{V},\mathcal{W},\Psi_0,\Psi_1) = (p_*(E^*), p_*(E^*(-1)), \theta_0, \theta_1),
$$
where $E$ comes from the exact sequence 
$$
0 \rightarrow E \rightarrow p^*V \rightarrow p^*(W)(1) \rightarrow 0.
$$
Note that by dualizing we obtain
$$
0 \rightarrow p^*(W^*)(-1) \rightarrow p^*(V^*) \rightarrow E^* \rightarrow 0.
$$

We can write the following long exact sequence for the direct image $p_*$:
$$
0 \rightarrow p_*(p^*(W^*)(-1)) \rightarrow p_*p^*(V^*) \rightarrow p_*(E^*) \rightarrow R^1p_*(p^*(W^*)(-1)) \rightarrow \cdots .
$$
It follows from the Projection Formula that 
\begin{align*}
& p_*(p^*(W^*)(-1)) = \mathcal{W}^* \otimes H^0(\mathbb{P}^1, \mathcal{O}(-1)) = 0 \\
& R^1p_*(p^*(W^*)(-1)) = \mathcal{W}^* \otimes H^1(\mathbb{P}^1, \mathcal{O}(-1)) = 0,
\end{align*}
so we have that $p_*(E^*)^* = p_*p^*(V^*)^*  = V$.  Similarly, using the Projection Formula, we obtain the long exact sequence
$$
0 \rightarrow p_*(p^*(W^*)(-2)) \rightarrow 0 \rightarrow p_*(E^*(-1)) \rightarrow R^1p_*(p^*(W^*)(-2)) \rightarrow 0
$$
from the short exact sequence
$$
0 \rightarrow p^*(W^*)(-2) \rightarrow p^*(V^*)(-1) \rightarrow E^*(-1) \rightarrow 0.
$$
It follows that 
$p_*(E^*(-1))^* = R^1p_*(p^*(W^*)(-2))^* = W \otimes H^1 (\mathbb{P}^1, \mathcal{O}(-2))^* = W$.  Furthermore, since $\theta_0, \theta_1$ are induced by $\Psi_0, \Psi_1$, then $RV(\mathcal{V},\mathcal{W},\Psi_0,\Psi_1)$ is isomorphic to $(p_*(E^*), p_*(E^*(-1)), \theta_0, \theta_1)$.  This defines a pair of mutually inverse natural transformations between the identity functor on $\mathscr{R}$ and $RV$.  It follows that the two functors are isomorphic.  Therefore, we have that $\mathscr{V}$ and $\mathscr{R}$ are equivalent.
\end{proof} 

Note in the subsequent definition of the moduli functor, and all of the following moduli functor definitions, we will define the functor on the objects of the category of schemes and assume that the functor is defined naturally on morphisms between schemes.
\begin{defn}
\label{6.1.1}
Let $E$ be as before but fix the degree and rank of each restriction $E|_{\{x\} \times \mathbb{P}^1}$ to be $d=-\alpha_{\infty}$ and $\alpha_0$, respectively.  Let us define a functor $F$, from the category of schemes over $K$ to the category of sets as $F(T) = \left\langle (E,s,t)\right\rangle$, where 
\begin{itemize}
\item $E$ \textrm{ is a vector bundle on } $T\times \mathbb{P}^1$,
\item $p_*(E^*)$ \textrm{ and } $p_*(E^*(-1))$ \textrm{ are trivial vector bundles},
\item $s: \mathcal{O}_T^{\alpha_0+\alpha_{\infty}} \simeq p_*(E^*)$,
\item $t: \mathcal{O}_T^{\alpha_{\infty}} \simeq p_*(E^*(-1))$.
\end{itemize}
\end{defn}

\begin{thm}
\label{6.1.2}
The moduli functor $F$ is represented by the space $KI(\alpha)$ of preinjective Kronecker quiver representations in the standard coordinate spaces.
\end{thm}
\begin{proof}
Fix a test scheme $T$, with $F(T) = \{\textrm{iso. classes of } (E, s: \mathcal{O}^{\alpha_0} \cong p_*(E^*), t: \mathcal{O}^{\alpha_{\infty}} \cong p_*(E^*(-1))\}$.  The construction of the functor $R$ in the proof of Theorem \ref{6.0.1} determines a family $(\mathcal{V}, \mathcal{W},\Psi_0, \Psi_1)$ over $T$ of preinjective Kronecker quiver representations from the vector bundle $E$.  Indeed, $\mathcal{V} = p_*(E^*)^* $ and $\mathcal{W} = p_*(E^*(-1))^*$, so $s,t$ identify $\mathcal{V},\mathcal{W}$ with trivial vector bundles on $T$, and $\Psi_0, \Psi_1$ with morphisms of trivial vector bundles on $S$.  It is evident that the rank of $\mathcal{V}$ is $\alpha_0 + \alpha_{\infty}$ and the rank of $\mathcal{W}$ is $\alpha_{\infty}$.  In other words, an element of $F(T)$ determines a morphism $\varphi: T \rightarrow KI(\alpha)$.  Similarly, it follows from the construction of the functor $R$ that $\varphi: T \rightarrow KI(\alpha)$ determines an element of $F(T)$.  This defines a pair of morphisms:
\begin{align*}
& \eta_T:  F(T) \rightarrow \textrm{Hom}(T,KI(\alpha))\\
& \rho_T:  \textrm{Hom}(T,KI(\alpha)) \rightarrow F(T).
\end{align*}
These are natural transformation between the functors $F$ and $\underline{KI(\alpha)}$, the functor of points for $KI(\alpha)$.  It follows from Theorem \ref{6.0.1} that $\eta_T$ and $\rho_T$ are mutually inverse, so the functors are isomorphic.  Therefore, $F$ is represented by $KI(\alpha)$.
\end{proof}

We can prove a statement similar to Theorem \ref{6.1.2} for parabolic bundles. Indeed, let $D, w, \alpha_0, \alpha_{ij}$ be as before and let $\alpha = (\alpha_0,\alpha_{ij})$. 

\begin{defn}
\label{6.1.3}
Let $E$ be as in Definition \ref{6.1.1}.  Let us define a functor $F'$, from the category of schemes over $K$ to the category of sets as $F'(T) = \left\langle (E,E^{i,j},s,t)\right\rangle$, where 
\begin{itemize}
\item $E$ \textrm{ is a vector bundle on } $T\times \mathbb{P}^1$,
\item $p_*(E^*)$ \textrm{ and } $p_*(E^*(-1))$ \textrm{ are trivial vector bundles},
\item $s: \mathcal{O}_T^{\alpha_0+\alpha_{\infty}} \simeq p_*(E^*)$,
\item $t: \mathcal{O}_T^{\alpha_{\infty}} \simeq p_*(E^*(-1))$,
\item $E|_{T \times \{x_i\}} \supset E^{i,1}\supset \dots \supset E^{i,w_i-1} \supset E^{i,w_i} = 0$ are filtrations by vector subbundles of fixed ranks $\textrm{rk } E^{i,j} = \alpha_{ij}$.
\end{itemize}
\end{defn}

Let $B$ be the universal family of vector bundles over $\mathbb{P}^1$ given by $F(KI(\alpha))$, and let $B_i = B|_{KI(\alpha) \times \{x_i\}}$.  We can define the scheme $$Fl(B) = Fl(B)_1 \times_{KI(\alpha)} \cdots \times_{KI(\alpha)} Fl(B)_k, $$ 
where $Fl(B)_i$ 
is a flag bundle for flags of type $(\alpha_{ij})$ over $KI(\alpha)$.  That is, given a trivialization 
$\{U_l^i,\psi_l^i\}_l$ of $B_i$, we can construct a scheme $U_l^i \times \textrm{Fl}(\alpha)$ for each $l$, where $\textrm{Fl}(\alpha)$ is the space of flags of type $(\alpha_{ij})$ in the standard coordinate space $K^r$.  The transition functions for $B_i$ glue the schemes $U_l \times \textrm{Fl}(\alpha)$ into a scheme $Fl(B)_i$.  It follows that there is a morphism $Fl(B)_i \rightarrow KI(\alpha)$ for each $i$, such that the fiber at each point is a flag of type $(\alpha_{ij})$.  Note that this means there is a morphism $Fl(B) \rightarrow KI(\alpha)$, such that the fiber at each point is a collection of $k$ flags. 

\begin{thm}
\label{6.1.4}
The moduli functor $F'$ is represented by $Fl(B)$.
\end{thm}
\begin{proof}
Fix a test scheme $T$.  By Theorem \ref{6.1.2}, an element of $F'(T)$ defines a morphism $\epsilon: T \rightarrow KI(\alpha)$, such that the vector bundle $E$ in that element is the pullback of $B$ along $\epsilon$.  It follows that each $B_i$ pulls back to $E^i = E|_{T \times \{x_i\}}$.  Therefore, we have that the flag $(E^{i,w_i})_y \subset \cdots \subset (E^{i,1})_y \subset (E^i)_y$ in the fiber of $(E^i)_y$ is equal to the flag in the fiber of $Fl(B)$ at $\epsilon(y)$ for all $y \in T$.  This means, the morphism that sends each point $y \in T$ to the flag $(E^{i,w_i})_y \subset \cdots \subset (E^{i,1})_y \subset (E^i)_y$ is a well defined morphism $T \rightarrow Fl(B)_i$.  Thus, combining these morphisms for each $i$ together with $\epsilon$, we have that an element in $F'(T)$ defines a morphism $T \rightarrow Fl(B)$.  Conversely, given a morphism $T \rightarrow Fl(B)$, we can compose it with the morphism $Fl(B) \rightarrow KI(\alpha)$ to get a morphism $\epsilon: T \rightarrow KI(\alpha)$.  By Theorem \ref{6.1.2}, this defines an isomorphism class $$(E, s: \mathcal{O}^{\alpha_0 + \alpha_{\infty}} \cong p_*(E^*), t: \mathcal{O}^{\alpha_{\infty}} \cong p_*(E^*(-1)).$$  
Note that the individual morphisms $T \rightarrow Fl(B)_i$ define filtrations by vector bundles $E^{i,w_i} \subset \cdots \subset E^{i,1} \subset E^i = E|_{T \times \{x_i\}}$ over $T$, for each $i$.  Therefore, we get an element of $F'(T)$.  It follows by construction and Theorem \ref{6.1.2} that we have a pair of mutually inverse natural transformations between $F'(B)$ and $\underline{Fl(B)}$, the functor of points for $Fl(B)$.  Therefore, $F'$ is represented by $Fl(B)$. 
\end{proof}

We can see that the the points of $Fl(B)$ can be thought of as isomorphism classes of parabolic bundles over $\mathbb{P}^1$ with fixed weight type $(D,w)$, fixed dimension vector $\alpha$, with an underlying vector bundle of degree $d$, such that its dual is generated by global sections. 

\begin{defn}
\label{6.1.5}
Let $E$ be as in Definition \ref{6.1.1}, and let $N \in \mathbb{Z}_{\ge 0}$.  We can generalize $F'$ by defining the following functor from the category schemes over $K$ to the category sets: $F''(T) = \left\langle (E,E^{i,j},s,t)\right\rangle$, where 
\begin{itemize}
\item $E$ \textrm{ is a vector bundle on } $T\times \mathbb{P}^1$,
\item $p_*(E^*(N))$ \textrm{ and } $p_*(E^*(N-1))$ \textrm{ are trivial vector bundles},
\item $s: \mathcal{O}_T^{(N+1)\alpha_0 + \alpha_{\infty}} \simeq p_*(E^*(N))$,
\item $t: \mathcal{O}_T^{N\alpha_0 + \alpha_{\infty}} \simeq p_*(E^*(N-1))$,
\item $E|_{T \times \{x_i\}} \supset E^{i,1}\supset \dots \supset E^{i,w_i-1} \supset E^{i,w_i} = 0$ are filtrations by vector subbundles of fixed ranks $\textrm{rk } E^{i,j} = \alpha_{ij}$.
\end{itemize}
Here $E^*(N) = E^*\otimes \pi(\mathcal{O}(N))$.
\end{defn}

It is clear that analogues of Theorem \ref{6.1.2} and Theorem \ref{6.1.3} hold in this case.  Therefore, we obtain:
\begin{cor}
\label{6.1.6}
The functor $F''$ is representable.
\end{cor}
By introducing additional rigidity, we can define a moduli space of parabolic bundles over $\mathbb{P}^1$ in terms of the squid representations defined in Section \ref{3.4}.

\subsection{Moduli functor: parabolic bundles and squids}
\label{6.2}
Let $E, p, \pi, T$ be as in the previous section.  Let $(D,w)$ be a parabolic bundle weight type (see Section 1.2).  Let $\mathscr{P}(D,w)$ be the category of vector bundles $E$ over $T \times_{K} \mathbb{P}^1$ such that $E^*|_{\{x\}\times \mathbb{P}^1}$ is generated by global sections for all $x \in T$, together with filtrations 
$$E|_{T \times \{x_i\}} = E^{i,0} \supset E^{i,1} \supset \cdots \supset E^{i,w_i} = 0, $$
for $1 \le i \le k$.  The morphisms of $\mathscr{P}(D,w)$ are vector bundle morphisms such that map filtrations to each other.  We can think of $\mathscr{P}(D,w)$ as the category of families over $T$ of parabolic bundles of weight type $(D,w)$ over $\mathbb{P}^1$, such that the dual to the underlying bundle is generated by global sections.  

Let $\mathcal{V}$ and $\mathcal{W}$ be vector bundle over $T$, and let $\Psi_0,\Psi_1$ be morphisms of vector bundles from $\mathcal{V}$ to $\mathcal{W}$ such that on every fiber over $x \in T$ all linear combinations $\lambda_0\Psi_0(x) + \lambda_1\Psi_1(x)$ for $(\lambda_0: \lambda_1) \in \mathbb{P}^1$ are surjective.  Let $\mathcal{V}_{ij}$ be vector bundles over $T$, for $1 \le i \le k$ and $1 \le j \le w_i-1$.  Let $C_{ij}: \mathcal{V}_{ij} \rightarrow \mathcal{V}_{ij-1}$ be injective morphisms of vector bundles such that $(\lambda_{i0}\Psi_0(x) + \lambda_{i1}\Psi_1(x))C_{i1}(x) = 0$ in the fiber over each $x \in T$, where $x_i = (\lambda_{i0}:\lambda_{i1})$ and $\mathcal{V}_{i0} = \mathcal{V}$. 

Let $\mathscr{S}(D,w)$ be a category where objects are collections $(\mathcal{V}, \mathcal{W}, \mathcal{V}_{ij}, \Psi_0,\Psi_1, C_{ij})$ and morphisms between $(\mathcal{V}, \mathcal{W}, \mathcal{V}_{ij}, \Psi_0, \Psi_1, C_{ij})$ and $(\mathcal{V}^{'}, \mathcal{W}^{'},\mathcal{V}_{ij}^{'}, \Psi_0^{'},\Psi_1^{'}, C_{ij}^{'})$ consists of a collection $(f,g, h_{ij})$ of vector bundle morphisms 
\begin{itemize}
\item[] $f: \mathcal{V} \rightarrow \mathcal{V}^{'}$ 
\item[] $g: \mathcal{W} \rightarrow \mathcal{W}^{'}$
\item[] $h_{ij}: \mathcal{V}_{ij} \rightarrow \mathcal{V}_{ij}^{'}$.
\end{itemize}
 such that: 
\begin{itemize} 
\item[] $g \circ \Psi_0 = \Psi_0^{'} \circ f$
\item[] $g \circ \Psi_1 = \Psi_1^{'} \circ f$
\item[] $h_{ij-1} \circ C_{ij} = C_{ij}^{'} \circ h_{ij}$ for $1 \le i \le k$ and $2 \le j \le w_i-1$  
\item[] $f \circ C_{i1} = C_{i1}^{'} \circ h_{i1}$ for $1 \le i \le k$.  
\end{itemize} 
 Note that the objects of $\mathscr{S}(D,w)$ are families of Kronecker-preinjective squid representations but not necessarily in coordinate spaces.

\begin{proof}[Proof of Theorem \ref{6.0.2}]

Let $(E, E_{ij})$ be an object in $\mathscr{P}(D,w)$.  By Theorem \ref{6.0.1}, we can use $E$ to construct a family $(\mathcal{V}, \mathcal{W},\Psi_0, \Psi_1)$ over $T$ of preinjective Kronecker quiver representations. In this construction, $\mathcal{V} = p_*(E^*)^*$ and $\mathcal{W} = p_*(E^*)^*$.  Consider the induced morphism $p_*(E^*) \rightarrow E^*|_{T\times\{x_i\}}$.  This morphism is surjective on the fibers, as the fiber of $p_*(E^*)$ at $y \in T$ is $H^0(\{y\}\times \mathbb{P}^1, E^*|_{\{y\}\times \mathbb{P}^1})$ and the fiber of $E^*|_{T\times\{x_i\}}$ at $y$ is $E^*_{\{y\}\times \{x_i\}}\otimes K(\{y\}\times \{x_i\})$, where $K(\{y\}\times \{x_i\})$ is the residue field at the point $\{y\}\times \{x_i\}$.  Therefore, we have the exact sequence $0 \rightarrow E|_{T\times\{x_i\}} \rightarrow \mathcal{V}$, where $E|_{T\times\{x_i\}}$ is a vector subbundle of $\mathcal{V}$.  Since the morphisms $\Psi_0$ and $\Psi_1$ are induced by the two inclusion $E^*(-1) \rightarrow E^*$, we have that $E|_{T\times\{x_i\}}$ lies in the kernel of $\lambda_{i0} \Psi_0 + \lambda_{i1} \Psi_1$.  Since the kernel and
$E|_{T\times\{x_i\}}$ are vector bundle of the same rank, the two must coincide.  

It follows that the filtration given by $E^{i,j}$ defines vector bundles $\mathcal{V}_{ij} = E^{i,j}$ and maps $C_{ij}: E^{i,j} \rightarrow E^{i,j+1}$ for $1 \le i \le k$ and $1 \le j \le w_i-2$, such that $(\lambda_{i0}\Psi_0(x) + \lambda_{i1}\Psi_1(x))C_{i1}(x) = 0$ and the $C_{ij}$ are injective.

Let $f$ be a morphism between objects $(E_1,E^{i,j}_1)$ and $(E_2,E^{i,j}_2)$ in $\mathscr{P}(D,w)$.  By Theorem \ref{6.0.1}, we can define the morphism $(f_1,f_2)$ between the objects corresponding to $E_1$ and $E_2$ in $\mathscr{R}$.  The restriction of $f$ to $E^{i,j}_1$ clearly defines morphisms $h_{ij}: E^{i,j}_1 \rightarrow E^{i,j}_2$ such that $(f_1,f_2, h_{ij})$ is a morphism in $\mathscr{S}(D,w)$ between the objects defined above.

Now, let $(\mathcal{V}, \mathcal{W}, \mathcal{V}_{ij}, \Psi_0,\Psi_1, C_{ij})$ be an object in $\mathscr{S}(D,w)$.  By Theorem \ref{6.0.1}, we can use this object to construct a vector bundle $E$, with dual generated by global sections.  Furthermore, we can see from the above construction that the morphisms $C_{ij}: V_{ij} \rightarrow V_{ij-1}$ define a filtration by vector bundles for each $1 \le i \le k$
$$E|_{T\times\{x_i\}} = \textrm{ker} (\lambda_{i0}\Psi_0 + \lambda_{i1}\Psi_1) \supset \textrm{Im } C_{i1} \supset \cdots \supset \textrm{Im }C_{i1}C_{i2} \cdots C_{iw_i-1} \supset 0.$$
Therefore, we obtain an object $(E, E^{i,j})$ of $\mathscr{P}(D,w)$, where $E^{i,j} = \textrm{Im }C_{i1}C_{i2} \cdots C_{ij}$ and $E^{i,w_i} = 0$.

Given a morphism $(f_1, f_2, h_{ij})$ in $\mathscr{S}(D,w)$, we can easily see that the morphism $f$ defined in Theorem \ref{6.0.1} from $(f_1, f_2)$ is actually a morphism in $\mathscr{P}(D,w)$.

Recall the definition of the functors $R$ and $V$ from the proof of Theorem \ref{6.0.1}.  We define the functor $S$ by:
\begin{align*}
&S(E) = (\mathcal{V}, \mathcal{W}, \mathcal{V}_{ij}, \Psi_0, \Psi_1, C_{ij})\\
&S(f) = (f_1,f_2, h_{ij}),
\end{align*}
where $(\mathcal{V}, \mathcal{W}, \Psi_0, \Psi_1)$ are as in the definition of $R$.  Similarly, we can define the functor $P$ by:
\begin{align*}
&P(\mathcal{V}, \mathcal{W}, \mathcal{V}_{ij}, \Psi_0,\Psi_1, C_{ij}) = (E, E^{i,j})\\ 
&P(f_1,f_2, h_{ij}) = f,
\end{align*}
where $E$ and $f$ are as in the definition of $V$.  

Now, from the proof of Theorem \ref{6.0.1} and the construction of the functors $P,S$, we can easily see that the functors $S$ and $P$ are mutually inverse to each other.  Therefore, the categories $\mathscr{P}(D,w)$ and $\mathscr{S}(D,w)$ are equivalent.
\end{proof}

We can modify the definition of the functor $F'$ from the previous section in order to obtain a functor representable by certain Kronecker-preinjective squid representations.  
\begin{defn}
\label{6.2.1}
Let $E$ be as in Definition \ref{6.1.1}.  Define the functor $\tilde{F}(T)$, from the category of schemes over $K$ to the category of sets as $\tilde{F}(T) = \left\langle (E,E^{i,j},s,t,r_{ij})\right\rangle$, where 
\begin{itemize}
\item $E$ \textrm{ is a vector bundle on } $T\times \mathbb{P}^1$,
\item $p_*(E^*(N))$ \textrm{ and } $p_*(E^*(N-1))$ \textrm{ are trivial vector bundles},
\item $s: \mathcal{O}_T^{(N+1)\alpha_0 + \alpha_{\infty}} \simeq p_*(E^*(N))$,
\item $t: \mathcal{O}_T^{N\alpha_0 + \alpha_{\infty}} \simeq p_*(E^*(N-1))$,
\item $E|_{T \times \{x_i\}} \supset E^{i,1}\supset \dots \supset E^{i,w_i-1} \supset E^{i,w_i} = 0$ are filtrations by trivial vector subbundles of fixed ranks $\textrm{rk } E^{i,j} = \alpha_{ij}$,
\item $r_{ij}: \mathcal{O}_T^{\alpha_{ij}} \simeq E^{i,j}$.
\end{itemize}
Here, $E^*(N) = E^*\otimes \pi(\mathcal{O}(N))$.
\end{defn}

We have the following:
\begin{thm}
\label{6.2.2}
The functor $\tilde{F}$ is representable by the scheme $KS(D,w,\alpha)$.
\end{thm}
\begin{proof}
Fix a test scheme $T$ and let $x_i = (\lambda_{i0}: \lambda_{i1})$.  By Theorem \ref{6.0.2}, $\tilde{F}(T)$ defines a family of elements of $KS(D,w,\alpha)$ over $T$.  Therefore, we have a morphism $T \rightarrow KS(D,w,\alpha)$.  Conversely, given a morphism $T \rightarrow KS(D,w,\alpha)$, by Theorem \ref{6.0.2} we have an element of $\tilde{F}(T)$.

We can now define a pair of natural transformations:
\begin{align*}
& \eta_T:  \tilde{F}(T) \rightarrow \textrm{Hom}(T,KS(\alpha))\\
& \rho_T:  \textrm{Hom}(T,KS(\alpha)) \rightarrow \tilde{F}(T),
\end{align*}
between the functor $\tilde{F}$ and the functor of points $\underline{KS(D,w,\alpha)}$ corresponding to $KS(D,w,\alpha)$.  It follows from construction and Theorem \ref{6.0.1} that $\eta_T$ and $\rho_T$ are mutual inverse.
Therefore, the functors are isomorphic, and $KS(D,w,\alpha)$ represents $\tilde{F}$.
\end{proof}

\subsection{The very good property for trivial bundles}
\label{6.3}
In this section, let $K = \mathbb{C}$. Let us consider an example of the moduli space $Fl(B)$ described in the previous section.  That is, for a fixed weight type $(D,w)$, set $\alpha_{\infty} = 0$ and $\alpha = (\alpha_0, \alpha_{ij})$.  Consider the corresponding moduli space $Fl(B)$, parameterizing parabolic bundles on $\mathbb{P}^1$ of weight type $(D,w)$, dimension vector $\alpha$, and trivial underlying vector bundle.  It is easy to see that the moduli space simplifies to the product of partial flag varieties $Fl(\alpha)$, described in section 1.2.  

There is a diagonal action by $\textrm{PGL}(\alpha_0,\mathbb{C})$ on $Fl(\alpha)$, so we may ask whether the quotient stack $\textrm{PGL}(\alpha_0,\mathbb{C}) \backslash Fl(\alpha)$ is very good.  Consider the diagonal $\textrm{GL}(\alpha_0,\mathbb{C})$-action on $Fl(\alpha)$ corresponding to this action.  It is easy to see $\textrm{PGL}(\alpha_0,\mathbb{C}) \backslash Fl(\alpha)$ is very good if and only if $\textrm{GL}(\alpha_0,\mathbb{C}) \backslash Fl(\alpha)$ is almost very good.  However, stabilizers of points under the $\textrm{GL}(\alpha_0,\mathbb{C})$ action clearly correspond to automorphism groups of the parabolic bundles represented by those points.  Applying Theorem 1.2.1, we obtain Theorem 1.2.2.

Below, we offer an alternative way of proving Theorem 1.2.2 by relating $Fl(\alpha)$ to quiver representations.  Indeed, recall from section \ref{3.4} that $\textrm{Rep}(Q^{st}_{D,w}, \alpha)$ is the space of star-shaped quiver representations.  Let $RI(Q^{st}_{D,w}, \alpha) \subset \textrm{Rep}(Q^{st}_{D,w}, \alpha)$ consist of representations for which the maps associated to $c_{ij}$ are injective.  The group $G(\alpha)$ acts on both
$RI(Q^{st}_{D,w}, \alpha)$ and  $\textrm{Rep}(Q^{st}_{D,w}, \alpha)$.

\begin{lmm}
\label{6.3.1}
If $G(\alpha) \backslash \textrm{Rep}(Q^{st}_{D,w}, \alpha)$ is very good, then $G(\alpha) \backslash RI(Q^{st}_{D,w}, \alpha)$ is very good.
\end{lmm}

\begin{proof}
We have that $RI(Q^{st}_{D,w}, \alpha) \subset \textrm{Rep}(Q^{st}_{D,w}, \alpha)$ is open, and therefore it follows that $\dimn RI(Q^{st}_{D,w}, \alpha) = \dimn \textrm{Rep}(Q^{st}_{D,w}, \alpha)$.  Furthermore, for all $d$, we have $$\{x \in RI(Q^{st}_{D,w}, \alpha)| \textrm{dim }G(\alpha)_x = d\} \subset \{y \in \textrm{Rep}(Q^{st}_{D,w}, \alpha)| \textrm{dim }G(\alpha)_y = d\}.$$  The statement of the lemma follows.
\end{proof}

\begin{lmm}
\label{6.3.2}
If $G(\alpha) \backslash RI(Q^{st}_{D,w}, \alpha)$ is very good, then $\textrm{PGL}(\alpha_0, \mathbb{C}) \backslash Fl(\alpha)$ is very good.
\end{lmm}
\begin{proof}
Note that there is an action of the subgroup $H(\alpha) = \prod_{\alpha_{i,j}} \textrm{GL}(\alpha_{i,j}) \subset G(\alpha)$ on $RI(Q^{st}_{D,w}, \alpha)$ induced by the action of $G(\alpha)$.  Furthermore, there is a morphism 
\begin{align*} & \varphi: RI(Q^{st}_{D,w}, \alpha) \rightarrow Fl(\alpha) \\  
	& \varphi(c_{ij}) = (\mathbb{C}^{\alpha_0} \supseteq \textrm{Im}(c_{i1}) \supseteq \cdots \supseteq \textrm{Im}(c_{i1} \cdots c_{iw_i})),	
\end{align*}
such that $H(\alpha)$ acts freely and transitively on the fibers (simply by changing the basis).  This means that the fibers of the morphism $\varphi$ are isomorphic to $H(\alpha)$, so they have dimension $\textrm{dim } H(\alpha)$.  Furthermore, the space $Fl(\alpha)$ is obtained as a quotient of $RI(Q^{st}_{D,w}, \alpha)$ by the action of $H(\alpha)$.  That is, we can pick open sets $U_i$ such that $Fl(\alpha) = \bigcup_i U_i$ and a morphism $p: RI(Q^{st}_{D,w}, \alpha) \rightarrow \textrm{Fl}(\alpha)$, with $p^{-1}(U_i) \cong U_i \times H(\alpha)$ ($p$ projects onto the first component).  Indeed, taking $U_i$ to be products of the standard coordinate charts on flag varieties and taking $p$ to be $\varphi$, these conditions are satisfied.  Since $H(\alpha)$ is a normal subgroup of $G(\alpha)$, then we obtain for $x \in Fl(\alpha)$ that $(G(\alpha)/H(\alpha))_{\varphi(x)} = G(\alpha)_x/H(\alpha)$.  Therefore, 
$$RI(Q^{st}_{D,w}, \alpha)^m := \{ y \in SI_{D,w}(\alpha)|\textrm{dim } G(\alpha)_y = m \}$$ 
maps to 
$$Fl(\alpha)^m := \{x \in Fl(\alpha)| \textrm{dim PGL}(\alpha_0, \mathbb{C})_x = m \}$$ 
under the morphism $\varphi$.  It follows from this that $$ \textrm{dim } RI(Q^{st}_{D,w}, \alpha)^m = \textrm{dim } H(\alpha) + \textrm{dim } Fl(\alpha)^m,$$ which implies that 
$$\textrm{dim } RI(Q^{st}_{D,w}, \alpha)^m + \alpha_0^2 = \textrm{dim } G(\alpha) + \textrm{dim } Fl(\alpha)^m$$
 or 
$$\textrm{dim } RI(Q^{st}_{D,w}, \alpha)^m + \textrm{dim }Fl(\alpha) = \textrm{dim }RI(Q^{st}_{D,w}, \alpha) + \textrm{dim }Fl(\alpha)^m.$$  
Since we get $\textrm{codim } Fl(\alpha)^m = \textrm{codim } RI(Q^{st}_{D,w}, \alpha)^m$ and $RI(Q^{st}_{D,w}, \alpha)$ is very good, then we obtain that 
$$\textrm{codim }Fl(\alpha)^m = \textrm{codim } RI(Q^{st}_{D,w}, \alpha)^m > m - 1 \textrm{ for all } m > 1.$$  
Thus, $Fl(\alpha)$ is very good. 
\end{proof}

\begin{proof}[Proof of Theorem 1.2.2]
We can see that the theorem follows from Lemma \ref{6.3.1}, Lemma \ref{6.3.2}, and Theorem \ref{3.6.1}.
\end{proof}

%% file: application_to_the_deligne-simpson_problem.tex
\subsection{Outline}
\label{7.0}
In this section, we wish to relate the almost very good property for the moduli of parabolic bundles to the space of solutions to the Deligne-Simpson problem.  Let $C_1, \dots, C_k$ be semisimple conjugacy classes of $n$-dimensional vector space automorphisms and let $D = (x_1, \dots, x_k)$ be collection of points on $\mathbb{P}^1$.  We can interpret a solution to the Deligne-Simpson problem as a logarithmic connection $\nabla$ on a rank $n$ vector bundle over $\mathbb{P}^1$ with singularities in $D$, which satisfies
$$
\textrm{Res}_{x_i} \nabla \in C_i.
$$
Therefore, we can interpret the space of solutions to the Deligne-Simpson problem as the moduli stack $\textrm{Conn}_{D,w,\alpha,\zeta}(\mathbb{P}^1)$ of such connections.  We can provide a presentation for this stack in terms of the fiber of a moment map on the cotangent bundle to certain squid representations (see Section \ref{3.4}).  This allows us to apply Theorem \ref{2.4.2} (and subsequent remarks) in order to prove Theorem 1.4.1.  That is, if the moduli stack of parabolic bundles $\textrm{Bun}_{D,w,\alpha}(\mathbb{P}^1)$ is almost very good, then $\textrm{Conn}_{D,w,\alpha,\zeta}(\mathbb{P}^1)$ is a nonempty, irreducible, locally complete intersection of dimension $2p(\alpha) - 1$.

Now, let $C_1, \dots, C_k$ be semismiple conjugacy classes of $n \times n$ complex matrices.  We define $ADS(C_1, \dots, C_k) \subset C_1 \times \cdots \times C_k$ as the subvariety consisting of solutions to the additive Deligne-Simpson problem.  For semisimple conjugacy classes, $C_1 \times \cdots \times C_k$ is an affine bundle over the cotangent bundle to the product of flag varieties $Fl(\alpha)$.  We use the very good property for the quotient stack $\textrm{PGL}(\alpha_0, \mathbb{C}) \backslash Fl(\alpha)$ to show that $ADS(C_1, \dots, C_k)$ is a nonempty, irreducible, complete intersection of dimension $2\dimn Fl(\alpha) - \alpha_0^2+1$, which proves Theorem 1.4.3.

If we let $C_1, \dots, C_k$ be semisimple conjugacy classes of $n \times n$ invertible complex matrices instead, we can can consider $MDS(C_1, \dots, C_k) \subset C_1 \times \cdots \times C_k$, the subvariety of solutions to the multiplicative Deligne-Simpson problem.  The Riemann-Hilbert correspondence gives an analytic isomorphism between $MDS(C_1, \dots, C_k)$ and a moduli space of logarithmic connections on $\mathbb{P}^1$ (analogous to \cite{In2013} or \cite{IIS2006}).  This allows us to transfer the properties obtained for $\textrm{Conn}_{D,w,\zeta}(\mathbb{P}^1)$ in Theorem 1.4.1 to $MDS(C_1, \dots, C_k)$.  This means that $MDS(C_1, \dots, C_k)$ is a nonempty, irreducible, complete intersection of dimension $2p(\alpha) + \alpha_0^2-1$ if the moduli stack $\textrm{Bun}_{D,w,\alpha}(\mathbb{P}^1)$ is almost very good, which proves Theorem 1.4.5. 

\subsection{Logarithmic Connections and Squid Representations}
\label{7.1}
As before, let $X$ be a smooth connected complex projective curve.  Let $D \subset X$ be a divisor on $X$, and let $j: X-D \rightarrow X$ be the inclusion.  Let $\Omega^1_X(\textrm{log } D)$ be the subsheaf of $j_*\Omega^1_{X-D}$ with sections that have poles of order at most $1$ along $D$.  We call this the sheaf of \textit{logarithmic 1-forms}.

\begin{rmk}
\label{7.1.1}
Note that the definition of a logarithmic differential form $\omega$ for varieties of higher dimension requires that both $\omega$ and $d\omega$ have poles of order at most $1$ along $D$.  However, since there are no higher order differential forms on a curve, the above definition is sufficient.
\end{rmk}

We can define the following:
\begin{defn}
\label{7.1.2}
Let $E$ be a vector bundle on $X$.  A logarithmic connection 
$$
\nabla: E \rightarrow E\otimes \Omega^1_X(\textrm{log } D)
$$
is a $\mathbb{C}$-linear morphism of sheaves that satisfies the Leibnitz rule
$$
\nabla (fs) = s\otimes df + f\nabla(s),
$$
where $f$ is a section of $\mathcal{O}_X$ and $s$ is a section of $E$.  Note that $\nabla$ has residues 
$$
\textrm{Res}_{x_i} \nabla \in \textrm{End}(E_{x_i}),
$$
for $x_i \in D$.
\end{defn}

From now on, let $X = \mathbb{P}^1$, let $D = (x_1, \dots, x_k)$ be a collection of points of $\mathbb{P}^1$, and let $w = (w_1, \dots, w_k)$ be a collection of positive integers.  

For a parabolic bundle $\mathbf{E}$ of weight type $(D,w)$ over $X$ we say that a logarithmic connection $\nabla: E \rightarrow E \otimes \Omega^1_X(\textrm{log } D)$ is a \textit{$\zeta$-parabolic connection} on $\mathbf{E}$ if
$$
(\textrm{Res}_{x_i} \nabla - \zeta_{ij}\cdot \textrm{Id})(E_{ij-1}) \subset E_{ij},
$$
where $E_{ij}$ are the subspaces of the flag in the fiber $E_{x_i} = E_{i0}$.

Recall from Section \ref{6.2} that $KS(D,w,\alpha)$ parametrizes parabolic bundles over $\mathbb{P}^1$ together with some rigidity conditions.  Let $\mu_{G(\alpha)}^{-1}(\theta^N)$ be the fiber of the moment map described in Section \ref{3.5} over 
$$ 
\theta^N = (N+1 + \sum_{1 \le i \le k} \zeta_{i1}, - N -\sum_{1 \le i \le k} \zeta_{i1} , \zeta_{i1} - \zeta_{i2}, \dots, \zeta_{iw_i-1}) \in \textrm{Mat}(\alpha^N)_0.
$$
Note this is well-defined for $\zeta$ coming from a parabolic connection with vector bundle of degree $-\alpha_{\infty}$, since we can compute:
\begin{align*}
\textrm{tr}(\theta^N) & = (N+1 + \sum_{1 \le i \le k} \zeta_{i1})(\alpha_{\infty} + N\alpha_0) + (- N - \sum_{1 \le i \le k} \zeta_{i1})(\alpha_{\infty} + (N+1)\alpha_0) \\ 
& + \sum_{1 \le i \le k} \sum_{1 \le j \le w_k-1} \alpha_{ij}(\zeta_{ij} - \zeta_{ij+1}) = \alpha_{\infty} - \alpha_0 \sum_{1 \le i \le k} \zeta_{i1} \\ & + \sum_{1 \le i \le k} \sum_{1 \le j \le w_k-1} \alpha_{ij}(\zeta_{ij} - \zeta_{ij+1})
 = \alpha_{\infty} - \sum_{i=1}^k \sum_{j=1}^{w_i} \zeta_{ij}(\alpha_{ij-1} - \alpha_{ij}) = 0,
\end{align*}  
by Remark \ref{7.1.7}.

In the following definition, we keep to the notation of Definition \ref{6.2.1}. The projections, $\pi$ and $p$ were defined at the beginning of Section \ref{6.1}.
\begin{defn}
\label{7.1.3}
Let us define a functor $L_{\zeta}(T)$, from the category of schemes over $\mathbb{C}$ to the category of sets as $L_{\zeta}(T) = \left\langle (E,E^{i,j},s,t,r_{ij}, \nabla)\right\rangle$, where 
\begin{itemize}
\item $E$ \textrm{ is a vector bundle on } $T\times \mathbb{P}^1$,
\item $p_*(E^*(N))$ \textrm{ and } $p_*(E^*(N-1))$ \textrm{ are trivial vector bundles},
\item $s: \mathcal{O}_T^{(N+1)\alpha_0 + \alpha_{\infty}} \simeq p_*(E^*(N))$,
\item $t: \mathcal{O}_T^{N\alpha_0 + \alpha_{\infty}} \simeq p_*(E^*(N-1))$,
\item $E|_{T \times \{x_i\}} \supset E^{i,1}\supset \dots \supset E^{i,w_i-1} \supset E^{i,w_i} = 0$ are filtrations by trivial vector subbundles of fixed ranks $\textrm{rk } E^{i,j} = \alpha_{ij}$,
\item $r_{ij}: \mathcal{O}_T^{\alpha_{ij}} \simeq E^{i,j}$,
\item $\nabla: E \rightarrow E \otimes \pi^* \Omega_{\mathbb{P}^1}^1(\textrm{log } D)$ is a $\mathbb{C}$-linear morphism of sheaves,
\item $\nabla (fs) = s\otimes df + f\nabla(s)$ for $s$ a section of $E$ and $f$ a section of $\pi^*(\mathcal{O}_{\mathbb{P}^1})\subset \mathcal{O}_{T \times \mathbb{P}^1}$,
\item $(\textrm{Res}_{x_i} \nabla - \zeta_{ij} \cdot \textrm{Id})(E^{i,j-1}) \subset E^{i,j}$, where $E^{i,0} = E|_{T \times \{x_i\}}$, and $\textrm{Res}_{x_i} \nabla := \nabla|_{T\times \{x_i\}}$.
\end{itemize}
Here, $E^*(N) = E^*\otimes \pi^*(\mathcal{O}(N))$.
\end{defn}

\begin{thm}
\label{7.1.4}
The functor $L_{\zeta}$ is represented by $\mu_{G(\alpha)}^{-1}(\theta^N)$.
\end{thm}
\begin{proof}
Let $L_{\zeta}(T) = (E,E^{i,j},s,t,r_{ij}, \nabla)$.
We know that by Theorem \ref{6.2.2} the functor $\tilde{F}$ is representable by the variety $KS(D,w,\alpha)$.  Let $\tilde{F}(T) = (E,E^{i,j},s,t,r_{ij})$.  Note that the natural pairing with the vector field $\frac{d}{dz}$ on $\mathbb{P}^1$ defines the $\mathbb{C}$-linear morphism
$$
\nabla_{\frac{d}{dz}}^*: E^* \rightarrow E^*(D),
$$
satisfying the Leibniz rule, where $E^*(D) = E^*\otimes \pi^*\mathcal{O}(D)$ (we regard $D$ as the divisor $x_1 + \cdots + x_k$).  Further note that this morphism uniquely determines $\nabla$.  We have that $\nabla_{\frac{d}{dz}}$ induces the morphism $E^*(N) \rightarrow E^*(N)(D)$.  In fact, it induces a morphism $B: E^*(N) \rightarrow E^*(N-1)(D)$.  From $B$ we obtain a $\mathbb{C}$-linear morphism 
$$
\tilde{B}: p_*(E^*(N)) \rightarrow p_*(E^*(N-1)(D)).
$$  
Similarly, from $\nabla_{z\frac{d}{dz}}: E \rightarrow E(D)$, we obtain 
$$
\tilde{B}': p_*(E^*(N)) \rightarrow p_*(E^*(N)(D)).
$$ 
Let $\Psi_0, \Psi_1: p_*(E^*(N))^* \rightarrow p_*(E^*(N-1))^*$ be the morphisms induced by the two inclusions $E^*(N-1) \hookrightarrow E^*(N)$ (corresponding to multiplication by the two global sections $1$ and $-z$ of $\pi^*(\mathcal{O}(1))$).  By the proof of Theorem \ref{6.0.2}, $\textrm{ker}(\lambda_{0i} \Psi_0 + \lambda_{1i} \Psi_1) \simeq E|_{T \times \{x_i\}}$.  Therefore, $\textrm{Res}_{x_i} \nabla$ defines the maps
\begin{align*}
& \tilde{C}_{i1}:=(\textrm{Res}_{x_i} \nabla - \zeta_{i1} \cdot \textrm{Id}): \textrm{ker}(\lambda_{0i} \Psi_0 + \lambda_{1i} \Psi_1) \rightarrow E^{i,1}\\
&\hat{C}_{ij}:= (\textrm{Res}_{x_i} \nabla - \zeta_{ij} \cdot \textrm{Id})|_{E^{i,j-1}}: E^{i,j-1} \rightarrow E^{i,j} \textrm{ for } 1 \le i \le k  \textrm{ and } 2 \le j \le w_i-1.
\end{align*}
We can extend $\tilde{C}_{i1}$ to $p_*(E^*(N))^*$.  Note that any two such extensions differ by a morphism that sends $\textrm{ker}(\lambda_{0i} \Psi_0 + \lambda_{1i} \Psi_1)$ to $0$.  Therefore, it has the form $A_i (\lambda_{0i} \Psi_0 + \lambda_{1i} \Psi_1)$ for some $A_i: p_*(E^*(N-1))^* \rightarrow E^{i,1}$.  Fix such an extension $\hat{C}_{i1}$ for each $1 \le i \le k$.  We can now define two morphisms of vector bundles: $\hat{B}_0, \hat{B}_1: p_*(E^*(N-1))^*\rightarrow p_*(E^*(N))^*$ in the following way: 
\begin{align*}
& \hat{B}_0^* = N\cdot \textrm{Id} - \tilde{B}' - \sum_{1 \le i \le k} \frac{x_i}{z-x_i}(\hat{C}_{i1}^*C_{i1}^*+\zeta_{i1}\cdot \textrm{Id})\\
& \hat{B}_1^* = -\tilde{B} - \sum_{1 \le i \le k} \frac{1}{z-x_i}(\hat{C}_{i1}^*C_{i1}^* + \zeta_{i1}\cdot \textrm{Id}),
\end{align*}
where $C_{ij}: E^{i,j} \rightarrow E^{i,j+1}$ are as defined in Theorem \ref{6.0.2}, and $z$ is the standard coordinate on $\mathbb{P}^1$.  Note that $\hat{B}_0, \hat{B}_1$ are well-defined by the construction of $\hat{C}_{ij}$.  

We can see that $\hat{B}_1$ (respectively $\hat{B}_0$) depends on the choice of extension in the construction of $\hat{C}_{i1}$.  However, any two such choices differ by $A_i (\lambda_{0i} \Psi_0 + \lambda_{1i} \Psi_1)$, so any two $\hat{B}_1$ (respectively $\hat{B}_0$) obtained in this way differ by 
$$
\sum_{1 \le i \le k} C_{i1}A_i (\lambda_{0i} \Psi_0 + \lambda_{1i} \Psi_1).
$$
By the Leibniz rule we have $[\tilde{B}, \Psi_0^*] = 0$ and $[\tilde{B}, \Psi_1^*] =  -\textrm{Id}$.  Also, note $\tilde{B}' = -\Psi_1^*\tilde{B}$ and $C_{i1}^*(x_i \Psi_0^* + \Psi_1^*) = 0$.  Therefore, we have:
\begin{align*}
&(\Psi_0\hat{B}_0 + \Psi_1\hat{B}_1)^* = (N \cdot \textrm{Id} - \tilde{B}')\Psi_0^* - \tilde{B}\Psi_1^* - \sum_{1 \le i \le k}(\frac{1}{z-x_i}\hat{C}_{i1}^*C_{i1}^*)(x_i\Psi_0^* + \Psi_1^*) \\ 
& - \sum_{1 \le i \le k} \frac{x_i}{z-x_i}\zeta_{i1}\cdot \textrm{Id} + \sum_{1 \le i \le k} \frac{z}{z-x_i}\zeta_{i1}\cdot \textrm{Id}\\
& = (N \cdot \textrm{Id} - \tilde{B}')\Psi_0^* - \tilde{B}\Psi_1^* + \sum_{1 \le i \le k} \zeta_{i1}\cdot \textrm{Id}\\ & =  (N+1 + \sum_{1 \le i \le k} \zeta_{i1}\cdot \textrm{Id}), 
\end{align*}
and
\begin{align*}
&\sum_{1 \le i \le k} \hat{C}_{i1}^* C_{i1}^*  - (\hat{B}_0\Psi_0 + \hat{B}_1\Psi_1)^* = \sum_{1 \le i \le k} \hat{C}_{i1}^* C_{i1}^* - \Psi_0^*(N\cdot \textrm{Id} - \tilde{B}'\\ &- \sum_{1 \le i \le k} \frac{x_i}{z-x_i}(\hat{C}_{i1}^*C_{i1}^*+\zeta_{i1}\cdot \textrm{Id})) + \Psi_1^*(\tilde{B} + \sum_{1 \le i \le k} \frac{1}{z-x_i}(\hat{C}_{i1}^*C_{i1}^* + \zeta_{i1}\cdot \textrm{Id}))\\
& =  \sum_{1 \le i \le k} \hat{C}_{i1}^* C_{i1}^* - \sum_{1 \le i \le k} (\hat{C}_{i1}^*C_{i1}^*+\zeta_{i1}\cdot \textrm{Id})- \Psi_0^*(N \cdot \textrm{Id} - \tilde{B}') + \Psi_1^*\tilde{B} \\ & = (- N -\sum_{1 \le i \le k} \zeta_{i1}) \cdot \textrm{Id}. 
\end{align*}
Furthermore, we have: 
\begin{align*}
C_{ij+1}\hat{C}_{ij+1} - \hat{C}_{ij}C_{ij}  = (\zeta_{ij} - \zeta_{ij+1}) \cdot \textrm{Id} \textrm{, where } 1 \le i \le k \textrm{ and } 1 \le j \le w_{i}-1.
\end{align*}
Since $\hat{B}_0, \hat{B}_1, \hat{C}_{ij}$ vary algebraically with the points of $T$, then the family $L_{\zeta}(T)$ defines a morphism $f: T \rightarrow \mu_{G(\alpha)}^{-1}(\theta^N)$ by construction.

Conversely, given a morphism $f: T \rightarrow \mu_{G(\alpha)}^{-1}(\theta^N)$, we get the corresponding morphism $T \rightarrow KS(D,w,\alpha)$.  Therefore, from the proof of Theorem \ref{6.0.2} (see Section \ref{6.2}) we get a collection $\tilde{F}(T) = (E,E^{i,j},s,t,r_{ij})$.  Moreover, $f$ defines the family $(\mathcal{V}, \mathcal{W}, \mathcal{V}_{ij}, \Psi_0,\Psi_1, C_{ij})$ of elements of $KS(D,w,\alpha)$, as well as families of morphisms $\hat{B}_0, \hat{B}_1: \mathcal{W} \rightarrow \mathcal{V}$ and $\hat{C}_{ij}: \mathcal{V}_{ij} \rightarrow \mathcal{V}_{ij+1}$.  Note that by the proof of Theorem \ref{6.0.1} (see Section \ref{6.1}) we have that $\mathcal{V} \simeq p_*(E^*(N))^*$ and $\mathcal{W} \simeq p_*(E^*(N-1))^*$.  From the construction of $\hat{B}_1$ above, we obtain:
$$
\tilde{B}: p_*(E^*(N)) \rightarrow p_*(E^*(N-1)(D)) \hookrightarrow p_*(E^*(N)(D)).
$$ 
Since $E^*(N)$ is generated by global sections, we can use the Leibniz rule to extend $\tilde{B}$ to a $\mathbb{C}$-linear morphism of vector bundles $B: E^*(N) \rightarrow E^*(N-1)(D)$ that satisfies
$$
B(fs) = s\otimes \frac{df}{dz} + f\nabla(s),
$$
for $s$ a section of $E^*(N)$ and $f$ a section of $\pi^*(\mathcal{O}_{\mathbb{P}^1})$.  We can further obtain a $\mathbb{C}$-linear morphism $\nabla_{\frac{d}{dz}}: E \rightarrow E(D)$ that satisfies the Leibniz rule.  This is the same as defining the $\mathbb{C}$-linear morphism
$$\nabla: E \rightarrow E \otimes \pi^* \Omega_{\mathbb{P}^1}^1(\textrm{log } D),$$
which satisfies the Leibniz rule.

Note that we have
$$\hat{C}_{i1}|_{\textrm{ker}(\lambda_{0i} \Psi_0 + \lambda_{1i} \Psi_1)} =  \nabla|_{T\times \{x_i\}} - \zeta_{i1} \cdot \textrm{Id}.$$ 
By Theorem \ref{6.0.2} we have $\mathcal{V}_{ij} = E^{i,j}$.  Therefore, $C_{ij+1}\hat{C}_{ij+1} - \hat{C}_{ij}C_{ij}  = (\zeta_{ij} - \zeta_{ij+1}) \cdot \textrm{Id}$ implies that
$$
(\textrm{Res}_{x_i} \nabla - \zeta_{ij} \cdot \textrm{Id})(E^{i,j-1}) \subset E^{i,j}.
$$
Thus $f: T \rightarrow \mu_{G(\alpha)}^{-1}(\theta^N)$ defines the family $L_{\zeta}(T) = (E,E^{i,j},s,t,r_{ij}, \nabla)$.

The above constructions define a pair of natural transformations:
\begin{align*}
& \eta_T:  L_{\zeta}(T) \rightarrow \textrm{Hom}(T, \mu_{G(\alpha)}^{-1}(\theta^N))\\
& \rho_T:  \textrm{Hom}(T, \mu_{G(\alpha)}^{-1}(\theta^N)) \rightarrow L_{\zeta}(T),
\end{align*}
between the functor $L_{\zeta}$ and the functor of points $\underline{\mu_{G(\alpha)}^{-1}(\theta^N)}$ corresponding to $\mu_{G(\alpha)}^{-1}(\theta^N)$.  It follows by construction and Theorem \ref{6.0.2} that $\eta_T$ and $\rho_T$ are mutual inverse.  Therefore, the functors are isomorphic, and $\mu_{G(\alpha)}^{-1}(\theta^N)$ represents $L_{\zeta}$.

\end{proof}

\begin{rmk}
\label{7.1.5}
We can follow the proof of Theorem \ref{7.1.4} in order to obtain that $\mu_{G(\alpha)}^{-1}(0)$ represents the functor 
$H(T) = \left\langle (E,E^{i,j},s,t,r_{ij}, \Phi)\right\rangle,$
where $E,E^{i,j},s,t,r_{ij}$ are as in Definition \ref{7.1.3} and $\Phi$ is a section of $\mathscr{E}nd(\pi^* \Omega^1_X(\textrm{log } D))$.  That is, $\mu_{G(\alpha)}^{-1}(0)$ is a moduli space parameterizing parabolic Higgs bundles over $\mathbb{P}^1$ together with rigidity.

It is easy to see that $\mu_{G(\alpha)}^{-1}(0)$ acts on $\mu_{G(\alpha)}^{-1}(\theta^N)$ by translation.  In fact, $\mu_{G(\alpha)}^{-1}(\theta^N)$ is a $\mu_{G(\alpha)}^{-1}(0)$-torsor.  This is natural, as parabolic Higgs bundles constitute the cotangent stack to the moduli stack of parabolic bundles, and $\zeta$-parabolic connections constitute the twisted cotangent stack to the moduli stack of parabolic bundles.
\end{rmk}

Assuming the conventions from Section \ref{4.2}, we have the following definition:
\begin{defn}
\label{7.1.6}
 The stack of $\zeta$-parabolic connections on parabolic bundles of weight type $(D,w)$ and of dimension type $\alpha$ over $X$ is a functor that associates to a test scheme $T$ the groupoid $\textrm{Conn}_{D,w,\alpha,\zeta}(T) = \left\langle (E, E^{i,j}, \nabla)_{1\le i \le k} \right\rangle$, where 
\begin{itemize}
\item $E$ is a vector bundle on $T\times X$,
\item $E|_{T \times \{x_i\}} \supset E^{i,1}\supset \dots \supset E^{i,w_i-1} \supset E^{i,w_i} = 0$
is a filtration by vector bundles,
\item $\textrm{rk}(E) = \alpha_0$ and $\textrm{rk}(E^{i,j}) = \alpha_{ij}$,
\item $\nabla: E \rightarrow E \otimes \pi^* \Omega_{\mathbb{P}^1}^1(\textrm{log } D)$ is a $\mathbb{C}$-linear morphism of sheaves,
\item $\nabla (fs) = s\otimes df + f\nabla(s)$ for $s$ a section of $E$ and $f$ a section of $\pi^*(\mathcal{O}_{\mathbb{P}^1})\subset \mathcal{O}_{T \times \mathbb{P}^1}$,
\item $(\textrm{Res}_{x_i} \nabla - \zeta_{ij} \cdot \textrm{Id})(E^{i,j-1}) \subset E^{i,j}$, where $E^{i,0} = E|_{T \times \{x_i\}}$, and $\textrm{Res}_{x_i} \nabla := \nabla|_{T\times \{x_i\}}$.
\end{itemize}
\end{defn}

\begin{rmk}
\label{7.1.7}
Note that if a $\zeta$-parabolic connection $\nabla$ exists on a parabolic bundle $\mathbf{E}$ of weight type $(D,w)$ and dimension vector $\alpha$ over $X$, then
$$
\sum_{i=1}^k \textrm{tr} (\textrm{Res}_{x_i} \nabla) = \sum_{i=1}^k \sum_{j=1}^{w_i} \zeta_{ij}(\alpha_{ij-1} - \alpha_{ij})  = - \textrm{deg } E.
$$
Therefore, fixing $\zeta$ automatically fixes $d = \textrm{deg } E$.
\end{rmk}

Set $d = - \alpha_{\infty}$, $\alpha^N = (\alpha_{\infty} + N, \alpha_{\infty} + \alpha_0 + N, \alpha_{ij})$, and $\theta^N$ as in Theorem \ref{7.1.4}.

By Theorem \ref{7.1.4}, we have that $U = \coprod_{N \in \mathbb{Z}_{\ge 0}} \mu_{G(\alpha)}^{-1}(\theta^N)$ is a presentation for the algebraic stack $\textrm{Conn}_{D,w,\alpha,\zeta}(X)$.  In fact, there exists an $N \in \mathbb{Z}_{\ge 0}$ such that $\mu_{G(\alpha)}^{-1}(\theta^N)$ is a presentation for $\textrm{Conn}_{D,w,\alpha,\zeta}(X)$.  Indeed, if a $\zeta$-parabolic connection exists on parabolic bundle $\mathbf{E}$, then the width (the difference between the maximal and minimal line bundle degrees in the Grothedieck Theorem decomposition of $E$) of $E$ is bounded (this follows, for example, from Theorem 7.1 in \cite{CB2004} and Lemma 1 in \cite{CB2010}).  Therefore, for a fixed $\zeta$, there is a single $N$ such that $E^*(N)$ is generated by global sections.  This implies the statement we need.

\begin{proof}[Proof of Theorem 1.4.1]
Let $\alpha_{\infty} = -d = \sum_{i=1}^k \sum_{j=1}^{w_i} \zeta_{ij}(\alpha_{ij-1} - \alpha_{ij})$.  Note that the stack $\textrm{Bun}_{D,w,\alpha}^d(X)$ admits the presentation $U = \coprod_{N \in \mathbb{Z}_{\ge 0}} KS(D,w,\alpha^N)$, where $\alpha^N = (\alpha_{\infty} + N, \alpha_{\infty} + \alpha_0 + N, \alpha_{ij})$.  Since $KS(D,w,\alpha^N)$ is irreducible for each $N$ and the fibers are products of general linear groups, then $\textrm{Bun}_{D,w,\alpha}^d(X)$ is irreducible.  It follows that the irreducible components of $\textrm{Bun}_{D,w,\alpha}(X)$ are the $\textrm{Bun}_{D,w,\alpha}^d(X)$.

Let $\theta^N$ be as in Theorem \ref{7.1.4}.  Fix $N \ge 0$ such that $\mu_{G(\alpha)}^{-1}(\theta^N)$ is a presentation for $\textrm{Conn}_{D,w,\alpha,\zeta}(X)$.  If $\textrm{Bun}_{D,w,\alpha}(X)$ is almost very good, then $\textrm{Bun}_{D,w,\alpha}^d(X)$ is almost very good for each $d$.  Consequently, we have that the quotient stack $G(\alpha^N) \backslash KS(D,w,\alpha^N)$ is very good.  

By Corollary \ref{2.4.3}, we have that $\mu_{G(\alpha)}^{-1}(\theta^N)$ is nonempty, irreducible, complete intersection of dimension 
$$\dimn  2(G(\alpha) + p(\alpha)) - \dimn G(\alpha) = 2p(\alpha) + \dimn G(\alpha).$$  
It follows that $\textrm{Conn}_{D,w,\alpha,\zeta}(X)$ is a nonempty, irreducible, locally complete intersection of dimension $2p(\alpha) - 1$.  
\end{proof}

We also get:
\begin{proof}[Proof of Corollary 1.4.2]
This instantly follows from Theorems 1.2.1 and 1.4.1.
\end{proof}

\begin{rmk}
\label{7.1.8}
Let $C_1, \dots, C_k$ be semisimple conjugacy classes of endomorphisms of $E_{x_1}, \dots, E_{x_k}$, respectively.  We may interpret $\textrm{Conn}_{D,w,\alpha,\zeta}(X)$ as the moduli stack of solutions to the Deligne-Simpson problem.  

Indeed, a solution of the Deligne-Simpson problem is a connection $\nabla$ on a vector bundle $E$ over $\mathbb{P}^1$, with regular singularities in $D$ such that $\textrm{Res}_{x_i} \nabla \in C_i$ for all $x_i \in D$.  This determines a dimension vector $\alpha = (\alpha_0, \alpha_{ij})$, where $\alpha_0 = \textrm{rk }E$ and $\alpha_{ij} = \textrm{rk }(\textrm{Res}_{x_i} \nabla - \zeta_{ij} \cdot \textrm{Id})$ is the dimension of the direct sum of the first $w_i - j$ eigenspaces of $C_i$ ordered from least to greatest, and a vector of eigenvalues $\zeta$ (accounting for multiplicity). Therefore, $\nabla$ is a $\zeta$-parabolic connection on a parabolic bundle with underlying vector bundle $E$, weight type $(D,w)$, and dimension type $\alpha$.

Conversely, any parabolic $\zeta$-connection in $\textrm{Conn}_{D,w,\alpha,\zeta}(X)$ has residues lying in the conjugacy classes $C_i$ with eigenvalues in $\zeta$ (accounting for multiplicity), and eigenspaces ordered from least to greatest of dimensions $\alpha_{ij}$.
\end{rmk}

\begin{rmk}
\label{7.1.9}
Note that the above remark is a special case of Theorem 2.1 in \cite{CB2004}.  In general, this theorem implies that a regular singular connection $\nabla$ on $\mathbb{P}^1$ is a $\zeta$-parabolic connection if and only if its residues lie in the closures of conjugacy classes defined by $\zeta$.

Therefore, if we relax the conditions in the statement of the Deligne-Simpson problem to allow solutions to lie in conjugacy class closures (rather than the conjugacy classes themselves), we may interpret $\textrm{Conn}_{D,w,\alpha,\zeta}(X)$ as the moduli stack of solutions. 
\end{rmk}

\subsection{The very good property and the additive Deligne-Simpson problem}
\label{7.2}
Recall from the Introduction that the additive Deligne-Simpson problem asks whether there exist
matrices $A_1, \dots, A_k $ in prescribed conjugacy classes $C_1, \dots, C_k$ such that $A_1 + \cdots + A_k = 0$. 

\begin{defn}
\label{7.2.1}
Let $C_1, \dots, C_k$ be conjugacy classes of matrices in $\mathscr{gl}_n(\mathbb{C})$.  We denote by 
$$ADS(C_1, \dots, C_k) := \{(A_1, \dots, A_k) \in C_1 \times \cdots \times C_k| A_1 + \cdots + A_k = 0 \}$$
the algebraic subvariety of solutions of the additive Deligne-Simpson problem in $C_1 \times \cdots \times C_k$ .
\end{defn}

We are now ready to prove that the very good property for the quotient stack $\textrm{PGL}(\alpha_0, \mathbb{C}) \backslash Fl(\alpha)$ implies that $ADS(C_1, \cdots, C_k)$ is nonempty, irreducible, and a complete intersection of dimension $2\dimn Fl(\alpha) - \alpha_0^2+1$, as long as $\textrm{tr}(A_1 + \cdots + A_k) = 0$ for $A_i \in C_i$.  
\begin{proof}[Proof of Theorem 1.4.3]

By Lemma \ref{6.3.2}, $G(\alpha) \backslash RI(Q_{D,w}^{st}, \alpha)$ is very good.  Therefore, by Corollary \ref{2.4.3}, we have that $\mu_{G(\alpha)}^{-1}(\theta^N)$ is a nonempty, irreducible, complete intersection of dimension $2\dimn RI(Q_{D,w}^{st}, \alpha) - \dimn G(\alpha)$.

From the proof of Lemma \ref{6.3.2} we have that $Fl(\alpha)$ is the locally trivial quotient of $RI(Q_{D,w}^{st}, \alpha)$ by the group $H(\alpha)$.  Moreover, by Theorem \ref{7.1.4} we see that the locally trivial quotient $\mu_{G(\alpha)}^{-1}(\theta^N)/ H(\alpha)$ is isomorphic to $ADS(C_1, \dots, C_k)$.  It follows that $ADS(C_1, \dots, C_k)$ is a nonempty, irreducible, complete intersection of dimension
\begin{align*}
&2\dimn RI(Q_{D,w}^{st}, \alpha) - \dimn G(\alpha) - H(\alpha) = 2(G(\alpha) + p(\alpha)) - \dimn G(\alpha) - H(\alpha)\\
&= 2 p(\alpha) + \alpha_0^2-1 = 2\dimn Fl(\alpha) - \alpha_0^2+1.
\end{align*}

\end{proof}

\begin{rmk}
\label{7.2.2}
Note that the dimension formula $\dimn ADS(C_1, \dots, C_k) = 2p(\alpha) + \alpha_0^2 - 1$ is similar to the formula given in Theorem 1.2 of \cite{CB2001}.
\end{rmk}

\begin{rmk}
\label{7.2.3}
By Remark \ref{2.4.4}, to prove that $ADS(C_1, \dots, C_k)$ is a nonempty, equidimensional complete intersection (of dimension $2\dimn Fl(\alpha) - \alpha_0^2+1$), it suffices to show that $\textrm{PGL}(\alpha_0) \backslash Fl(\alpha)$ is good, rather than very good.  The very good property is only used in order to prove that there is only one irreducible component.
\end{rmk}

If $\delta(\alpha) > 0$ and we assume that the eigenvalues of $C_1, \dots, C_k$ are ordered as in Section \ref{7.1}, then we obtain that $\alpha$ is in the fundamental region.  
\begin{proof}[Proof of Corollary 1.4.4]
This follows from Theorem 1.2.2 and Theorem 1.4.3. 
\end{proof}

\subsection{The very good property and the multiplicative Deligne-Simpson problem}
\label{7.3}
The multiplicative Deligne-Simpson asks whether there exist matrices $A_1, \dots, A_k $ in prescribed conjugacy classes $C_1, \dots, C_k$ such that $A_1 \cdots  A_k = \textrm{Id}$. 

\begin{defn}
\label{7.3.1}
Let $C_1, \dots, C_k$ be conjugacy classes of matrices in $\textrm{GL}(n, \mathbb{C})$.  We denote by 
$$MDS(C_1, \dots, C_k) := \{(A_1, \dots, A_k) \in C_1 \times \cdots \times C_k| A_1 \cdot A_2 \cdots  A_k = \textrm{Id} \}$$
the algebraic subvariety of solutions of the multiplicative Deligne-Simpson problem in $C_1 \times \cdots \times C_k$ .
\end{defn}

Instead of using the moduli space of $\zeta$-parabolic connections defined in Section \ref{7.1}, we will introduce a different moduli space, representing the following functor: 
\begin{defn}
\label{7.3.2}
Let $E$ be as in Definition \ref{7.1.3}, and let $y \in \mathbb{P}^1$.  Let us define a functor $\tilde{L}_{\zeta}(T)$, from the category of schemes over $\mathbb{C}$ to the category of sets as $\tilde{L}_{\zeta}(T) = \left\langle (E,E^{i,j},r,\nabla)\right\rangle$, where 
\begin{itemize}
\item $E$ \textrm{ is a vector bundle on } $T\times \mathbb{P}^1$,
\item $E|_{T \times \{y\}}$ \textrm{ is a trivial vector bundles},
\item $r: E|_{T \times \{y\}} \simeq \mathcal{O}_T^{\alpha_0}$,
\item $E|_{T \times \{x_i\}} \supset E^{i,1}\supset \dots \supset E^{i,w_i-1} \supset E^{i,w_i} = 0$ are filtrations by vector subbundles of fixed ranks $\textrm{rk } E^{i,j} = \alpha_{ij}$,
\item $\nabla: E \rightarrow E \otimes \pi^* \Omega_{\mathbb{P}^1}^1(\textrm{log } D)$ is a $\mathbb{C}$-linear morphism of sheaves,
\item $\nabla (fs) = s\otimes df + f\nabla(s)$ for $s$ a section of $E$ and $f$ a section of $\pi^*(\mathcal{O}_{\mathbb{P}^1})\subset \mathcal{O}_{T \times \mathbb{P}^1}$,
\item $(\textrm{Res}_{x_i} \nabla - \zeta_{ij} \cdot \textrm{Id})(E^{i,j-1}) \subset E^{i,j}$, where $E^{i,0} = E|_{T \times \{x_i\}}$, and $\textrm{Res}_{x_i} \nabla := \nabla|_{T\times \{x_i\}}$.
\end{itemize}
\end{defn}
 
Similar to Theorem 6.13 in \cite{Si1995} and Section 4 in \cite{Si1994}, it follows that the functor $\tilde{L}_{\zeta}$ is representable by a quasiprojective scheme.  We will denote this scheme by $R_{DR}(D, w, y, \alpha, \zeta)$.
 
We need one more concept, in order for the Riemann-Hilbert correspondence to establish a well-defined analytic isomorphism between $R_{DR}(D, w, y, \alpha, \zeta)$ and the space $MDS(C_1, \dots, C_k)$.  A \textit{transversal} to $\mathbb{Z}$ in $\mathbb{C}$ is a subset $T \subset \mathbb{C}$ such that $t \mapsto \exp(-2\pi\sqrt{-1}t)$ bijectively maps $T$ to $\mathbb{C}^*$ (see e.g. \cite{CB2004}).  We will henceforth denote $T = (T_1, \dots, T_k)$ is a collection of transversals.

Assume that $C_1, \dots, C_k$ are semisimple.  Let $\tau = (\tau_{ij})$ be the vector of eigenvalues (counting multiplicity) for the conjugacy classes $C_1, \dots, C_k$. Fix a collection of transversals $T$, and let $\zeta$ be defined by $\tau_{ij} = \exp(-2\pi\sqrt{-1}\zeta_{ij})$ such that $\zeta_{ij} \in T_i$. The multiplicities of the eigenvalues $\tau$ define a dimension vector $\alpha$ as in Remark \ref{7.1.8}.  Fix some $D = (x_1, \dots, x_k)$ and $y \in \mathbb{P}^1$ such that $y \notin D$.

\begin{thm}
\label{7.3.3}
The Riemann-Hilbert correspondence establishes an isomorphism of analytic spaces between $R_{DR}(D, w, y, \alpha, \zeta)$ and $MDS(C_1, \dots, C_k)$. 
\end{thm}
\begin{proof}
Let $(\mathbf{E}, r, \nabla) \in R_{DR}(D, w, y, \alpha, \zeta)$ be a triple consisting of a parabolic bundle $\mathbf{E}$, a $\zeta$-parabolic connection on $\mathbf{E}$, and a trivialization $r$ of the fiber $E_y$.  We have the following map:
\begin{align*}
& RH: R_{DR}(D, w, y, \alpha, \zeta) \rightarrow MDS(C_1, \dots, C_k) \\
&(\mathbf{E}, r, \nabla) \mapsto (\rho_y(a_1), \dots,\rho_y(a_k)),
\end{align*}
where $\rho_y: \pi_1(\mathbb{P}^1 - D, y) \rightarrow E_y \simeq \mathbb{C}^{\alpha_0}$ is the monodromy representation defined by the pair $(\mathbf{E}, \nabla)$ under the Riemann-Hilbert correspondence, and $a_1, \dots, a_k$ are the loops at base point $y$ around the punctures $x_i$.  This map is well-defined.  

Indeed, $\pi_1(\mathbb{P}^1 - D, y)$ is the group freely generated by the loops $a_i$, satisfying the relation $a_1 \cdots a_k = 1$.  Therefore, for the corresponding monodromy operators satisfy $\rho_y(a_1) \cdots \rho_y(a_k) = \textrm{Id}$.  Furthermore, it is a well-known fact (see e.g. Lemma 6.2 in \cite{CB2004}) that $\rho_y(a_i)$ is conjugate to $\exp(-2\pi\sqrt{-1}\textrm{Res}_{x_i} \nabla)$ if $\nabla$ is a $\zeta$-parabolic connection with $\zeta$ as defined above.  Therefore, by construction, $\rho_y(a_i) \in C_i$.  Since $\sum_{ij} \zeta_{ij} = -\textrm{deg } E$ is an integer, then $\prod_{ij} \tau_{ij} = 1$.  If the pair $(\mathbf{E}, \nabla)$ is defined by complex analytic parameters, then the local system corresponding to this pair, and the monodromy operators $\rho_{a_i}$ depend analytically on these parameters.  It follows that $RH$ is analytic. 

The Riemann-Hilbert correspondence provides the map $RH$ with a well-defined inverse, sending the $k$-tuple of monodromy operators $(\rho_y(a_1), \dots,\rho_y(a_k))$ to the corresponding triple $(\mathbf{E}, r, \nabla)$.  As above, we can see that the inverse is complex analytic.  Therefore, $RH$ is an analytic isomorphism between $R_{DR}(D, w, y, \alpha, \zeta)$ and $MDS(C_1, \dots, C_k)$.
\end{proof}

\begin{proof}[Proof of Theorem 1.4.5]
There is a smooth, representable morphism 
$$R_{DR}(D, w, y, \alpha, \zeta) \rightarrow \textrm{Conn}_{D,w,\alpha,\zeta}(X),$$ 
defined by forgetting the rigidity condition on $R_{DR}(D, w, y, \alpha, \zeta)$.  It is therefore easy to see that $R_{DR}(D, w, y, \alpha, \zeta)$ is an irreducible, complete intersection of dimension $2p(\alpha) + \alpha_0^2 - 1$.  By Theorem \ref{7.3.3} there is an analytic isomorphism between $R_{DR}(D, w, y, \alpha, \zeta)$ and $MDS(C_1, \dots, C_k)$.  It follows that $MDS(C_1, \dots, C_k)$ is a complete intersection of dimension $2p(\alpha) + \alpha_0^2 - 1$.  Since the smooth locus of $R_{DR}(D, w, y, \alpha, \zeta)$ is irreducible, it is connected.  Therefore, the smooth locus of $MDS(C_1, \dots, C_k)$ is also connected.  Thus, $MDS(C_1, \dots, C_k)$ is irreducible.
\end{proof}

As before, if we assume an appropriate ordering on the eigenvalues of $C_1, \dots, C_k$, then $\alpha$ is automatically in the fundamental region. 
\begin{proof}[Proof of Corollary 1.4.6]
This follows immediately from Theorems 1.2.1 and 1.4.5.
\end{proof}

%% file: dspaper3.bbl
\def\cydot{\leavevmode\raise.4ex\hbox{.}}
\begin{thebibliography}{10}

\bibitem{Ar2010}
{\sc Arinkin, D.}
\newblock Rigid irregular connections on {$\mathbb{P}\sp 1$}.
\newblock {\em Compos. Math. 146}, 5 (2010), 1323--1338.

\bibitem{BD1991}
{\sc Beilinson, A., and Drinfeld, V.}
\newblock Quantization of {H}itchin's {I}ntegrable {S}ystem and {H}ecke
  {E}igensheaves.
\newblock \url{www.math.uchicago.edu/ mitya/langlands/hitchin/BD-hitchin.pdf},
  1991.

\bibitem{Bo2008}
{\sc Boalch, P.}
\newblock Irregular connections and {K}ac-{M}oody root systems.
\newblock arXiv:0806.1050.

\bibitem{CB1993}
{\sc Crawley-Boevey, W.}
\newblock Geometry of {R}epresentations of {A}lgebras.
\newblock \url{http://www.maths.leeds.ac.uk/~pmtwc/geomreps.pdf}, 1993.

\bibitem{CB2001}
{\sc Crawley-Boevey, W.}
\newblock Geometry of the moment map for representations of quivers.
\newblock {\em Compositio Math. 126}, 3 (2001), 257--293.

\bibitem{CB2003}
{\sc Crawley-Boevey, W.}
\newblock On matrices in prescribed conjugacy classes with no common invariant
  subspace and sum zero.
\newblock {\em Duke Math. J. 118}, 2 (2003), 339--352.

\bibitem{CB2004}
{\sc Crawley-Boevey, W.}
\newblock Indecomposable parabolic bundles and the existence of matrices in
  prescribed conjugacy class closures with product equal to the identity.
\newblock {\em Publ. Math. Inst. Hautes \'Etudes Sci.}, 100 (2004), 171--207.

\bibitem{CB2010}
{\sc Crawley-Boevey, W.}
\newblock Kac's theorem for weighted projective lines.
\newblock {\em J. Eur. Math. Soc. (JEMS) 12}, 6 (2010), 1331--1345.

\bibitem{CB2013}
{\sc Crawley-Boevey, W.}
\newblock personal communication, 2013.

\bibitem{CBS2006}
{\sc Crawley-Boevey, W., and Shaw, P.}
\newblock Multiplicative preprojective algebras, middle convolution and the
  {D}eligne-{S}impson problem.
\newblock {\em Adv. Math. 201}, 1 (2006), 180--208.

\bibitem{De1970}
{\sc Deligne, P.}
\newblock {\em \'{E}quations diff\'erentielles \`a points singuliers
  r\'eguliers}.
\newblock Lecture Notes in Mathematics, Vol. 163. Springer-Verlag, Berlin-New
  York, 1970.

\bibitem{F1993}
{\sc Faltings, G.}
\newblock Stable {$G$}-bundles and projective connections.
\newblock {\em J. Algebraic Geom. 2}, 3 (1993), 507--568.

\bibitem{GP2007}
{\sc Garc{\'{\i}}a-Prada, O., Gothen, P.~B., and Mu{\~n}oz, V.}
\newblock Betti numbers of the moduli space of rank 3 parabolic {H}iggs
  bundles.
\newblock {\em Mem. Amer. Math. Soc. 187}, 879 (2007), viii+80.

\bibitem{Gi2001}
{\sc Ginzburg, V.}
\newblock The global nilpotent variety is {L}agrangian.
\newblock {\em Duke Math. J. 109}, 3 (2001), 511--519.

\bibitem{Ha1977}
{\sc Hartshorne, R.}
\newblock {\em Algebraic geometry}.
\newblock Springer-Verlag, New York-Heidelberg, 1977.
\newblock Graduate Texts in Mathematics, No. 52.

\bibitem{Hi2013}
{\sc Hiroe, K.}
\newblock Linear differential equations on the {R}iemann sphere and
  representations of quivers.
\newblock arXiv:1307.7438.

\bibitem{In2013}
{\sc Inaba, M.-A.}
\newblock Moduli of parabolic connections on curves and the {R}iemann-{H}ilbert
  correspondence.
\newblock {\em J. Algebraic Geom. 22}, 3 (2013), 407--480.

\bibitem{IIS2006}
{\sc Inaba, M.-A., Iwasaki, K., and Saito, M.-H.}
\newblock Moduli of stable parabolic connections, {R}iemann-{H}ilbert
  correspondence and geometry of {P}ainlev\'e equation of type {VI}. {I}.
\newblock {\em Publ. Res. Inst. Math. Sci. 42}, 4 (2006), 987--1089.

\bibitem{Kac1980}
{\sc Kac, V.~G.}
\newblock Infinite root systems, representations of graphs and invariant
  theory.
\newblock {\em Invent. Math. 56}, 1 (1980), 57--92.

\bibitem{Kac1982}
{\sc Kac, V.~G.}
\newblock Infinite root systems, representations of graphs and invariant
  theory. {II}.
\newblock {\em J. Algebra 78}, 1 (1982), 141--162.

\bibitem{Kac1990}
{\sc Kac, V.~G.}
\newblock {\em Infinite-dimensional {L}ie algebras}, third~ed.
\newblock Cambridge University Press, Cambridge, 1990.

\bibitem{Katz1996}
{\sc Katz, N.~M.}
\newblock {\em Rigid local systems}, vol.~139 of {\em Annals of Mathematics
  Studies}.
\newblock Princeton University Press, Princeton, NJ, 1996.

\bibitem{King1994}
{\sc King, A.~D.}
\newblock Moduli of representations of finite-dimensional algebras.
\newblock {\em Quart. J. Math. Oxford Ser. (2) 45}, 180 (1994), 515--530.

\bibitem{Ko1999}
{\sc Kostov, V.~P.}
\newblock On the {D}eligne-{S}impson problem.
\newblock {\em C. R. Acad. Sci. Paris S\'er. I Math. 329}, 8 (1999), 657--662.

\bibitem{Ko2001}
{\sc Kostov, V.~P.}
\newblock The {D}eligne-{S}impson problem for zero index of rigidity.
\newblock In {\em Perspectives of complex analysis, differential geometry and
  mathematical physics ({S}t. {K}onstantin, 2000)}. World Sci. Publ., River
  Edge, NJ, 2001, pp.~1--35.

\bibitem{Ko2002}
{\sc Kostov, V.~P.}
\newblock On the {D}eligne-{S}impson problem.
\newblock {\em Tr. Mat. Inst. Steklova 238}, Monodromiya v Zadachakh Algebr.
  Geom. i Differ. Uravn. (2002), 158--195.

\bibitem{Ko2003}
{\sc Kostov, V.~P.}
\newblock On some aspects of the {D}eligne-{S}impson problem.
\newblock {\em J. Dynam. Control Systems 9}, 3 (2003), 393--436.

\bibitem{Ko22004}
{\sc Kostov, V.~P.}
\newblock The {D}eligne-{S}impson problem---a survey.
\newblock {\em J. Algebra 281}, 1 (2004), 83--108.

\bibitem{Ko12004}
{\sc Kostov, V.~P.}
\newblock On the {D}eligne-{S}impson problem and its weak version.
\newblock {\em Bull. Sci. Math. 128}, 2 (2004), 105--125.

\bibitem{Ko2005}
{\sc Kostov, V.~P.}
\newblock The connectedness of some varieties and the {D}eligne-{S}impson
  problem.
\newblock {\em J. Dyn. Control Syst. 11}, 1 (2005), 125--155.

\bibitem{La1987}
{\sc Laumon, G.}
\newblock Correspondance de {L}anglands g\'eom\'etrique pour les corps de
  fonctions.
\newblock {\em Duke Math. J. 54}, 2 (1987), 309--359.

\bibitem{La1988}
{\sc Laumon, G.}
\newblock Un analogue global du c\^one nilpotent.
\newblock {\em Duke Math. J. 57}, 2 (1988), 647--671.

\bibitem{LMB2000}
{\sc Laumon, G., and Moret-Bailly, L.}
\newblock {\em Champs alg\'ebriques}, vol.~39 of {\em Ergebnisse der Mathematik
  und ihrer Grenzgebiete. 3. Folge. A Series of Modern Surveys in Mathematics
  [Results in Mathematics and Related Areas. 3rd Series. A Series of Modern
  Surveys in Mathematics]}.
\newblock Springer-Verlag, Berlin, 2000.

\bibitem{MS1980}
{\sc Mehta, V.~B., and Seshadri, C.~S.}
\newblock Moduli of vector bundles on curves with parabolic structures.
\newblock {\em Math. Ann. 248}, 3 (1980), 205--239.

\bibitem{Se1977}
{\sc Seshadri, C.~S.}
\newblock Moduli of vector bundles on curves with parabolic structures.
\newblock {\em Bull. Amer. Math. Soc. 83}, 1 (1977), 124--126.

\bibitem{Si1991}
{\sc Simpson, C.~T.}
\newblock Products of matrices.
\newblock In {\em Differential geometry, global analysis, and topology
  ({H}alifax, {NS}, 1990)}, vol.~12 of {\em CMS Conf. Proc.} Amer. Math. Soc.,
  Providence, RI, 1991, pp.~157--185.

\bibitem{Si1994}
{\sc Simpson, C.~T.}
\newblock Moduli of representations of the fundamental group of a smooth
  projective variety. {I}.
\newblock {\em Inst. Hautes \'Etudes Sci. Publ. Math.}, 79 (1994), 47--129.

\bibitem{Si1995}
{\sc Simpson, C.~T.}
\newblock Moduli of representations of the fundamental group of a smooth
  projective variety. {II}.
\newblock {\em Inst. Hautes \'Etudes Sci. Publ. Math.}, 80 (1995), 5--79
  (1995).

\bibitem{Si2009}
{\sc Simpson, C.~T.}
\newblock Katz's middle convolution algorithm.
\newblock {\em Pure Appl. Math. Q. 5}, 2, Special Issue: In honor of Friedrich
  Hirzebruch. Part 1 (2009), 781--852.

\bibitem{Su2007}
{\sc Sugiyama, K.-i.}
\newblock A quantization of the {H}itchin hamiltonian system and the
  {B}eilinson-{D}rinfeld isomorphism.
\newblock arXiv:0708.2957.

\bibitem{Sta2014}
{\sc {The Stacks Project Authors}}.
\newblock Stacks {P}roject.
\newblock \url{http://stacks.math.columbia.edu/}, 2014.

\bibitem{We1994}
{\sc Weibel, C.~A.}
\newblock {\em An introduction to homological algebra}, vol.~38 of {\em
  Cambridge Studies in Advanced Mathematics}.
\newblock Cambridge University Press, Cambridge, 1994.

\bibitem{Y2008}
{\sc Yamakawa, D.}
\newblock Geometry of multiplicative preprojective algebra.
\newblock {\em Int. Math. Res. Pap. IMRP\/} (2008), Art. ID rpn008, 77pp.

\end{thebibliography}
